\documentclass[AMS,STIX1COL]{WileyNJD-v2}
\newtheorem{thm}{Theorem}[section]

\newtheorem{lem}[thm]{Lemma}

\newtheorem{rem}[thm]{Remark}

\usepackage{blebbing}
\usepackage{tikz}
\usetikzlibrary{calc,arrows}

\newcommand{\equalscolon}{=:}
\newcommand{\colonequals}{:=}

\newlength{\squareLen}
\newlength{\cortexRad}
\newlength{\membraneRad}
\newlength{\linkerFootRad}
\newlength{\legendMargin}
\newlength{\legendHeight}
\pgfmathsetlength{\membraneRad}{\cortexRad+.25cm}
\pgfmathsetlength{\linkerFootRad}{\cortexRad+1mm}

\newcommand{\phaseFieldPT}{\varphi}
\renewcommand{\phaseFieldProto}{\phaseFieldPT}
\newcommand{\surfaceProto}{\Phi}
\newcommand{\scalarFunProto}{q}
\newcommand{\vecFunProto}{\vec{q}}
\newcommand{\matFunProto}{Q}

\newcommand{\surfEnergy}[2][]{\ONEARGOP{#1{\mathcal{S}}}{#2}}
\newcommand{\surfEnergyOne}[1]{\surfEnergy[\schemeSubscr{\interfOne{}}]{#1}}
\newcommand{\surfEnergyTwo}[1]{\surfEnergy[\schemeSubscr{\interfTwo{}}]{#1}}
\newcommand{\pfSurfEnergyProto}[1]{
    \def\mod##1{\schemeSubscr{\surfaceProto}{\schemeSupscr{\phaseFieldParam}{##1}}}
    \surfEnergy[\mod]{#1}
}
\newcommand{\pfSurfEnergyOne}[1]{
    \def\mod##1{\schemeSubscr{\interfOne{}}{\schemeSupscr{\phaseFieldParam}{##1}}}
    \surfEnergy[\mod]{#1}
}
\newcommand{\pfSurfEnergyTwo}[1]{
    \def\mod##1{\schemeSubscr{\interfTwo{}}{\schemeSupscr{\phaseFieldParam}{##1}}}
    \surfEnergy[\mod]{#1}
}
\newcommand{\totalEnergyFunct}[2][]{\ONEARGOP{#1{U}}{#2}}

\renewcommand{\ginzburgLandauEnergyFunct}[3][]{%
  	\def\tmp{#3}
    \ifx\tmp\empty
    	\ifx\indicatePhaseField\undefined
	  		\ONEARGOP{#1{\mathcal{G}}}{#2}    	 
        \else
	  		\ONEARGOP{#1{\mathcal{G}_{\phaseFieldParam}}}{#2}    	 
        \fi
    \else
    	\ifx\indicatePhaseField\undefined
	  		\ONEARGOP{#1{\mathcal{G}_{#3}}}{#2}
        \else
	  		\ONEARGOP{#1{\mathcal{G}_{\phaseFieldParam,#3}}}{#2}
        \fi
    \fi
}

\newcommand{\refBndCond}{%
    \eqref{equ:modelling:phase field:resulting PDEs:boundary conditions:%
    velocity},  %
    \eqref{equ:modelling:phase field:resulting PDEs:boundary conditions:%
    phase field},  %
    \eqref{equ:modelling:phase field:resulting PDEs:boundary conditions:%
    fluxes},  %
    \eqref{equ:modelling:phase field:resulting PDEs:boundary conditions:%
    species}%
}

\renewcommand{\membrPhaseFieldSym}{\phi}
\renewcommand{\cortexPhaseFieldSym}{\psi}
\newcommand{\couplingIntCortex}{C_{\interfTwo{}}}
\newcommand{\couplingIntMembr}{C_{\interfOne{}}}
\newcommand{\pfCouplingIntCortex}[1][]{#1{C}_{#1{\cortexPhaseFieldSym}}}
\newcommand{\pfCouplingIntMembr}[1][]{#1{C}_{#1{\membrPhaseFieldSym}}}

\newcommand{\refAss}[1]{Assumption~\ref{#1}}

\newcommand{\solTimeInt}{[0,\finTime]}
\def\schemeNormExt#1{\bar{#1}^{\normal{}}}
\newcommand{\fastVar}{z}
\newcommand{\projVar}{s}
\newcommand{\tanhIntConst}{Z}

\newcommand{\pfVelocityOuter}[2]{\pfVelocity[\schemeOuterExp]{#1}{#2}}
\newcommand{\pfSpeciesOneOuter}[2]{\pfSpeciesOne[\schemeOuterExp]{#1}{#2}}

\newcommand{\outerRegion}{\domain_0}
\newcommand{\minExpInd}{-N}
\newcommand{\absMinExpInd}{N}
\newcommand{\maxExpInd}{n}
\newcommand{\tangParamVar}{s}
\newcommand{\interfProto}[1]{\ONEARGFUN{S}{#1}}

\begin{document}

\title{Sharp interface analysis of a diffuse interface model for cell blebbing with linker dynamics}
\author[1]{Philipp Nöldner}
\author[2]{Martin Burger}
\author[3]{Harald Garcke}

\address[1,2]{\orgdiv{Department Mathematik}, \orgname{FAU Erlangen-Nürnberg}, \orgaddress{Cauerstr. 11, 91058 Erlangen, \country{Germany}}}
\address[3]{\orgdiv{Fakult\"at  für Mathematik}, \orgname{Universität Regensburg}, \orgaddress{\state{93040 Regensburg}, \country{Germany}}}

\abstract[Summary]{%
We investigate the convergence of solutions of a recently proposed diffuse interface/phase field model for cell blebbing by means of matched asymptotic expansions.
It is a biological phenomenon that increasingly attracts attention by both experimental and theoretical communities. Key to understanding the process of cell blebbing mechanically are proteins that link the cell cortex and the cell membrane. Another important model component is the bending energy of the cell membrane and cell cortex which accounts for differential equations up to sixth order.
Both aspects pose interesting mathematical challenges that will be addressed in this work like showing non-singularity formation for  the pressure at boundary layers, deriving equations for asymptotic series coefficients of uncommonly high order, and dealing with a highly coupled system of equations.}


\maketitle

\footnotetext{}

\section{Introduction}
	The phenomenon of cell blebbing is connected with various biological processes such as locomotion
    of primordial germ or cancer cells, the programmed cell death (apoptosis), or cell division.
    Its importance has been recognised and emphasized in the last decade  \cite{CP2008, PR2013, Othm2018}, 
    and attracts more and more interest.
    Cell blebbing results from chemical reactions that cause the selection of sites on the cell cortex,
    which lies underneath the cell membrane,
    where it contracts. This contraction causes the fluid
    inside the cell (the cytosol) to be pushed towards the cell membrane, which is then stretched out and moved away
    from the cell cortex. The cell membrane is pinned to the cell cortex via linker proteins. Only if 
    a sufficient amount of protein bonds can be broken, the membrane can
    freely develop a protrusion that is called a bleb.

    Besides experimental studies \cite{LGR2019} there are also many endeavours to understand cell blebbing from a theoretical
    perspective, cf. \cite{S2021, LCM2012, SG2013, SCLG2015, ACBS2015, AC2016, SDN2020, WBP2020}. While all these modelling
    approaches concentrate on selected aspects of the whole process, a full 3D model that brings together the linker proteins,
    their surface diffusion, and the fluid-structure interaction has only recently been proposed
    in \cite{WBFG2021}:
    the authors derive a phase field model in which cell cortex and cell membrane are defined by two
    coupled phase fields, with phase field parameter $\phaseFieldParam$,
    that interact with the cytosol. The coupling of the phase fields reflects the linker proteins
    connecting both surfaces and brings in new interesting mathematical challenges
    such as well-posedness of equations on evolving `diffuse manifolds' (the linker protein densities
    on the cell cortex undergo changes due to surface diffusion and bond breaking),
    developing numerical schemes for solving non-linear, sixth-order phase field equations,
    and answering the question what model is reached in the limit $\phaseFieldParam \to 0$.
    
    This article is aimed at investigating the last problem
    and showing that the phase field model of \cite{WBFG2021} formally approximates a sharp interface model
    that has also been derived by physical first principles \cite{W2021}. For that we will use
    the method of formal asymptotic analysis. 
    The techniques we employ are similar to those applied for the asymptotic analysis of 
    related phase field models like \cite{CF1986}, the Stokes-Allen-Cahn system in \cite{AL2018}, or the Willmore $L^2$-flow \cite{FL2021}.
    Another related asymptotic analysis is that of \cite{Wa2008} for minimisers of the Canham-Helfrich energy.

    We start by briefly recalling the phase field model from \cite{WBFG2021} and show the sharp interface
    system that is expected in the limit.
    After we have introduced the notation and gained some understanding of the system of 
    partial differential equations, we introduce foundations of the technique we use to pass to the limit 
    $ \phaseFieldParam \to 0 $. The major part of this paper follows, which is to plug in series expansions 
    of the solutions of the phase field model in powers of $ \phaseFieldParam $. Via separation of scales, we
    are able to derive equations for the leading order summands of the series. Using these findings, we can finally
    pass to the limit in  the equations of the phase field model
    and find the sharp interface system of equations that we initially reviewed.

    \paragraph{Preliminaries}
    	We denote the $n$-dimensional Lebesgue measure by $ \lebesgueM{n} $ and 
        the Hausdorff measure of Hausdorff dimension $ m $ by $ \hausdorffM{m} $. 
        Recall that for a two-dimensional submanifold $ \Gamma \subset \reals^3 $ with a smooth global chart
        $ \varphi\colon \Gamma \rightarrow \reals^2 $, and for a summable function $ f\colon\Gamma \rightarrow \reals $,
        it holds by definition
        $$
        	\integral{\Gamma}{}{f}{\hausdorffM{2}}
            =
            \integral{\varphi(\Gamma)}{}{f\concat\varphi^{-1} \jacobian{\varphi^{-1}}}{\lebesgueM{2}},
        $$
        where
        $ \jacobian{u} = \sqrt{\det\left(\transposed{\grad{}{}{}u}\grad{}{}{}u\right)} $
        is the Jacobian of $ u=\varphi^{-1} $.


        Let $ \Gamma \subseteq \reals^3 $ be a sufficiently smooth submanifold.
        We denote by $ \tubNeighbourhood{\delta}{\Gamma} = \set{ x \in \reals^3 }{ \dist{\Gamma}{x} < \delta } $
        the tubular neighbourhood around $\Gamma$,
        where $ \dist{\Gamma}{x} $ is the distance of $ x $ to $ \Gamma $ defined via the orthogonal projection;
        by $ \sdist{\Gamma}{x} $ we denote the signed distance.
        If we partition
        $ \tubNeighbourhood{\delta}{\Gamma} = \bigcup_{r\in(-\delta,\delta)} \Gamma_r $, where
        $ \Gamma_r = \set{x\in\tubNeighbourhood{\delta}{\Gamma}}{\sdist{\Gamma}{x} = r} $, we may define extensions 
        of quantities defined on $ \Gamma $ into $ \tubNeighbourhood{\delta}{\Gamma} $
        (cf.~\cite[Sec.~14.6]{GT2001}).
        The extended principal curvatures are defined as
        \begin{equation*}
            \begin{split}
                \pCurv[\bar]{i}{}{}\colon \tubNeighbourhood{\delta}{\Gamma} & \rightarrow \reals,
                \\
                x & \mapsto \pCurv{\Gamma_{\sdist{\Gamma}{x}},i}{}{x},
            \end{split}
        \end{equation*}
        where $ \pCurv{S,i}{}{x} $ is the $i$th principal curvature of the surface $ S $.
        Accordingly, the mean curvature is
        \begin{equation*}
            \begin{split}
                \extMeanCurv{}{}\colon\tubNeighbourhood{\delta}{\Gamma} & \rightarrow \reals,
                \\
                x & \mapsto \meanCurv{\Gamma_{\sdist{\Gamma}{x}}}{x}.
            \end{split}
        \end{equation*}
        Another extension method we will encounter is the \emph{normal extension} of a quantity $ f\colon\Gamma\rightarrow V $, 
        for a set $ V $,
        meaning that the quantity is extended constantly in normal direction. We denote 
        those extensions by $ \schemeNormExt{f} $.

        The surface gradient, $ \grad{\Gamma}{}{}f \vert_p $, of a function $ \funSig{f}{\Gamma}{\reals} $ in a point 
        $ p \in \Gamma $ is the vector
        $$
        	\grad{\Gamma}{}{}f\vert_p = \mathbb{P}_\Gamma(p) \grad{}{}{}\schemeNormExt{f}\vert_p,
        $$
        where $ \mathbb{P}_\Gamma(p) = \identityMatrix - \normal[\Gamma]{p}\tensorProd\normal[\Gamma]{p} $ 
        is the tangential projection onto $ \Gamma $. Other surface differential
        operators such as the divergence or the Jacobi-Matrix can be derived analogously.

        For a differentiable functional $ S \colon X \rightarrow [0,\infty) $ on a Banach space
        $ X $, the element $ \grad{}{Y}{} S(u) $, $ u \in X $, of a subspace $ Y \subseteq X $ that fulfills
        $$
        	\innerProd[Y]{\grad{}{Y}{}S(u)}{v} = S^\prime(u)v \quad\quad \forall v \in Y,
        $$
        where $ S^\prime(u) $ is the Gateaux-derivative of $ S $ in $u$, is called the $ Y $-gradient of $ S $.
        Consider, e.g., the functional 
        $ S(u) = \integral{B_1(0)}{}{\abs{\grad{}{}{}u(x)}^2}{\lebesgueM{3}(x)} $; its $ L^2$-gradient is
        $ \grad{}{L^2}{}S(u) = -\laplacian{}{}{}u $, whereas the $ H_0^1$-gradient takes the
        form $ \grad{}{H_0^1}{}S(u) = u $, and the $ H^{-1} $-gradient is
        $ \grad{}{H^{-1}}{}S(u) = \laplacian{}{2}{}u $.


\section{Modelling}
	Besides the numerical advantage of making topological changes such as pinch-offs (like when vesicles form out of the
    membrane) easy to handle, a phase field approach for modelling cell blebbing is also apt for bio-physical reasons:
    cell membranes are bilayers of lipid molecules which can be subject to undulations, and so the membrane
    is not strictly demarcated to the surrounding fluid. Depending on the scale we look at these membranes, the
    diameter of the lipid molecules involved, and the spacing between them, it may be desirable to model uncertainty in the
    lipid molecules' position and thus take them to be diffuse layers of some thickness $ \phaseFieldParam $. 
    Another pecularity when considering cell blebbing is experimental evidence
    \cite{LGR2019} that at sites where blebbing occurs, the cell membrane is folded multiple times
    providing for enough material to be unfolded, and is thus thicker than a typical biological membrane.

	Let us assume we observe the process of cell blebbing for a certain time $ \finTime \in (0,\infty) $
    in a domain $ \domain \subseteq \reals^3 $.    
    We consider two evolving diffuse interfaces---the cell membrane and the cell cortex---that can be defined as those subsets 
    of $ \domain $, on which
    phase fields $ \membrPhaseField{}{} $ (modelling the membrane) and $ \cortexPhaseField{}{} $ 
    (modelling the cell cortex) are close to zero, respectively. Additionally, there is a surrounding fluid with
    density $ \rho $, velocity $ \pfVelocity{}{} $, and pressure $ \pfPressure{}{} $.
    Also in the domain, but concentrated on the cell cortex, are linker proteins 
    with mass volume density $ \pfSpeciesOne{}{} $. They
    connect the cell membrane and the cell cortex. The linker proteins behave like springs, but may break if overstretched, so
    we introduce another density $ \pfSpeciesTwo{}{} $ which gives the mass of linkers per volume that are broken.
    This is important because `repairing mechanisms' of the cell take care of reconnecting those broken linkers 
    back to the cell membrane. A scheme in which the aforementioned quantities are all depicted together is given 
    in Figure~\ref{fig:phase field cell scheme}.
    \begin{figure}[h]
        \setlength{\squareLen}{8cm}
        \setlength{\cortexRad}{2cm}
        \setlength{\legendMargin}{3mm}
        \setlength{\legendHeight}{1cm}
        \pgfmathsetlength{\membraneRad}{\cortexRad+1cm}
        \pgfmathsetlength{\linkerFootRad}{\cortexRad+1mm}
      	\centering
        \begin{tikzpicture}[tips,scale=.7]
            \draw[fill=white] (0,0) -- (\squareLen,0) -- 
                (\squareLen,\squareLen) -- (0,\squareLen) -- (0,0);

            \draw[] (\squareLen/2,\squareLen/2) circle (\membraneRad+1mm);
            \draw[dotted,fill=yellow!70!red] (\squareLen/2,\squareLen/2) circle (\membraneRad);
            \draw[] (\squareLen/2,\squareLen/2) circle (\membraneRad-1mm);
            \pgfmathsetlengthmacro\ax{\squareLen/2 + cos(135)*(\membraneRad+.1cm)}
            \pgfmathsetlengthmacro\ay{\squareLen/2 + sin(135)*(\membraneRad+.1cm)}      
            \pgfmathsetlengthmacro\bx{\squareLen/2 + cos(135)*(\membraneRad-.1cm)}
            \pgfmathsetlengthmacro\by{\squareLen/2 + sin(135)*(\membraneRad-.1cm)} 
            \draw (\ax,\ay) node[anchor=south] {$ \membrPhaseField{t}{} $};
            \draw (\ax,\ay) node[anchor=north west] {$\eps$} --(\bx,\by);

            \draw[] (\squareLen/2,\squareLen/2) circle (\cortexRad+1mm);
            \draw[dotted,fill=yellow!40!red] (\squareLen/2,\squareLen/2) circle (\cortexRad);
            \draw[] (\squareLen/2,\squareLen/2) circle (\cortexRad-1mm);
            \pgfmathsetlengthmacro\ax{\squareLen/2 + cos(45)*(\cortexRad-.1cm)}
            \pgfmathsetlengthmacro\ay{\squareLen/2 + sin(45)*(\cortexRad-.1cm)}
            \pgfmathsetlengthmacro\bx{\squareLen/2 + cos(45)*(\cortexRad+.1cm)}
            \pgfmathsetlengthmacro\by{\squareLen/2 + sin(45)*(\cortexRad+.1cm)} 
            \draw (\ax,\ay) node[anchor=north east] {$ \cortexPhaseField{t}{} $};
            \draw (\ax,\ay) node[anchor=south west] {$\eps$} --(\bx,\by);

            \draw ($(1cm,\squareLen-0.5cm)$) node {
                $ \domain $
            };

            \draw (-1,\squareLen*0.25) node[anchor=west] {
              	$ \boundary{\domain} $
            };
        \end{tikzpicture}
        \caption{Illustration of the relationship of the two diffuse layers.
        The dotted lines indicate the centers of the
        transition layers of $ \membrPhaseFieldSym $ and $ \cortexPhaseFieldSym $.
        In the white region, both $ \membrPhaseField{}{} $ and $ \cortexPhaseField{}{} $ take values close to $ 1 $.
        In the light orange region, $ \cortexPhaseField{}{} $ takes values close to $ 1 $, but $ \membrPhaseField{}{} $ 
        has values close to $ -1 $.
        In the dark orange region, both $ \membrPhaseField{}{} $ and $ \cortexPhaseField{}{} $ 
        have values close to $ -1 $.}
        \label{fig:phase field cell scheme}
    \end{figure}    

    For deriving the phase field model, Onsager's variational principle \cite{O1931,QWS2006} is combined with a
    reaction-diffusion-like surface evolution equation for the active and inactive linker proteins.
    To establish a basic understanding of how a PDE system for cell blebbing can be obtained, let us mention
    the principle steps in the derivation.
    \begin{enumerate}
      \item Definition of an energy functional 
        $ \totalEnergyFunct{\pfVelocity{}{},\pfPressure{}{},\membrPhaseField{}{},\cortexPhaseField{}{}}{} $
        that is the sum of all kinds of energy of the cell:
        the ingredients are the kinetic energy of the fluid, 
        the surface and bending energy of the cell cortex and cell membrane,
        and a potential energy that accounts for the coupling of both membrane and cortex via the linker 
        proteins.
      \item Definition of appropriate boundary conditions (see below).
      \item Variation of $ \totalEnergyFunct{}{} $ plus a dissipation functional. With regard to the
        linker proteins, our process is assumed to be quasi-static, i.e., we assume
        the linker proteins to be given parameters of $ \totalEnergyFunct{}{} $ although their evolution
        is given by a reaction-diffusion-like surface equation.
      \item Extending the stationarity condition derived by the previous variation step, the aforementioned
        surface evolution equations
        for the linker proteins are added.
    \end{enumerate}

    \subsection{Phase field model}
        Several computations and formulae are the same for the phase field representing the cell membrane
        $ \membrPhaseField{}{} $ and that representing the cell cortex $ \cortexPhaseField{}{} $. For those,
        we always use the symbols
        $ \phaseFieldProto \in \{\membrPhaseFieldSym, \cortexPhaseFieldSym\} $, and
        $ \surfaceProto \in \{\interfOne{},\interfTwo{}\} $ to avoid copious repetition.

        In the phase field approach, we approximate two important geometrical quantities known from 
        the sharp interface perspective, namely the normal 
        \begin{align*}
          \normal[\surfaceProto]{} = \normal[\phaseFieldProto_\phaseFieldParam]{} + \landauBigo{\phaseFieldParam},\quad
          \normal[\phaseFieldProto_\phaseFieldParam]{}=\frac{\grad{}{}{}\phaseFieldProto_\phaseFieldParam}{\abs{\grad{}{}{}\phaseFieldProto_\phaseFieldParam}}
        \end{align*}
        (everywhere where $\phaseFieldProto_\phaseFieldParam\neq0$), and the mean curvature
        \begin{align*}
          \meanCurv{\surfaceProto}{} = \meanCurv{\phaseFieldProto_\phaseFieldParam}{} +
                                       \landauBigo{\phaseFieldParam},\quad
          \meanCurv{\phaseFieldProto_\phaseFieldParam}{} =
          \abs{\grad{}{}{}\phaseFieldProto_\phaseFieldParam}(-\phaseFieldParam\laplacian{}{}{}\phaseFieldProto_\phaseFieldParam + 
          \phaseFieldParam^{-1}\doubleWellPot[\schemeSupscr{\prime}]{\phaseFieldProto_\phaseFieldParam})
        \end{align*}
        with $ \doubleWellPot[]{\phaseFieldProto_\phaseFieldParam} = \frac{1}{4}\left(\phaseFieldProto_\phaseFieldParam^2 - 1\right)^2$.
        Having the velocity $ \velocity{}{} $ and density $ \fluidMassDens $ of the fluid, we may express the kinetic energy
        as 
        $$
            \frac{1}{2}\integral{\domain}{}{\fluidMassDens\abs{\velocity{}{}}^2}{\lebesgueM{3}}.
        $$
        Let us consider the following energies at a
        particular point in time $ t \in[0,\finTime]$, so we can ignore the time-dependency for now.
        The surface energy of the diffuse cell membrane with a surface tension proportional to $ \elComprMod_{\interfOne{}} $ 
        is given by the Ginzburg-Landau energy
        $$
            \ginzburgLandauEnergyFunct{\membrPhaseFieldSym}{\interfOne{}}
            =
            \elComprMod_{\interfOne{}}
            \integral{\domain}{}{
                \frac{\phaseFieldParam}{2}\abs{\grad{}{}{}\membrPhaseFieldSym}^2
                +
                \frac{1}{\phaseFieldParam}
                \doubleWellPot{\membrPhaseFieldSym}
            }{\lebesgueM{3}}
            =
            \elComprMod_{\interfOne{}}
            \integral{\domain}{}{
                \ginzburgLandauEnergyDens{\membrPhaseFieldSym}{}
            }{\lebesgueM{3}}
        $$
        with $ \ginzburgLandauEnergyDens{\membrPhaseFieldSym}{} = \frac{\phaseFieldParam}{2}
        \abs{\grad{}{}{}\membrPhaseFieldSym}^2 + \frac{1}{\phaseFieldParam}\doubleWellPot{\membrPhaseFieldSym} $.
        A well-established \cite{C1970,H1973,QLRW2009} model for the bending energy of a cell membrane
        with bending rigidity $ \bendRig_{\interfOne{}} $ and spontaneous mean curvature 
        $ \pfSpontMeanCurv^{\interfOne{}} $ is the phase field version of the Canham-Helfrich energy
        $$
            \pfGenWillmoreEnergyFunct{\membrPhaseFieldSym}{\interfOne{}}
            =
            \frac{\bendRig_{\interfOne{}}}{2\phaseFieldParam}
            \integral{\domain}{}{
                \left(
                    -\phaseFieldParam\laplacian{}{}{}\membrPhaseFieldSym 
                    + 
                    \left(
                        \frac{1}{\phaseFieldParam}
                        \membrPhaseFieldSym
                        +
                        \pfSpontMeanCurv^{\interfOne{}}
                    \right)
                    \left(
                        \membrPhaseFieldSym^2-1
                    \right)
                \right)^2
            }{\lebesgueM{3}}.
        $$    
        The spontaneous mean curvature corresponds to an intrinsic bending of the membrane which is typical 
        for biomembranes. The additional term in the energy introduced by that, however, does not introduce
        new theoretical challenges compared to using a Willmore functional, 
        which is why we will omit it for the sake of a straightforward presentation,
        i.e, $ \pfSpontMeanCurv^{\interfOne{}} = 0 $. 
        In this configuration, $ \pfGenWillmoreEnergyFunct{\membrPhaseFieldSym}{\interfOne{}} $
        is the phase field version of the Willmore energy.
        We simplify the situation for the cell cortex in that we assume it to be just a stiffer
        membrane thus employing the same types of energies just with different surface tension and
        bending rigidity.
        Both energies associated to membrane and cortex are summarised in the energy functionals
        $$
            \pfSurfEnergyProto{\phaseFieldProto} =
                \pfGenWillmoreEnergyFunct{\phaseFieldProto}{\surfaceProto}
                +
                \ginzburgLandauEnergyFunct{\phaseFieldProto}{\surfaceProto}.
        $$

        For the coupling of cell membrane and cell cortex, we account with a generalised Hookean spring energy.
        $$
            \pfCouplingEnergy{\membrPhaseField{}{},\cortexPhaseField{}{},\pfSpeciesOne{}{}}{} =
            \integral{\domain}{}{
                \ginzburgLandauEnergyDens{\membrPhaseFieldSym}{}(y)
                \frac{\linkersSpringConst}{2}\integral{\domain}{}{
                    \ginzburgLandauEnergyDens{\cortexPhaseFieldSym}{}(x)
                    \abs{x-y}^2
                    \pfSpeciesOne{t}{x}
                    \pfCouplingProb{x,y,\normal[\cortexPhaseField{}{}]{}}{}
                }{\lebesgueM{3}(x)}
            }{\lebesgueM{3}(y)},
        $$
        where
        $ \linkersSpringConst $ is a spring constant, and $ \pfCouplingProb{x,y,\normal[\cortexPhaseField{}{}]{}}{} $
        assigns to points $ x, y \in \domain $
        the particle-per-volume density of protein linkers connecting in direction $ x-y$.
        A possible choice is
        $$
            \pfCouplingProb{x,y,\normal[\cortexPhaseField{}{}]{}}{} =
            \tilde{\pfCouplingProb{}{}}\left(
                \frac{
                    (x-y)\cdot\normal[\cortexPhaseField{}{}]{x}
                }{
                    \abs{x-y}
                }
            \right),\quad
            \tilde{\pfCouplingProb{}{}}(r) =
            \hat{\pfCouplingProb{}{}} \exp\left(\frac{(r-1)^2}{s^2}\right)
        $$
        with $ s $ being a suitable standard deviation and $ \hat{\pfCouplingProb{}{}} $ an appropriate scaling factor.
        To outline the idea for this modelling choice, we first point out that $ \ginzburgLandauEnergyDens{\membrPhaseField{}{}}{} $
        can be pictured as a `smooth Dirac delta function' if $ \membrPhaseField{}{} $ is the so-called optimal profile
        $$
            x \mapsto \tanh\left( \frac{\sdist{\interfOne{}}{x}}{\phaseFieldParam\sqrt{2}} \right).
        $$ 
        The same holds for $ \ginzburgLandauEnergyDens{\cortexPhaseField{}{}}{} $, so that
        $$
            \integral{\domain}{}{\ginzburgLandauEnergyDens{\membrPhaseFieldSym}{} \cdot}{\lebesgueM{3}}
            \approx
            \frac{2\sqrt{2}}{3}\integral{\interfOne{t}}{}{\cdot}{\hausdorffM{2}}
        $$
        and
        $$
            \integral{\domain}{}{\ginzburgLandauEnergyDens{\cortexPhaseFieldSym}{} \cdot}{\lebesgueM{3}}
            \approx
            \frac{2\sqrt{2}}{3}\integral{\interfTwo{t}}{}{\cdot}{\hausdorffM{2}}
        $$
        approximate surface integrals for small $ \phaseFieldParam $. Thus,
        \begin{align*}
            \pfCouplingEnergy{\membrPhaseField{}{},\cortexPhaseField{}{},\pfSpeciesOne{}{}}{} &=
            \integral{\domain}{}{
                \ginzburgLandauEnergyDens{\membrPhaseFieldSym}{}(y)
                \frac{\linkersSpringConst}{2}\integral{\domain}{}{
                    \ginzburgLandauEnergyDens{\cortexPhaseFieldSym}{}(x)
                    \abs{x-y}^2
                    \pfSpeciesOne{t}{x}
                    \pfCouplingProb{x,y,\normal[\cortexPhaseField{}{}]{}}{}
                }{\lebesgueM{3}(x)}
            }{\lebesgueM{3}(y)} \\
            &\approx
            \integral{\interfOne{t}}{}{
                \frac{\linkersSpringConst}{2} 
                 \integral{\interfTwo{t}}{}{
                    \abs{x-y}^2
                    \pfSpeciesOne{t}{x}
                    \pfCouplingProb{x,y,\normal[\interfTwo{t}]{}}{}
                }{\hausdorffM{2}(x)}
            }{\hausdorffM{2}(y)}.
        \end{align*}

        Looking at the sharp interface equivalent of the coupling energy, we can identify
        \begin{enumerate}
          \item $ \frac{\linkersSpringConst}{2} \abs{x-y}^2 $ as a Hookean energy density, which is integrated 
            over the membrane and cortex, and weighted additionally by
          \item $ \pfCouplingProb{x,y,\normal[\interfTwo{t}]{}}{} $ 
            to incorporate the likeliness of the two spatial points $ x \in \interfTwo{} $,
            $ y \in\interfOne{} $ being connected, and
          \item the volume-density of linker particles $ \pfSpeciesOne{t}{x} $ actually linking.
        \end{enumerate}
        The Hookean energy ansatz accounts for the earlier mentioned assumption that the linker proteins
        behave like springs. Additionally, since 
        linkers might not be distributed homogeneously, 
        we should scale the coupling force by their actual amount, which explains 3.
        The necessity to consider a weight $ \pfCouplingProb{}{} $ might not be so obvious:
        it has not yet been agreed upon in the biological literature how to identify the pairs of 
        points $(x,y)\in\interfTwo{}\times\interfOne{}$ that are connected by protein linkers.
        That is why we allow the weight
        $ \pfCouplingProb{}{} $ to model a certain probability for this state.
        An easy way to describe such a probability is in terms of the angle between $ y-x $ and a gauge direction.
        As this gauge direction, we chose the cortex normal, which enters as the third argument
        of $ \pfCouplingProb{}{} $.

        \begin{rem}
            It shall be remarked that there are other choices for `smooth Dirac delta functions' like
            $ \frac{1}{\phaseFieldParam} \doubleWellPot{\phaseFieldProto} $, which is smoother and easier to handle
            analytically and numerically.
            It turns out, however, that for passing to the limit 
            $ \phaseFieldParam \to 0 $, the latter two choices are not appropriate. The reason for that 
            becomes clear when we compare the right hand side $ K $ of the momentum balance for the different
            choices of the integral weight: only for $ \ginzburgLandauEnergyDens{\phaseFieldProto}{} $, we have
            phase field counterparts in $ K $ for every term we expect in the sharp interface system
            as derived from physical first principles (cf. \cite{WBFG2021}). 
        \end{rem}

        Summing all potential energies, we obtain the Helmoltz free Energy of the cell as
        \begin{equation*}
            \freeEnergyFunct{\membrPhaseField{}{},\cortexPhaseField{}{},\pfSpeciesOne{}{}}{\phaseFieldParam}
            =
            \pfSurfEnergyOne{\membrPhaseField{}{}}
            +
            \pfSurfEnergyTwo{\cortexPhaseField{}{}}
            +
            \pfCouplingEnergy{
              \membrPhaseField{}{},\cortexPhaseField{}{},\pfSpeciesOne{}{}
            }{},
        \end{equation*}
        and the inner energy as
        \begin{equation*}
          \totalEnergyFunct{
            \pfVelocity{}{},\membrPhaseField{}{},\cortexPhaseField{}{},\pfSpeciesOne{}{}
          } = 
          \frac{1}{2}\integral{\domain}{}{\fluidMassDens\abs{\pfVelocity{}{}}^2}{\lebesgueM{3}}
          +
          \freeEnergyFunct{\membrPhaseField{}{},\cortexPhaseField{}{},\pfSpeciesOne{}{}}{\phaseFieldParam}.
        \end{equation*}

        Via Onsager's variational principle (cf.~\cite{WBFG2021}), 
        the following system of partial differential equations is then found as stationarity conditions
        \begin{subequations}
            \label{equ:modelling:phase field:resulting PDEs}
            \begin{equation}
                \label{equ:modelling:phase field resulting PDEs:momentum balance}
                \fluidMassDens(\pDiff{t}{}{}\pfVelocity{}{}
                +
                (\pfVelocity{}{}\cdot
                \grad{}{}{})
                \pfVelocity{}{})
                -
                \diver{}{}{
                    \viscosity
                    \left(
                        \grad{}{}{}\pfVelocity{}{}
                        +
                        \transposed{\grad{}{}{}\pfVelocity{}{}}
                    \right)
                    -
                    \pfPressure{}{}
                }
                = K,
            \end{equation}
            \begin{equation}
                \label{equ:modelling:phase field resulting PDEs:incompressibility}
                \diver{}{}{}\velocity{}{} = 0,
            \end{equation}
            \begin{equation}
                \label{equ:modelling:phase field resulting PDEs:first interface evolution}
                \pDiff{t}{}{}\symMembrPhaseField 
                +
                \symVelocity
                \cdot
                \grad{}{}{}
                \symMembrPhaseField
                =
                \diver{}{}{
                    \membrPhaseFieldMob{\symMembrPhaseField}
                    \left(
                    \grad{}{}{
                        \grad{\membrPhaseFieldSym}{L^2}{}
                        \pfSurfEnergyOne{\membrPhaseFieldSym}
                    }
                    +
                    \grad{}{}{
                        \grad{\symMembrPhaseField}{L^2}{}
                        \pfCouplingEnergy{\membrPhaseFieldSym,\cortexPhaseFieldSym,
                        \pfSpeciesOne{}{}}{}
                    }
                    \right)
                },
            \end{equation}
            \begin{equation}
                \label{equ:modelling:phase field:resulting PDEs:second interface evolution}
                \pDiff{t}{}{}\symCortexPhaseField 
                +
                \symVelocity
                \cdot
                \grad{}{}{}
                \symCortexPhaseField
                =
                \diver{}{}{
                    \cortexPhaseFieldMob{\symCortexPhaseField}
                    \left(
                    \grad{}{}{
                        \grad{\cortexPhaseFieldSym}{L^2}{}
                        \pfSurfEnergyTwo{\cortexPhaseFieldSym}
                    }
                    +
                    \grad{}{}{
                        \grad{\symCortexPhaseField}{L^2}{}
                        \pfCouplingEnergy{\membrPhaseFieldSym,\cortexPhaseFieldSym,
                        \pfSpeciesOne{}{}}{}
                    }
                    \right)
                },
            \end{equation}
            where
            \begin{align*}
                K &= 
                \grad{}{L^2}{}
                \pfSurfEnergyOne{\membrPhaseFieldSym}
                \grad{}{}{}\symMembrPhaseField
                +
                \grad{}{L^2}{}
                \pfSurfEnergyTwo{\cortexPhaseFieldSym}
                \grad{}{}{}\symCortexPhaseField
                \\
                &+\grad{\membrPhaseFieldSym}{L^2}{}
                \pfCouplingEnergy{\membrPhaseFieldSym,\cortexPhaseFieldSym,\pfSpeciesOne{}{}}{}
                \grad{}{}{}\membrPhaseFieldSym
                +\grad{\cortexPhaseFieldSym}{L^2}{}
                \pfCouplingEnergy{\membrPhaseFieldSym,\cortexPhaseFieldSym,\pfSpeciesOne{}{}}{}
                \grad{}{}{}\cortexPhaseFieldSym
                \\
                &-
                \integral{\domain}{}{
                    \ginzburgLandauEnergyDens{\membrPhaseField{}{}}{}(y)
                    \pDiff{\pfSpeciesOne{}{}}{}{} \pfCouplingEnergyDens{}{}
                    (\cdot,y,\pfSpeciesOne{}{},\normal[\cortexPhaseFieldSym]{})
                    \meanCurv{\cortexPhaseField{}{}}{}
                    \pfSpeciesOne{}{}
                    \normal[\cortexPhaseFieldSym]{}
                }{\lebesgueM{3}(y)}
                \\
                &-
                \integral{\domain}{}{
                    \ginzburgLandauEnergyDens{\membrPhaseField{}{}}{}(y)
                    \orthProjMat{\normal[\cortexPhaseFieldSym]{}}{}
                    \grad{}{}{\pDiff{\pfSpeciesOne{}{}}{}{} \pfCouplingEnergyDens{}{}
                    (\cdot,y,\pfSpeciesOne{}{},\normal[\cortexPhaseFieldSym]{})
                    }
                    \ginzburgLandauEnergyDens{\cortexPhaseFieldSym}{}
                    \pfSpeciesOne{}{}
                }{\lebesgueM{3}(y)}.
            \end{align*}
            The imposed boundary conditions are
            \begin{align}
                \label{equ:modelling:phase field:resulting PDEs:boundary conditions:%
                velocity}
                \restrFun{\pfVelocity{}{}}{\boundary{\domain}} &= 0,
                \\
                \label{equ:modelling:phase field:resulting PDEs:boundary conditions:%
                phase field}
                \restrFun{\pDiff{\normal{}}{}{}\membrPhaseField{}{}}{\boundary{\domain}} = 
                \restrFun{\pDiff{\normal{}}{}{}\cortexPhaseField{}{}}{\boundary{\domain}}
                &= 0,
                \\
                \label{equ:modelling:phase field:resulting PDEs:boundary conditions:%
                fluxes}
                \restrFun{\symMembrPhaseFieldFlux}{\boundary{\domain}}\cdot\normal{} =
                \restrFun{\symCortexPhaseFieldFlux}{\boundary{\domain}}\cdot\normal{} &= 0,
                \\
                \label{equ:modelling:phase field:resulting PDEs:boundary conditions:%
                species}
                \restrFun{\pfSpeciesOne{}{}}{\boundary{\domain}} = 
                \restrFun{\pfSpeciesTwo{}{}}{\boundary{\domain}} &= 0,
            \end{align}
            where 
            $$ 
                \symMembrPhaseFieldFlux =
                \grad{}{}{
                    \grad{\membrPhaseFieldSym}{L^2}{}
                    \pfSurfEnergyOne{\membrPhaseFieldSym}
                    +
                    \grad{\symMembrPhaseField}{L^2}{}
                    \pfCouplingEnergy{\membrPhaseFieldSym,\cortexPhaseFieldSym,
                    \pfSpeciesOne{}{}}{}
                },
            $$
            and 
            $$ 
                \symCortexPhaseFieldFlux = 
                \grad{}{}{
                    \grad{\cortexPhaseFieldSym}{L^2}{}
                    \pfSurfEnergyTwo{\cortexPhaseFieldSym}
                    +
                    \grad{\symCortexPhaseField}{L^2}{}
                    \pfCouplingEnergy{\membrPhaseFieldSym,\cortexPhaseFieldSym,
                    \pfSpeciesOne{}{}}{}
                }.
            $$

            In addition, we consider evolution equations for the active and inactive linkers 
            on the diffuse surface of the cell cortex:
            \begin{align}
                \label{equ:modelling:phase field:resulting PDEs:reaction-diffusion of species one}
                \ginzburgLandauEnergyDens{\cortexPhaseField{}{}}{}
                \pDiff{t}{}{}\pfSpeciesOne{}{}
                -
                \symVelocityNormal{\cortexPhaseField{}{}}
                \meanCurv{\cortexPhaseField{}{}}{}
                \pfSpeciesOne{}{}
                -
                \diver{}{}{
                    \ginzburgLandauEnergyDens{\cortexPhaseField{}{}}{}
                    \speciesOneDiffusiv{}
                    \grad{}{}{}\speciesOne{}{}
                }
                +
                \diver{}{}{
                    \ginzburgLandauEnergyDens{\cortexPhaseField{}{}}{}
                    \pfVelocity{}{}_\tau
                    \speciesOne{}{}
                }
                =\notag\\
                \ginzburgLandauEnergyDens{\cortexPhaseField{}{}}{}
                \reactionOne{\speciesOne{}{},\speciesTwo{}{},\membrPhaseField{}{},
                \normal[\cortexPhaseFieldSym]{}}{},
                \\
                \label{equ:modelling:phase field:resulting PDEs:reaction-diffusion of species two}
                \ginzburgLandauEnergyDens{\cortexPhaseField{}{}}{}
                \pDiff{t}{}{}\pfSpeciesTwo{}{}
                -
                \symVelocityNormal{\cortexPhaseField{}{}}
                \meanCurv{\cortexPhaseField{}{}}{}
                \pfSpeciesTwo{}{}
                -
                \diver{}{}{
                    \ginzburgLandauEnergyDens{\cortexPhaseField{}{}}{}
                    \speciesTwoDiffusiv{}
                    \grad{}{}{}\speciesTwo{}{}
                }
                +
                \diver{}{}{
                    \ginzburgLandauEnergyDens{\cortexPhaseField{}{}}{}
                    \pfVelocity{}{}_\tau
                    \speciesTwo{}{}
                }
                =\notag\\
                -\ginzburgLandauEnergyDens{\cortexPhaseField{}{}}{}
                \reactionOne{\speciesOne{}{},\speciesTwo{}{},\membrPhaseField{}{},      	
                \normal[\cortexPhaseFieldSym]{}}{},
            \end{align}
        \end{subequations}
        where
        \begin{align*}
            \reactionOne{\pfSpeciesOne{}{},\pfSpeciesTwo{}{},\membrPhaseField{}{},
            \normal[\cortexPhaseFieldSym]{}}{}
            =
            \repairRate \pfSpeciesTwo{}{}
            - 
            \pfSpeciesOne{}{}
            \discRate{ 
                \membrPhaseField{}{},
                \normal[\cortexPhaseFieldSym]{}
            }.
        \end{align*}
        The term $ \repairRate \pfSpeciesTwo{}{} $ is the effective reconnection rate, $ \repairRate \geq 0 $, 
        of the inactive linkers, and
        \[
            \pfSpeciesOne{}{}
            \discRate{ 
                \membrPhaseField{}{},
                \normal[\cortexPhaseFieldSym]{}
            }
        \]
        is the effective disconnection rate of the active linkers in relation to the membrane position in space
        and the orientation of the cortex given by its normal. 

        For a thorough discussion and further references, the reader may please refer to \cite{W2021}.
        In the following section, we describe steps one, two and four, 
        but leave out the lengthy calculations involved for step three.
        For the following discussion, however, we need
        the concrete expression for all the $ L^2 $-gradients of the energies, so we give them here without doing the calculations.
        Note that these calculations depend on the boundary conditions 
        \eqref{equ:modelling:phase field:resulting PDEs:boundary conditions:%
        velocity},
        \eqref{equ:modelling:phase field:resulting PDEs:boundary conditions:%
        phase field},
        \eqref{equ:modelling:phase field:resulting PDEs:boundary conditions:%
        fluxes}, and
        \eqref{equ:modelling:phase field:resulting PDEs:boundary conditions:%
        species}:
        \begin{subequations}
        \begin{align}
            \grad{\phaseFieldProto}{L^2}{}\pfSurfEnergyProto{} &=
            \grad{\phaseFieldProto}{L^2}{}
            \pfGenWillmoreEnergyFunct{}{}
            +
            \grad{\phaseFieldProto}{L^2}{}
            \ginzburgLandauEnergyFunct{}{},
            \\                                             
            \grad{\phaseFieldProto}{L^2}{}
            \ginzburgLandauEnergyFunct{}{}
            &=
            -\phaseFieldParam \laplacian{}{}{}\phaseFieldProto
            +
            \frac{1}{\phaseFieldParam}\doubleWellPot[\schemeSupscr{\prime}]{\phaseFieldProto}
            \equalscolon
            \genGlChemPot{\phaseFieldProto}{},
            \\
            \grad{\phaseFieldProto}{L^2}{}
            \pfGenWillmoreEnergyFunct{}{}
            &= -\laplacian{}{}{ 
                    \genGlChemPot{\phaseFieldProto}{}
               }
               +
               \genGlChemPot{\phaseFieldProto}{}
               \frac{1}{\phaseFieldParam^{2}}
              \doubleWellPot[\schemeSupscr{\prime\prime}]{\phaseFieldProto}{}.
        \end{align}
        For easier expression of the coupling energy gradients, we introduce
        \begin{align*}
            \pfCouplingIntCortex(t,y) &= \integral{\domain}{}{
                \ginzburgLandauEnergyDens{\symCortexPhaseField}{}(x)
                \pfCouplingEnergyDens{x,y}{\speciesOne{t}{x},
                  \normal[\cortexPhaseField{t}{\NOARG}]{x}}
            }{\lebesgueM{3}(x)},
            \\
            \pfCouplingIntMembr(t,x) &= \integral{\domain}{}{
                \ginzburgLandauEnergyDens{\membrPhaseFieldSym}{}(y)
                \pfCouplingEnergyDens{x,y}{\speciesOne{t}{x},
                  \normal[\cortexPhaseField{t}{\NOARG}]{x}}
            }{\lebesgueM{3}(y)}.
        \end{align*}
        Then,
        \begin{align}
            \label{equ:modelling:grad mempf coupling energy}
            \grad{\membrPhaseField{}{}}{L^2}{}
            \pfCouplingEnergy{}{}
            &=
            \genGlChemPot{\membrPhaseFieldSym}{0}(y)\pfCouplingIntCortex(y)
            -
            \integral{\domain}{}{
                \phaseFieldParam
                \ginzburgLandauEnergyDens{\cortexPhaseField{}{}}{}(x)
                \grad{y}{}{}\membrPhaseField{}{}
                \cdot
                \grad{y}{}{\pfCouplingEnergyDens{x,y}{\pfSpeciesOne{}{},\normal[\cortexPhaseFieldSym]{}}}
            }{\lebesgueM{3}(x)},
            \\
            \label{equ:modelling:grad cortpf coupling energy}
            \grad{\cortexPhaseField{}{}}{L^2}{}
            \pfCouplingEnergy{}{}
            &=
            \genGlChemPot{\cortexPhaseField{}{}}{0}(x)
            \pfCouplingIntMembr(x)
            -\integral{\domain}{}{
                \phaseFieldParam
                \ginzburgLandauEnergyDens{\membrPhaseField{}{}}{}(y)
                \grad{x}{}{
                    \pfCouplingEnergyDens{x,y}{\pfSpeciesOne{}{},\normal[\cortexPhaseFieldSym]{}}
                }
                \cdot
                \grad{x}{}{}\cortexPhaseField{}{}
            }{\lebesgueM{3}(y)}
            \\\nonumber
            &-
            \integral{\domain}{}{
                \ginzburgLandauEnergyDens{\membrPhaseField{}{}}{}(y)
                \diver{x}{}{
                    \ginzburgLandauEnergyDens{\cortexPhaseField{}{}}{}(x)
                    \transposed{
                    \grad{\normal{}}{}{}\pfCouplingEnergyDens{x,y}{\pfSpeciesOne{}{},
                    \normal[\cortexPhaseFieldSym]{}}
                    }
                    \frac{1}{\abs{\grad{}{}{}\cortexPhaseFieldSym}}
                    \orthProjMat{\normal[\cortexPhaseFieldSym]{}}{}
                }
            }{\lebesgueM{3}(y)}.
        \end{align}
        \end{subequations}
    
    Solutions of \eqref{equ:modelling:phase field:resulting PDEs}
    fulfil an energy inequality, cf. \cite{WBFG2021}.
    This energy inequality reads
    \begin{equation}
        \label{equ:formal asymptotics:energy inequ:pre}
        \begin{split}
            \diff{t}{}{}\freeEnergyFunct{\membrPhaseField{}{},\cortexPhaseField{}{},\pfSpeciesOne{}{}}{\phaseFieldParam}
            \leq
            &-
            \VLTNorm{\domain}{(3,3)}{\grad{}{}{}\pfVelocity{}{}}^2
            \\
            &-
            \membrPhaseFieldMob{\membrPhaseField{}{}}
            \LTNorm{\domain}{
                \grad{}{}{
                    \grad{\membrPhaseField{}{}}{L^2}{}
                    \pfSurfEnergyOne{\membrPhaseField{}{}}{}
                    +
                    \grad{\membrPhaseField{}{}}{L^2}{}
                    \pfCouplingEnergy{\membrPhaseField{}{},\cortexPhaseField{}{},\pfSpeciesOne{}{}}{}
                }
            }^2
            \\
            &-
            \cortexPhaseFieldMob{\cortexPhaseField{}{}}
            \LTNorm{\domain}{
                \grad{}{}{
                    \grad{\cortexPhaseField{}{}}{L^2}{}
                    \pfSurfEnergyTwo{\cortexPhaseField{}{}}{}
                    +
                    \grad{\cortexPhaseField{}{}}{L^2}{}
                    \pfCouplingEnergy{\membrPhaseField{}{},\cortexPhaseField{}{},\pfSpeciesOne{}{}}{}
                }
            }^2\\
            &+
            \integral{\domain}{}{
                \ginzburgLandauEnergyDens{\membrPhaseFieldSym}{}(y)
                \integral{\domain}{}{
                    \ginzburgLandauEnergyDens{\cortexPhaseFieldSym}{}(x)
                    \pDiff{\pfSpeciesOne{}{}}{}{}
                    \pfCouplingEnergyDens{}{}[
                    \pDiff{t}{}{}\pfSpeciesOne{}{}]
                }{\lebesgueM{3}(x)}
            }{\lebesgueM{3}(y)}
            \\
            &-\integral{\domain}{}{
                \meanCurv{\cortexPhaseFieldSym}{t,x}
                \pfSpeciesOne{t}{x}
                \symVelocityNormal{\cortexPhaseFieldSym}(t,x)
                \integral{\domain}{}{
                    \ginzburgLandauEnergyDens{\membrPhaseFieldSym}{}(y)
                    \pDiff{\pfSpeciesOne{}{}}{}{}
                    \pfCouplingEnergyDens{x,y,\pfSpeciesOne{}{},\normal[\cortexPhaseFieldSym]{}}{}
                }{\lebesgueM{3}(y)}
            }{\lebesgueM{3}(x)}\\
            &
            -
            \integral{\domain}{}{
            \integral{\domain}{}{
                \ginzburgLandauEnergyDens{\membrPhaseField{}{}}{}(y)
                \grad{}{}{\pDiff{\pfSpeciesOne{}{}}{}{} \pfCouplingEnergyDens{}{}
                (\cdot,y,\pfSpeciesOne{}{},\normal[\cortexPhaseFieldSym]{})
                }
                \cdot
                \symVelocityTangVec
                \ginzburgLandauEnergyDens{\cortexPhaseFieldSym}{}
                \pfSpeciesOne{}{}
            }{\lebesgueM{3}(y)}
            }{\lebesgueM{3}(x)}.
        \end{split}
    \end{equation}

    \subsection{Sharp interface model}
	We introduce two evolving, two-dimensional manifolds $ \interfOne{}_\finTime = \left(\interfOne{t}\right)_{t\in[0,\finTime]} $ for
    the cell membrane, and $ \interfTwo{}_\finTime = \left(\interfTwo{t}\right)_{t\in[0,\finTime]} $ for the cell cortex.
    These evolving manifolds can also be described as the level sets $ \interfOne{t} =\membrLvlSet{t}{0} $
    and $ \interfTwo{t} = \cortexLvlSet{t}{0} $ of functions $ \membrLvlSetFun{}{}\colon \domain \times [0,\finTime] \rightarrow
    \reals $ and $ \cortexLvlSetFun{}{}\colon \domain \times [0,\finTime] \rightarrow \reals $.
    The cell we consider is swimming in a fluid with pressure $ \pressure{}{} $ and velocity $ \velocity{}{} $.
    Additionally, we have the density $ \speciesOne{}{}\colon \interfTwo{}_\finTime \rightarrow
    \reals $ of linker proteins connecting cell membrane and cell cortex,
    which we call active linkers. Another density $ \speciesTwo{}{}\colon \interfTwo{}_\finTime \rightarrow \reals $ 
    is introduced to model the density of the
    disconnected or broken proteins, called inactive linkers; these no longer couple cell membrane and cell cortex, but may be 
    reconnected due to healing mechanisms inside the cell.
    \def\intDom{%
      \overset{\circ}{\domain}%
    }
    $\intDom = \domain\setminus(\interfOne{t}\cup\interfTwo{t}) $

    \begin{subequations}
        \label{equ:modelling:sharp interface system}
        \begin{align}
            \label{equ:modelling:sharp interface classical PDE model:momentum balance}
            \fluidMassDens(\pDiff{t}{}{}\velocity{}{} 
            +
            \left( \velocity{}{} \cdot \grad{}{}{} \right) \velocity{}{})
            - 
            \diver{}{}{} \cauchyStressTens{}{}
            &=
            0
            &\text{in}\; & \intDom
            \\
            \label{equ:modelling:sharp interface classical PDE model:mass cons}
            \diver{}{}{}\velocity{}{} &= 0
            &
            \text{in}\; & \intDom
            \\
            \velocity{t}{} &= 0
            &\text{on}\; & \boundary{\domain},
            \\
            \label{equ:modelling:sharp interface classical PDE model:no jump if one}
            \jump{\velocity{}{}}_{\interfOne{t}}
            &=
            0
            & \text{on}\; & \interfOne{t},
            \\
            \label{equ:modelling:sharp interface classical PDE model:no jump if two}
            \jump{\velocity{}{}}_{\interfTwo{t}}
            &=
            0
            & \text{on}\; & \interfTwo{t},
            \\\label{equ:modelling:sharp interface classical PDE model:normal stress jump if one}
            -\jump{ \cauchyStressTens{}{}\normal{}}
            &=
            \grad{\membrLvlSetFun{}{}}{L^2}{}
            \surfEnergyOne{}
            \grad{}{}{}\membrLvlSetFun{}{}
            -
            \left(\grad{y}{}{}\couplingIntCortex^0\cdot\normal[\interfOne{}]{}\right)
            \normal[\interfOne{}]{}
            +
            \meanCurv{\interfOne{}}{}
            \couplingIntCortex^0
            \normal[\interfOne{}]{}
            &
            \text{on}\; & \interfOne{t},
            \\\nonumber
            -\jump{ \cauchyStressTens{}{}\normal{}}
            &= 
            \grad{\cortexLvlSetFun{}{}}{L^2}{}
            \surfEnergyTwo{}
            \grad{}{}{}\cortexLvlSetFun{}{}
            -
            \left(\grad{x}{}{}\couplingIntMembr^0\cdot\normal[\interfTwo{}]{}\right)
            \normal[\interfTwo{}]{}
            +
            \meanCurv{\interfTwo{}}{}
            \couplingIntMembr^0
            \normal[\interfTwo{}]{}
            \\\nonumber
            &-
            \pDiff{\speciesOne{}{}}{}{}
            \couplingIntMembr^0
            \meanCurv{\interfTwo{}}{}
            \speciesOne{}{}
            \normal[\interfTwo{}]{}
            -
            \grad{\interfTwo{}}{}{
            \pDiff{\speciesOne{}{}}{}{}
            \couplingIntMembr^0}
            \speciesOne{}{}
            -
            \diver{\interfTwo{}}{}{
                \grad{\normal{}}{}{}\couplingIntMembr^0
            }
            \normal[\interfTwo{}]{}
            \\\label{equ:modelling:sharp interface classical PDE model:normal stress jump if two}
            &-
            \meanCurv{\interfTwo{}}{}
            \left(
            \grad{\normal{}}{}{}\couplingIntMembr^0
            \cdot
            \normal[\interfTwo{}]{}
            \right)
            \normal[\interfTwo{}]{}                
            &
            \text{on}\; & \interfTwo{t},
            \\\label{equ:modelling:sharp interface classical PDE: HJE membr}
            \pDiff{t}{}{} \membrLvlSetFun{}{}
            +
            \velocity{}{}
            \cdot
            \grad{}{}{}\membrLvlSetFun{}{}
            &=
            0
            &\text{in}\; & \domain,
            \\\label{equ:modelling:sharp interface classical PDE: HJE cortex}
            \pDiff{t}{}{} \cortexLvlSetFun{}{}
            +
            \velocity{}{}
            \cdot
            \grad{}{}{}\cortexLvlSetFun{}{}
            &=
            0
            &\text{in}\; & \domain,
            \\\label{equ:modelling:sharp interface classical PDE:species one evoluation}
            \pDiff{t}{}{}\speciesOne{}{} 
            -
            \meanCurv{}{}
            \symVelocityNormal{\cortexLvlSetFun{}{}}
            \speciesOne{}{}
            -
            \diver{\interfTwo{t}}{}{
              \speciesOneDiffusiv{}
              \grad{}{}{}\speciesOne{}{}
            }
            +
            \diver{\interfTwo{t}}{}{\speciesOne{}{}\symVelocityTangVec}
            &=
            \reactionOne{
              \speciesOne{}{},\speciesTwo{}{},
              \membrLvlSetFun{}{},\normal[\interfTwo{}]{}
            }{}
            &
            \text{on}\; & \interfTwo{t},
            \\
            \pDiff{t}{}{}\speciesTwo{}{} 
            -
            \meanCurv{}{}
            \symVelocityNormal{\cortexLvlSetFun{}{}}
            \speciesTwo{}{}
            - 
            \diver{\interfTwo{t}}{}{
              \speciesTwoDiffusiv{}
              \grad{}{}{}\speciesTwo{}{}
            }
            +
            \diver{\interfTwo{t}}{}{\speciesTwo{}{}\symVelocityTangVec}
            &=
            \reactionTwo{
              \speciesOne{}{},\speciesTwo{}{},
              \membrLvlSetFun{}{},\normal[\interfTwo{}]{}
            }{}
            &\text{on}\; & \interfTwo{t}.
        \end{align}
    \end{subequations}
\section{Formal asymptotic analysis}
    Having outlined the physical principles, we are going to analyse the
    sharp interface limit of the phase field model.
    Let us now turn to the main result of this paper: we will demonstrate, using the method
    of formal asymptotic expansions, that classical solutions of the system \eqref{equ:modelling:phase field:resulting PDEs}
    converge, for $ \phaseFieldParam \to 0 $, to solutions of \eqref{equ:modelling:sharp interface system}.
    For a thorough theoretical introduction into the subject of formal asymptotic expansions, we refer to 
    \cite{E1979}, whereas a more application-oriented perspective is taken in \cite{Ho1995}.

    \subsection{Interfacial coordinates}
    	For the following analysis, we will need a coordinate transformation typical for asymptotic
        analysis of phase field equations for which boundary layers are expected in the regions where the
        phase fields are close to zero.

        Let us denote a tubular neighbourhood of a 
        smooth, orientable hypersurface $ S \subseteq \reals^3 $ by $ \tubNeighbourhood{\ifCoordsDelta}{S} $.
        We require that $ \ifCoordsDelta \in (0,\infty) $ is small enough such that
        $ \tubNeighbourhood{\ifCoordsDelta}{\interfOne{t}} \cap
        \tubNeighbourhood{\ifCoordsDelta}{\interfTwo{t}} = \emptyset $ for all $ t\in[0,\finTime] $.
        The local boundary layer coordinates, or \emph{interfacial coordinates} (as they are most often
        termed in this context), with respect to $ S $ are defined by the map
        \begin{equation*}
            \begin{split}
                \interfCoords[S,\phaseFieldParam]{}{}&\colon
                \tubNeighbourhood{\delta}{S}
                \rightarrow
                S
                \times
                \reals,
                \\
                x &\mapsto 
                \left(
                    \orthProj{S}{x},
                    \frac{\sdist{S}{x}}{\phaseFieldParam} 
                \right).
            \end{split}
        \end{equation*}
        For two evolving manifolds $\interfOne{}_\finTime  $, $\interfTwo{}_\finTime $,
        we extent this definition to
        \begin{equation*}
            \begin{split}
                \interfCoords[\phaseFieldParam]{}{}&\colon
                \bigcup_{t\in\solTimeInt}
                \left\{t\right\}
                \times
                \left(
                \tubNeighbourhood{\delta}{\interfOne{t}}
                \cup
                \tubNeighbourhood{\delta}{\interfTwo{t}}
                \right)
                \rightarrow
                \bigcup_{t\in\solTimeInt}
                \left\{t\right\}
                \times
                \left(
                \interfOne{t}
                \cup
                \interfTwo{t}
                \right)
                \times
                \reals,
                \\
                (t,x) &\mapsto
                \begin{cases}
                    (t,\interfCoords[\interfOne{t},\phaseFieldParam]{\NOARG}{x}) 
                    & 
                    x \in \tubNeighbourhood{\delta}{\interfOne{t}} 
                    \\
                    (t,\interfCoords[\interfTwo{t},\phaseFieldParam]{\NOARG}{x}) 
                    & 
                    x \in \tubNeighbourhood{\delta}{\interfTwo{t}} 
                \end{cases},
            \end{split}
        \end{equation*}
        and then set 
        $$
            \interfCoords[S_\finTime,\phaseFieldParam]{}{}
            =
            \restrFun{\interfCoords[\phaseFieldParam]{}{}}{
                \bigcup_{t\in\solTimeInt}
                \left\{t\right\}
                \times
                S(t)
            }
        $$
        for $ S(t) \in \left\{\interfOne{t},\interfTwo{t}\right\} $.
        We always consider $ \ifCoordsDelta $ small enough such that
        the interfacial coordinate transformations are well-defined.
        Generally, for a function $ f $ on 
        $ \bigcup_{t\in[0,\finTime]}\{t\}\times\left(\tubNeighbourhood{\ifCoordsDelta}{\interfOne{t}} 
        \cup \tubNeighbourhood{\ifCoordsDelta}{\interfTwo{t}}\right)$,
        we define
        \begin{equation*}
            \schemeLocalFun{f}\concat \interfCoords[\phaseFieldParam]{t}{x}
            =
            f(t,x).
        \end{equation*}
        The function $ \schemeLocalFun{f} $ depends on three arguments:
        The first is time, the second a point on one of the manifolds $ \interfOne{t} $ or $ \interfTwo{t} $,
        and the third a real number from 
        $ \left(-\frac{\delta}{\phaseFieldParam},\frac{\delta}{\phaseFieldParam}\right) $.
        The latter is occasionally referred to as `fast variable' and derivatives with respect
        to this variable are denoted by $ \simpleDeriv{\prime}{(\cdot)} $; derivatives with respect
        to the first variable are denoted by $ \pDiff{s}{}{\cdot} $.

        The following (standard) formulae will be important later.
        \begin{lem}[cf.~{\cite[Sec.~14.6]{GT2001}}]
            \label{lemma:formal asymptotics:expansions}
            Let $ S \subseteq \reals^3 $ be a real, orientable, and sufficiently smooth
            submanifold and $ \tubNeighbourhood{\delta}{S} $, $ \delta \in \reals_{>0} $,
            a tubular neighbourhood
            on which all the following extended functions are defined. 
            For all $ x \in \tubNeighbourhood{\ifCoordsDelta}{S} $,
            it holds
            \begin{subequations}
                \begin{equation}
                    \label{it:lemma:formal asymptotics:expansions:mean curv}
                    \begin{split}
                        \extMeanCurv{x}
                        &=
                        \sum^2_{i=1} 
                        \frac{
                            \pCurv[\schemeNormExt]{S,i}{x}{}
                        }{
                            1
                            -
                            \sdist{S}{x}\pCurv[\schemeNormExt]{S,i}{x}{}
                        }
                        \\
                        &=
                        \sum_{i=1}^2 \pCurv[\schemeNormExt]{S,i}{x}{}
                        +
                        \sdist{S}{x}\pCurv[\schemeNormExt]{S,i}{x}{}^2
                        +
                        \landauBigo{\sdist{S}{x}^2}
                        \\
                        &=
                        \sum_{i=1}^2
                        \left(
                            \pCurv[\hat]{S,i}{}{}
                            +
                            \phaseFieldParam\fastVar
                            \pCurv[\hat]{S,i}{}{}^2
                            +
                            \landauBigo{\phaseFieldParam^2}
                        \right)
                        \concat
                        \interfCoords[S,\phaseFieldParam]{\NOARG}{\orthProj{S}{x}},
                    \end{split}
                \end{equation}
                \begin{equation}
                    \label{it:lemma:formal asymptotics:expansions:mean curv grad}
                    \begin{split}
                        \grad[x]{}{}{\extMeanCurv{}}\cdot\extNormal{x}
                        &=
                        \sum_{i=1}^2 
                        \frac{
                            \pCurv[\schemeNormExt]{S,i}{x}{}^2
                        }{
                            (1-\sdist{S}{x}\pCurv[\schemeNormExt]{S,i}{x}{})^2
                        }
                        \\
                        &=
                        \sum_{i=1}^2 
                        \pCurv[\schemeNormExt]{S,i}{x}{}^2
                        +
                        2
                        \sdist{S}{x}
                        \pCurv[\schemeNormExt]{S,i}{x}{}^3
                        +
                        \landauBigo{\sdist{S}{x}^2}
                        \\
                        &=
                        \left(
                            \sum_{i=1}^2 
                            \pCurv[\hat]{S,i}{}{}^2
                            +
                            2
                            \phaseFieldParam
                            \fastVar
                            \pCurv[\hat]{S,i}{}{}^3
                            +
                            \landauBigo{\phaseFieldParam^2}
                        \right)
                        \concat
                        \interfCoords[S,\phaseFieldParam]{\NOARG}{\orthProj{S}{x}}.
                    \end{split}
                \end{equation}
                \begin{equation}
                    \label{item:formal asymptotics:expanding the L2 gradient of the canham-helfrich energy:%
                      mean curvature normal deriv}
                    \grad[x]{}{}{\extMeanCurv{}} \cdot \extNormal{x}
                    =
                    \extMeanCurv{x}^2 - 2 \extGaussCurv{x}.
                \end{equation}
                \begin{equation}
                    \label{item:formal asymptotics:expanding the L2 gradient of the canham-helfrich energy:%
                      mean curvature hessian normal part}
                    \grad[x]{}{2}{\extMeanCurv{}} 
                    \frobProd 
                    \extNormal{x} \tensorProd \extNormal{x}
                    =
                    2 \extMeanCurv{x}\left(\extMeanCurv{x}^2 - 3 \extGaussCurv{x}\right).
                \end{equation}
            \end{subequations}
        \end{lem}

    \subsection{Assumptions on the solution} 
        Typically, formal asymptotic theories rely on non-trivial properties on the solution of the 
        system under investigation, \eqref{equ:modelling:phase field:resulting PDEs} in our case. 
        A rigorous justification requires treatment of its own and is not in the scope of this work.
        We shall restrict ourselves to clearly formulating the properties we need in form of assumptions,
        and rather focus on the relation of the quantities of a solution of \eqref{equ:modelling:phase field:resulting PDEs} 
        that assure a sensible behaviour in the limit.
        These assumptions can serve as a hint what needs to be investigated when a mathematical proof is to be given.
        \begin{enumerate}
            \item \label{it:formal asymptotics:ass:existence of solutions:ex}
            For every $ \phaseFieldParam > 0 $ the system
            \eqref{equ:modelling:phase field:resulting PDEs}
            with boundary conditions
            \refBndCond,
            and initial data 
            $
                \membrPhaseField{0}{}
            $,
            $ \cortexPhaseField{0}{} $,
            $ \pfSpeciesOne{0}{} $,
            which converge in $ \sobolevHSet{2}{\domain}\times\sobolevHSet{2}{\domain}\times
            \sobolevHSet{1}{\domain} $ for $ \phaseFieldParam\searrow0 $ and form a recovery 
            sequence of $ \freeEnergyFunct{}{} $,
            has a classical solution
            $$
                \left(
                    \pfVelocity{}{},
                    \pfPressure{}{},
                    \membrPhaseField{}{},
                    \cortexPhaseField{}{},
                    \pfSpeciesOne{}{},
                    \pfSpeciesTwo{}{}
                \right)
            $$
            on $ \domain_\finTime = \solTimeInt\times\domain $
            for some time $ \finTime> 0 $ being independent of $ \phaseFieldParam $. 
            Throughout this work, we choose the mobilities of the phase field 
            to be a power of $ \phaseFieldParam $: $ \membrPhaseFieldMob{\membrPhaseFieldSym}
            = \cortexPhaseFieldMob{\cortexPhaseFieldSym} = \phaseFieldParam^{\alpha} $ for 
            $ \alpha \in \reals_{>0} $.
            \item \label{it:formale asymptotics:ass:existence of solutions:interf}
            Additionally, there shall be
            two-dimensional, orientable, smoothly evolving 
            manifolds 
            $$ 
                \interfOne{t} = \set{x\in\domain}{\membrPhaseField{t}{x}=0},
                \quad\quad
                \interfTwo{t} = \set{x\in\domain}{\cortexPhaseField{t}{x}=0},
            $$ 
            which both enclose open
            sets $ \enclIFVol{\interfOne{t}} $ and $ \enclIFVol{\interfTwo{t}} $.
            The corresponding outer domains are defined such that
            $ \extIFVol{\interfOne{t}} = \domain \setminus \enclIFVol{\interfOne{t}}
            \setminus \interfOne{t} $ and $ \extIFVol{\interfTwo{t}} = \domain \setminus
            \enclIFVol{\interfTwo{t}} \setminus \interfTwo{t} $.
            It shall hold,
            $ \lim_{\phaseFieldParam\searrow0} \membrPhaseField{0}{} = -1 $ pointwise on $ \enclIFVol{\interfOne{}} $
            and 
            $ \lim_{\phaseFieldParam\searrow0} \membrPhaseField{0}{} = 1 $ pointwise on $ \extIFVol{\interfOne{}} $,
            and analogously for $ \cortexPhaseFieldSym $ and $ \interfTwo{} $.
            \item \label{it:formal asymptotics:ass:existence of solutions:non-inters}
            For sufficiently small $ \finTime $, it shall hold $
            \interfOne{t} \cap \interfTwo{t} = \emptyset $ for all $ t \in [0,\finTime] $.
            \item
            The components of every classical solution to  \eqref{equ:modelling:phase field:resulting PDEs}
            shall have a regular asymptotic expansion 
            in every compact subset $ U $ of 
            $ \outerRegion = \domain \setminus \interfOne{t} \setminus \interfTwo{t} $,
            i.e.,
            for every
            $ 
                q \in \{
                    \membrPhaseField{}{},\cortexPhaseField{}{},
                    \pfVelocity{}{},\pfPressure{}{},
                    \pfSpeciesOne{}{},\pfSpeciesTwo{}{}
                \} 
            $,
            it holds
            \begin{equation}
                \label{equ:formal asymptotics:ass:boundary layer:exp}
                \restrFun{q}{U}(t,x) = 
                \sum_{i=0}^\maxExpInd \schemeOuterExp{q}_i(t,x) \phaseFieldParam^i + \landauSmallO{\phaseFieldParam^n},
                \quad\quad
            \end{equation}
            for some $\maxExpInd\in\nats_0 $.
            All $ \schemeOuterExp{q}_i $ shall be as smooth as $ q $.
            We call these series \emph{outer expansions of $ q $}.
            This implies that a boundary layer
            is to be expected at most at $ \interfOne{t}\cup\interfTwo{t}$.
            \item 
            If $ q = \membrPhaseFieldSym $, \eqref{equ:formal asymptotics:ass:boundary layer:exp}
            shall even hold for
            all $ U \Subset \outerRegion\cup \interfTwo{} $ and 
            if $ q = \cortexPhaseFieldSym $, for all $ U \Subset \outerRegion\cup \interfOne{} $. Thus,
            every phase field is expected to have only one boundary layer.
            \item \label{it:formal asymptotics:ass:time deriv species}
              The species densities' evolution is irrelevant outside the 
            diffuse layers around $ \interfOne{}_\finTime $, $ \interfTwo{}_\finTime $. 
            We thus consider them to be asymptotically constant in time away from the diffuse layers:
            For every $ U \Subset \outerRegion $, it holds
            $ 
                \pDiff{t}{}{\restrFun{\pfSpeciesOne{}{}}{U}} \in 
                \landauBigo{\phaseFieldParam^2}, 
            $
            which is equivalent to claiming
            $ \pDiff{t}{}{}{ {\pfSpeciesOneOuter{}{}}_0} = 0 = \pDiff{t}{}{} {\pfSpeciesOneOuter{}{}}_1 $.
            \item \label{it:formal asymptotics:inner expansion struct:general}
            The components of every classical solution to  \eqref{equ:modelling:phase field:resulting PDEs}
            shall have a regular asymptotic expansion
            in $ \tubNeighbourhood{\ifCoordsDelta}{S} $, 
            $ S\in\{\interfOne{t}, \interfTwo{t}\} $,
            after transformation into local coordinates:
            For all
            $ 
                q 
                \in \{
                    \membrPhaseField{}{},\cortexPhaseField{}{},
                    \pfVelocity{}{},
                    \pfPressure{}{},
                    \pfSpeciesOne{}{},\pfSpeciesTwo{}{}
                \} 
            $, it holds
            $
                    \restrFun{q}{\interfOne{t}\cup\interfTwo{t}}
                    =
                    \schemeLocalFun{q}\concat\interfCoords[\phaseFieldParam]{}{}
            $ such that
            \begin{equation*}
                \begin{split}
                    \schemeLocalFun{q}(t,\projVar,\fastVar) 
                    = 
                    \sum_{k=\minExpInd}^\maxExpInd
                    \phaseFieldParam^k 
                    \schemeInnerExp{\schemeLocalFun{q}}_k(t,\projVar,\fastVar)
                    +
                    \landauSmallO{\phaseFieldParam^\maxExpInd}
                \end{split}
            \end{equation*}
            for $ \absMinExpInd, \maxExpInd\in\nats_0 $, where all
            $ 
                \schemeInnerExp{\schemeLocalFun{q}}_k
            $ shall be integrable in $ \fastVar $ and as smooth as $ q $.
            We call these series \emph{inner expansions of $ q $}.
            \item \label{it:formal asymptotics:inner expansion struct:cond for phase fields}
            Physically, the phase fields model the volume fraction of phases. Thus, 
            they should always take values between $ -1 $ and $ 1 $, independent of how small $ \phaseFieldParam $
            may be. Hence, for $ q \in\{ \membrPhaseFieldSym, \cortexPhaseFieldSym \} $, we assume
            $ \schemeInnerExp{\schemeLocalFun{q}}_\ell = 0 $ for all $ \ell \in \{\minExpInd,\dots,-1\} $.
            \item For the species density $ \pfSpeciesOne{}{} $, we additionally require 
            that blow-ups are of order at most $ -1 $, i.e.,
            $ \schemeInnerExp{\schemeLocalFun{\pfSpeciesOne{}{}}_\ell} = 0 $ for 
            all $ \ell \in \{\minExpInd,\dots,-2\} $.
            The reason why we cannot naturally expect boundedness here is that $ \pfSpeciesOne{}{} $ does
            not give the volume fraction, but the number of particles per volume of the active linkers.
        \end{enumerate}
        We will often have to compute differential operators of functions that are expressed in interfacial
        coordinates:
        \begin{rem}
            \label{lemma:formal asymptotics:transformation diff op}
            For a sufficiently smooth function
            $ 
                \funSig{\scalarFunProto}{
                    \mathcal{S}_\finTime
                    \times\reals
                }{\reals} 
            $ on an evolving manifold $ \mathcal{S}_\finTime = \bigcup_{t\in[0,\finTime]} \{t\}\times
            S_t $, 
            and $ t^*\in[0,\finTime] $, $ x^*\in\domain $, 
            it holds
            \begin{align}
                \label{equ:appendix:formally matched asymptotics:%
                gradient rescaled distance}
                \grad{x}{}{ \scalarFunProto 
                \concat 
                \interfCoords[\mathcal{S}_\finTime,\phaseFieldParam]{}{} }(t^*,x^*)
                &=
                \phaseFieldParam^{-1}
                \simpleDeriv{\prime}{
                    \scalarFunProto 
                }
                (\interfCoords[\mathcal{S}_\finTime,\phaseFieldParam]{t^*}{x^*})
                \extNormal{t^*,x^*}
                +
                \grad{
                    S(t^*)_{\sdist{}{x^*}}
                }{}{}
                \scalarFunProto 
                (\interfCoords[\mathcal{S}_\finTime,\phaseFieldParam]{t^*}{x^*}),
                \\
                \label{equ:appendix:formally matched asymptotics:%
                laplacian rescaled distance}
                \laplacian{x}{}{ 
                    \scalarFunProto \concat \interfCoords[\mathcal{S}_\finTime,\phaseFieldParam]{}{} 
                } (t^*,x^*)
                &=
                \phaseFieldParam^{-2}
                \simpleDeriv{\prime\prime}{
                    \scalarFunProto
                }
                (\interfCoords[\mathcal{S}_\finTime,\phaseFieldParam]{t^*}{x^*})
                -
                \phaseFieldParam^{-1}
                \simpleDeriv{\prime}{
                    \scalarFunProto
                }
                (\interfCoords[\mathcal{S}_\finTime,\phaseFieldParam]{t^*}{x^*})
                \extMeanCurv{t^*,x^*}
                \\\nonumber
                &+
                \laplacian{S(t^*)_{\sdist{}{x^*}}}{}{}
                \scalarFunProto 
                (\interfCoords[\mathcal{S}_\finTime,\phaseFieldParam]{t^*}{x^*}),
                \\
                \label{equ:formally matched asymptotics:time derivative rescaled distance}
                \pDiff{t}{}{q\concat\interfCoords[\mathcal{S}_\finTime,\phaseFieldParam]{}{}}(t^*,x^*)
                &=
                -
                \phaseFieldParam^{-1}
                \shapeTransfVelNormal^{S}(t^*,x^*)
                \simpleDeriv{\prime}{\scalarFunProto}
                (\interfCoords[\mathcal{S}_\finTime,\phaseFieldParam]{t^*}{x^*})
                +
                \pDiff{t}{}{}\scalarFunProto(\interfCoords[\mathcal{S}_\finTime,\phaseFieldParam]{t^*}{x^*}).
            \end{align}
            For $ \funSig{\vecFunProto}{\mathcal{S}_\finTime\times\reals}{\reals^n} $,
            it holds
            \begin{align}
                \label{equ:formally matched asymptotics:div rescaled distance}
                \diver{x}{}{
                    \vecFunProto\concat\interfCoords[\mathcal{S}_\finTime,\phaseFieldParam]{}{}
                }(t^*,x^*)
                &=
                \phaseFieldParam^{-1}\simpleDeriv{\prime}{\vecFunProto}
                (\interfCoords[\mathcal{S}_\finTime,\phaseFieldParam]{t^*}{x^*})
                \cdot
                \extNormal{t^*,x^*}
                +
                \diver{S(t^*)_{\sdist{}{x^*}}}{}{}\vecFunProto
                (\interfCoords[\mathcal{S}_\finTime,\phaseFieldParam]{t^*}{x^*}),
                \\
                \label{equ:formally matched asymptotics:jac rescaled distance}
                \grad{x}{}{\vecFunProto\concat\interfCoords[\mathcal{S}_\finTime,\phaseFieldParam]{}{}}
                (t^*,x^*)
                &=
                \phaseFieldParam^{-1}
                \simpleDeriv{\prime}{\vecFunProto}
                (\interfCoords[\mathcal{S}_\finTime,\phaseFieldParam]{t^*}{x^*})
                \tensorProd
                \extNormal{t^*,x^*}            
                +
                \grad{S(t^*)_{\sdist{}{x^*}}}{}{}
                \vecFunProto
                (\interfCoords[\mathcal{S}_\finTime,\phaseFieldParam]{t^*}{x^*}),
                \\
                \label{equ:formally matched asymptotics:laplacian vec fun rescaled distance}
                \laplacian{x}{}{\vecFunProto\concat\interfCoords[\mathcal{S}_\finTime,\phaseFieldParam]{}{}}
                (t^*,x^*)
                &=
                \phaseFieldParam^{-2}
                \simpleDeriv{\prime\prime}{\vecFunProto}
                (\interfCoords[\mathcal{S}_\finTime,\phaseFieldParam]{t^*}{x^*})
                -
                \phaseFieldParam^{-1}\simpleDeriv{\prime}{(\vecFunProto)}
                \concat\interfCoords[\mathcal{S}_\finTime,\phaseFieldParam]{t^*}{x^*}
                \extMeanCurv{t^*,x^*}
                \\\nonumber
                &+
                \laplacian{S(t^*)_{\sdist{}{x^*}}}{}{}
                \vecFunProto(\interfCoords[\mathcal{S}_\finTime,\phaseFieldParam]{t^*}{x^*}).
            \end{align}
            For $ \funSig{\matFunProto}{\mathcal{S}_\finTime\times\reals}{\reals^{(n,n)}} $, it holds
            \begin{align}
                \label{equ:formally matched asymptotics:div mat rescaled distance}
                \diver{x}{}{
                    \matFunProto\concat\interfCoords[\mathcal{S}_\finTime,\phaseFieldParam]{}{}
                }(t^*,x^*)
                =
                \phaseFieldParam^{-1}
                \simpleDeriv{\prime}{\matFunProto}
                (\interfCoords[\mathcal{S}_\finTime,\phaseFieldParam]{t^*}{x^*})
                \extNormal{t^*,x^*}
                +
                \diver{S(t^*)_{\sdist{}{x^*}}}{}{}
                \matFunProto
                (\interfCoords[\mathcal{S}_\finTime,\phaseFieldParam]{t^*}{x^*}).
            \end{align}
        \end{rem}
        Let us further exercise some smaller expansions.
        \begin{lem}
            \label{lemma:formal asymptotics:expansion of the L2 gradient of the coupling energy:%
            helpful expansion one}
            For $ \phaseFieldProto \in \{\cortexPhaseFieldSym,\membrPhaseFieldSym\} $
            the following expansions hold
            \begin{align}
                \label{equ:formal asymptotics:transformation interf coords:phase field grad}
                \abs{\grad{}{}{}\phaseFieldProto}
                &=
                \normal{}\cdot\grad{}{}{}\phaseFieldProto
                =
                \left(
                \phaseFieldParam^{-1}\simpleDeriv{\prime}{\schemeLocalFun{\phaseFieldProto}_0}
                +
                \simpleDeriv{\prime}{\schemeLocalFun{\phaseFieldProto}_1}
                +
                \phaseFieldParam
                \simpleDeriv{\prime}{\schemeLocalFun{\phaseFieldProto}_2} 
                \right)
                \concat\interfCoords[\phaseFieldParam]{}{}
                +
                \landauBigo{\phaseFieldParam^2},
                \\
                \doubleWellPot{\phaseFieldProto}
                &=
                \doubleWellPot{\phaseFieldProto_0}
                +
                \phaseFieldParam
                \doubleWellPot[\schemeSupscr{\prime}]{\phaseFieldProto_0}
                \phaseFieldProto_1
                +
                \phaseFieldParam^2
                \left(
                    \doubleWellPot[\schemeSupscr{\prime}]{\phaseFieldProto_0}
                    \phaseFieldProto_2
                    +
                    \doubleWellPot[\schemeSupscr{\prime\prime}]{\phaseFieldProto_0}
                    \phaseFieldProto_1^2
                \right)
                +
                \landauBigo{\phaseFieldParam^3},
                \\
                \label{equ:formal asymptotics:expansion of the L2 gradient of the coupling energy:%
                helping expansions:ginzburg landau density}
                \ginzburgLandauEnergyDens{\phaseFieldProto}{}
                &=
                \phaseFieldParam^{-1}
                \left(
                    \frac{1}{2}
                    \left(\simpleDeriv{\prime}{\schemeLocalFun{\phaseFieldProto}_0}\right)^2
                    \concat
                    \interfCoords[\phaseFieldParam]{}{}
                    +
                    \doubleWellPot{\phaseFieldProto_0}
                \right)
                +
                2
                (\simpleDeriv{\prime}{\schemeLocalFun{\phaseFieldProto}_0}
                \simpleDeriv{\prime}{\schemeLocalFun{\phaseFieldProto}_1}
                )
                \concat
                \interfCoords[\phaseFieldParam]{}{}
                +
                \doubleWellPot[\schemeSupscr{\prime}]{\phaseFieldProto_0}
                \phaseFieldProto_1
                +
                \landauBigo{\phaseFieldParam},
                \\
                \label{equ:formal asymptotics:expansion of the L2 gradient of the coupling energy:%
                helping expansions:phase field inv grad mod}
                \abs{\grad{}{}{}\phaseFieldProto}^{-1}
                &=
                \phaseFieldParam
                (\simpleDeriv{\prime}{\schemeLocalFun{\phaseFieldProto}_0})^{-1}
                \concat
                \interfCoords[\phaseFieldParam]{}{}
                +
                \phaseFieldParam^2
                \frac{
                    \simpleDeriv{\prime}{\schemeLocalFun{\phaseFieldProto}_1}
                    \concat
                    \interfCoords[\phaseFieldParam]{}{}
                }{
                    \left(\simpleDeriv{\prime}{\schemeLocalFun{\phaseFieldProto}_0}\right)^2
                    \concat
                    \interfCoords[\phaseFieldParam]{}{}
                }
                +
                \landauBigo{\phaseFieldParam^3}.
            \end{align}
            If $ \phaseFieldProto$ is the optimal profile at leading order, i.e.,
            $
                \simpleDeriv{\prime\prime}{\schemeLocalFun{\phaseFieldProto}_0}\concat
                \interfCoords[\phaseFieldParam]{}{}
                - 
                \doubleWellPot[\schemeSupscr{\prime}]{\phaseFieldProto_0}
                =
                0,
            $
            we further have
            \begin{equation}
                \label{equ:formal asymptotics:expansion of the L2 gradient of the coupling energy:%
                helping expansions:phase field mean curv}
                \meanCurv{\phaseFieldProto}{}
                =
                \phaseFieldParam^{-1}
                \simpleDeriv{\prime}{\schemeLocalFun{\phaseFieldProto}_0}\concat\interfCoords[\phaseFieldParam]{}{}
                \left(
                    \simpleDeriv{\prime}{\schemeLocalFun{\phaseFieldProto}_0}
                    \concat \interfCoords[\phaseFieldParam]{}{}
                    \extMeanCurv{}{}
                    +
                    \simpleDeriv{\prime\prime}{\schemeLocalFun{\phaseFieldProto}_1}
                    \concat \interfCoords[\phaseFieldParam]{}{}
                    +
                    \doubleWellPot[\schemeSupscr{\prime\prime}]{\phaseFieldProto_0}
                    \phaseFieldProto_1
                \right)
                +
                \landauBigo{1}.
            \end{equation}
        \end{lem}
    \begin{proof}
        Ad~\eqref{equ:formal asymptotics:expansion of the L2 gradient of the coupling energy:%
        helping expansions:ginzburg landau density}:
        We use \eqref{equ:formal asymptotics:transformation interf coords:phase field grad}
        to compute
        \begin{equation*}
            \ginzburgLandauEnergyDens{\rho}{}
            =
            \phaseFieldParam^{-1}
            \left(
                \frac{1}{2}
                \left(\simpleDeriv{\prime}{\schemeLocalFun{\rho}_0}\right)^2
                \concat
                \interfCoords[\phaseFieldParam]{}{}
                +
                \doubleWellPot{\rho_0}
            \right)
            +
            2
            (\simpleDeriv{\prime}{\schemeLocalFun{\rho}_0}
            \simpleDeriv{\prime}{\schemeLocalFun{\rho}_1}
            )
            \concat
            \interfCoords{}{}
            +
            \doubleWellPot[\schemeSupscr{\prime}]{\rho_0}
            \rho_1
            +
            \landauBigo{\phaseFieldParam}.
        \end{equation*}

        Ad~\eqref{equ:formal asymptotics:expansion of the L2 gradient of the coupling energy:%
        helping expansions:phase field inv grad mod}:
        Observe,
        \begin{equation}
            \label{equ:formal asymptotics:expanding L2 grad coupling energy cortex:inv grad phase field cortex}
            \abs{\grad{}{}{\rho_0+\phaseFieldParam r_1}}^{-1}
            =
            \abs{\grad{}{}{}\rho_0}^{-1}
            -
            \phaseFieldParam
            \frac{
                \grad{}{}{}\rho_{0}
                \cdot
                \grad{}{}{}r_1
            }{\abs{\grad{}{}{}\rho_0}^3}
            +
            \landauBigo{\phaseFieldParam^3}.
        \end{equation}
        Note further that
        $
            \abs{\grad{}{}{}\rho_0}^{-1}
            =
            (
                \grad{}{}{}\rho_0
                \cdot
                \normal{}
            )^{-1}
            =
            (
                \phaseFieldParam^{-1}
                \simpleDeriv{\prime}{\schemeLocalFun{\rho}_0}
                \concat
                \interfCoords[\phaseFieldParam]{}{}
            )^{-1}
            =
            \phaseFieldParam
            (
                \simpleDeriv{\prime}{\schemeLocalFun{\rho}_0}
                \concat
                \interfCoords[\phaseFieldParam]{}{}
            )^{-1}            
        $
        and
        $ 
            \grad{}{}{}r_1 
            = 
            \phaseFieldParam^{-1}
            \simpleDeriv{\prime}{\schemeLocalFun{\rho}_1}\concat\interfCoords[\phaseFieldParam]{}{}
            \normal{}
            +
            \landauBigo{1}
        $
        so that
        \begin{equation*}
            \phaseFieldParam
            \frac{
                \grad{}{}{}\rho_0
                \cdot
                \grad{}{}{}r_1
            }{\abs{\grad{}{}{}\rho_0}^3}
            =
            \frac{
                \phaseFieldParam^{-1}
                \simpleDeriv{\prime}{\schemeLocalFun{\rho}_0}\concat\interfCoords[\phaseFieldParam]{}{}
                \simpleDeriv{\prime}{\schemeLocalFun{\rho}_1}\concat\interfCoords[\phaseFieldParam]{}{}
                +
                \landauBigo{1}
            }{\abs{\grad{}{}{}\rho_0}^3}                
            \in
            \landauBigo{\phaseFieldParam^2}.
        \end{equation*}

        Ad~\eqref{equ:formal asymptotics:expansion of the L2 gradient of the coupling energy:%
        helping expansions:phase field mean curv}:
        Expand
        \begin{equation*}
            \begin{split}
                \meanCurv{\rho}{}
                &=
                \abs{\grad{}{}{}\rho}(
                    -\phaseFieldParam\laplacian{}{}{}\rho
                    +
                    \phaseFieldParam^{-1}\doubleWellPot[\schemeSupscr{\prime}]{\rho}
                )
                \\
                &\overset{(1)}{=}
                (
                    \phaseFieldParam^{-1}\simpleDeriv{\prime}{\schemeLocalFun{\rho}_0}
                    \concat \interfCoords[\phaseFieldParam]{}{}
                    +
                    \simpleDeriv{\prime}{\schemeLocalFun{\rho}_1}
                    \concat \interfCoords[\phaseFieldParam]{}{}
                    +
                    \landauBigo{\phaseFieldParam}
                )
                (
                    -\phaseFieldParam^{-1}
                    \simpleDeriv{\prime\prime}{\schemeLocalFun{\rho}_0}
                    \concat \interfCoords[\phaseFieldParam]{}{}
                    +
                    \simpleDeriv{\prime}{\schemeLocalFun{\rho}_0}
                    \concat \interfCoords[\phaseFieldParam]{}{}
                    \extMeanCurv{}{}
                    +
                    \simpleDeriv{\prime\prime}{\schemeLocalFun{\rho}_1}
                    \concat \interfCoords[\phaseFieldParam]{}{}
                    +
                    \phaseFieldParam^{-1}\doubleWellPot[\schemeSupscr{\prime}]{\rho_0}
                    +
                    \doubleWellPot[\schemeSupscr{\prime\prime}]{\schemeLocalFun{\rho}_0}
                    \rho_1
                    +
                    \landauBigo{\phaseFieldParam}
                )
                \\
                &\overset{(2)}{=}
                \phaseFieldParam^{-1}
                \simpleDeriv{\prime}{\schemeLocalFun{\rho}_0}\concat\interfCoords[\phaseFieldParam]{}{}
                \left(
                    \simpleDeriv{\prime}{\schemeLocalFun{\rho}_0}
                    \concat \interfCoords[\phaseFieldParam]{}{}
                    \extMeanCurv{}{}
                    +
                    \simpleDeriv{\prime\prime}{\schemeLocalFun{\rho}_1}
                    \concat \interfCoords[\phaseFieldParam]{}{}
                    +
                    \doubleWellPot[\schemeSupscr{\prime\prime}]{\rho_0}
                    \rho_1
                \right)
                +\landauBigo{1},
            \end{split}
        \end{equation*}
        where for $ (1) $ we employ
        \eqref{equ:appendix:formally matched asymptotics:laplacian rescaled distance},
        and for $ (2) $ the optimal profile equation.
    \end{proof}
    A common principle, which we will make use of in the following multiple times, is summarised in the following
    \begin{lem}
        \label{lemma:addendum formal asymptotics:concentration}
        Let $ \Gamma \subseteq \domain $ be a smooth hypersurface.
        Let $ p \in \lebesgueSet{1}{\reals} $ with
        $$
            \sup_{\abs{t}>s} \abs{p(t)t} \leq \frac{C}{s^m}
        $$
        for some $ C \in [0,\infty) $ and $ m \in (0,\infty) $,
        $ f_\phaseFieldParam \in \diffSet{}{\domain} $,
        and for all sequences $ x_\phaseFieldParam  \specConv{\phaseFieldParam\to 0} x $, it holds
        $ f_\phaseFieldParam(x_\phaseFieldParam) \specConv{\phaseFieldParam\to 0}
        f(x) $ with $ \norm[\lebesgueSet{\infty}{\domain}]{f_\phaseFieldParam } < M $ for some $ M \in (0,\infty) $
        being independent of $\phaseFieldParam $. Then,
        $$
            \phaseFieldParam^{-1}
            \integral{\domain}{}{
                p\left(
                    \frac{\sdist{\Gamma}{x}}{\phaseFieldParam}
                \right)
                f_\phaseFieldParam(x)
            }{\lebesgueM{3}(x)}
            \specConv{\phaseFieldParam\to 0}
            \integral{-\infty}{\infty}{
                p(s)
            }{\lebesgueM{1}(s)}
            \integral{\Gamma}{}{
                f(x)
            }{\hausdorffM{2}(x)}.
        $$
    \end{lem}
    After the preliminaries are fixed, we shall proceed by analysing the asymptotic behaviour of the solution
    of \eqref{equ:modelling:phase field:resulting PDEs}.
    \subsection{Outer expansion}
        \label{sec:formal asymptotics:outer expansion:phase fields}
        We start with investigating the solutions' behaviour away from the boundary layer, i.e. on a set $ \domain_\delta =
        \domain \setminus \left( \tubNeighbourhood{\delta}{\interfOne{}{}} \cup \tubNeighbourhood{\delta}{\interfTwo{}{}} \right) $ for
        some $ \delta > 0 $.
        Let $ \phaseFieldPT \in \{\membrPhaseField{}{}, \cortexPhaseField{}{}\} $ for the following 
        considerations.

        Due to the recovery sequence property of the initial data
        postulated in \refAss{it:formal asymptotics:ass:existence of solutions:ex},
        $$
            \freeEnergyFunct{
                \pfVelocity{0}{},
                \membrPhaseField{0}{},
                \cortexPhaseField{0}{},
                \pfSpeciesOne{0}{}
            }{\phaseFieldParam}
            \in \landauBigo{1}.
        $$
        Further, the 
        sufficiently fast decay of the species densities' time derivative, see
        \refAss{it:formal asymptotics:ass:time deriv species} imply
        $$
            \integral{\domain_\delta}{}{
                \ginzburgLandauEnergyDens{\membrPhaseFieldSym}{}(y)
                \integral{\domain_\delta}{}{
                    \ginzburgLandauEnergyDens{\cortexPhaseFieldSym}{}(x)
                    \pDiff{\pfSpeciesOne{}{}}{}{}
                    \pfCouplingEnergyDens{}{}[
                    \pDiff{t}{}{}\pfSpeciesOne{}{}]
                }{\lebesgueM{3}(x)}
            }{\lebesgueM{3}(y)}
            \in\landauBigo{1}.
        $$
        From \eqref{equ:modelling:phase field resulting PDEs:first interface evolution}
        and \eqref{equ:modelling:phase field:resulting PDEs:second interface evolution},
        we also obtain
        $$ 
            \phaseFieldParam^{\alpha}
            \laplacian{}{}{
                \grad{\phaseFieldProto}{L^2}{}
                \pfSurfEnergyOne{\phaseFieldProto}
                +
                \grad{\phaseFieldProto}{L^2}{}
                \pfCouplingEnergy{}{}
            }
            \in \landauBigo{1},
        $$
        so for $ \alpha < 1 $, 
        $ 
            \grad{\phaseFieldProto}{L^2}{}
            \pfSurfEnergyOne{\phaseFieldProto}
            +
            \grad{\phaseFieldProto}{L^2}{}
            \pfCouplingEnergy{}{} \in \landauBigo{1}.
        $ (Bringing $ \phaseFieldParam^{\alpha} $ to the right,
        all leading order terms of 
        $\laplacian{}{}{
            \grad{\phaseFieldProto}{L^2}{}
            \pfSurfEnergyOne{\phaseFieldProto}
            +
            \grad{\phaseFieldProto}{L^2}{}
            \pfCouplingEnergy{}{}
        }$
        from order $ -3 $ to $ -1 $ have no match on the right hand side and thus have to be zero
        following the separation of scales argument. 
        Using that the Neumann boundary conditions 
        \eqref{equ:modelling:phase field:resulting PDEs:boundary conditions:%
        fluxes} do not depend on $ \phaseFieldParam $, we can thus conclude that all these terms are of 
        order zero.)
        Comparing the right hand side of 
        \eqref{equ:modelling:phase field resulting PDEs:momentum balance} with its left hand side, 
        we conclude
        \begin{equation*}
            \begin{split}
            \integral{\domain_\delta}{}{
                \meanCurv{\cortexPhaseFieldSym}{x}
                \pfSpeciesOne{t}{x}
                \symVelocityNormal{\cortexPhaseFieldSym}(t,x)
                \integral{\domain_\delta}{}{
                    \ginzburgLandauEnergyDens{\membrPhaseFieldSym}{}(y)
                    \pDiff{\pfSpeciesOne{}{}}{}{}
                    \pfCouplingEnergyDens{x,y,\pfSpeciesOne{}{},\normal[\cortexPhaseFieldSym]{}}{}
                }{\lebesgueM{3}(y)}
            }{\lebesgueM{3}(x)}
            +
            \\
            \integral{\domain_\delta}{}{
            \integral{\domain_\delta}{}{
                \ginzburgLandauEnergyDens{\membrPhaseField{}{}}{}(y)
                \grad{}{}{\pDiff{\pfSpeciesOne{}{}}{}{} \pfCouplingEnergyDens{}{}
                (\cdot,y,\pfSpeciesOne{}{},\normal[\cortexPhaseFieldSym]{})
                }
                \cdot
                \symVelocityTangVec
                \ginzburgLandauEnergyDens{\cortexPhaseFieldSym}{}
                \pfSpeciesOne{}{}
              }{\lebesgueM{3}(y)}        
              }{\lebesgueM{3}(x)}
            \in\landauBigo{1}.
            \end{split}
        \end{equation*}

        It thus follows from \eqref{equ:formal asymptotics:energy inequ:pre} that
        $$
            \esssup_{t\in\solTimeInt} 
            \freeEnergyFunct{
                \membrPhaseField{}{},
                \cortexPhaseField{}{},
                \speciesOne{}{}
            }{\phaseFieldParam} \in \landauBigo{1}.
        $$
        Therefore, 
        $ 
            \integral{\domain_\delta}{}{
                \frac{\phaseFieldParam}{2}
                \abs{\grad{}{}{}\rho}^2
                +
                \phaseFieldParam^{-1}
                \doubleWellPot{\rho}
            }{\lebesgueM{3}}
        \in \landauBigo{1} $
        for all $ t \in [0,\finTime] $,
        and we conclude
        $
            \doubleWellPot{\phaseFieldProto}
            \in\landauBigo{\phaseFieldParam}
        $.
        Inserting the outer expansion of $ \phaseFieldProto $ into $ \doubleWellPot{\phaseFieldProto} $
        brings
        \begin{equation*}
            \doubleWellPot{\phaseFieldProto} = (\left(\schemeOuterExp{\phaseFieldProto}_0\right)^2-1)^2 
            + \landauBigO{\phaseFieldParam}.
        \end{equation*}
        Hence, it must hold 
        $$
           \restrFun{(\left(\schemeOuterExp{\phaseFieldProto}_0\right)^2-1)^2}{\domain_\ifCoordsDelta} = 0
        $$
        for any $ \ifCoordsDelta > 0 $. This further implies
        $
            \restrFun{\schemeOuterExp{\phaseFieldProto}_0}{\domain_\ifCoordsDelta}(t,\cdot) \in \{-1,1\}
        $
        for all $ t\in[0,\finTime] $.
        For the initial data of $ \phaseFieldProto $, we have 
        (cf.~\refAss{it:formale asymptotics:ass:existence of solutions:interf})
        $ \schemeOuterExp{\phaseFieldProto}_0(0,\cdot) = -1 $ in $ \enclIFVol{S} $ and
        $ 1 $ in $ \extIFVol{S} $, $ S \in \{\interfOne{},\interfTwo{}\} $,
        so that we can argue by continuity in time that 
        \begin{equation}
            \restrFun{\schemeOuterExp{\phaseFieldProto}_0}{\enclIFVol{S}
            \setminus \tubNeighbourhood{\ifCoordsDelta}{S}}(t,\cdot) = -1 
            \;\text{and}\;
            \restrFun{\schemeOuterExp{\phaseFieldProto}_0}{\extIFVol{S}\setminus
            \tubNeighbourhood{\ifCoordsDelta}{S}}(t,\cdot) = 1
            \;\text{for all}\; t \in [0,\finTime],
        \end{equation}
        which is the essential result of this paragraph.
	\subsection{Inner expansion}
        As there is no danger of confusion, we drop the subscript $ \phaseFieldParam $ on the physical quantities.
        \undef\indicatePhaseField
    	Let us first not that the result of the previous paragraph can be combined with the principle of 
        asymtptotic matching on the phase fields such that we obtain
        \begin{equation}
            \label{equ:formal asymptotics:lim matching principle:membr phase field:outer}
            \begin{split}
            \left(
            \lim_{\fastVar\nearrow\infty}\schemeInnerExp{\schemeLocalFun{\membrPhaseFieldSym}}_0
            (\cdot,\cdot,\fastVar)
            \right)
            \concat\interfCoords[\phaseFieldParam]{t}{x}
            &=
            \lim_{ \substack{ x\to\interfOne{} \\ x\in\extIFVol{\interfOne{}} } }
            \schemeOuterExp{\membrPhaseFieldSym}_0(t,x) = 1
            \end{split}
        \end{equation}
        for $ x \in \tubNeighbourhood{\ifCoordsDelta}{\interfOne{}} \cap \extIFVol{\interfOne{}} $,
        i.e., $ \sdist{\interfOne{}}{x} > 0 $. Analogously,
        \begin{equation}
            \label{equ:formal asymptotics:lim matching principle:membr phase field:inner}
            \begin{split}
            \left(
            \lim_{\fastVar\searrow-\infty}\schemeInnerExp{\schemeLocalFun{\membrPhaseFieldSym}}_0
            (\cdot,\cdot,\fastVar)
            \right)
            \concat\interfCoords[\phaseFieldParam]{t}{x}
            &=
            \lim_{ \substack{ x\to\interfOne{} \\ x\in\enclIFVol{\interfOne{}} } }
            \schemeOuterExp{\membrPhaseFieldSym}_0(t,x) = -1
            \end{split}
        \end{equation}
        for $ x \in \tubNeighbourhood{\ifCoordsDelta}{\interfOne{}} \cap \enclIFVol{\interfOne{}} $
        and mutatis mutandis for $ \cortexPhaseFieldSym $.

        \begin{rem}
          	An immediate consequence of the matching principle and the assumption that
            $ \schemeOuterExp{q}_{\ell} = 0 $ for all $ \ell\in\{\minExpInd,\dots,-1\} $ of the outer expansion
            is
            \begin{equation}
                \label{equ:formal asymptotics:inner expansion:negative order z limit}
                \lim_{z\to\pm\infty}\schemeLocalFun{q}_{\ell} = 0 \quad\quad \text{for\;all\;}\ell\in\{\minExpInd,\dots,-1\}
            \end{equation}
            of the inner expansion. This also holds for all derivatives as long as they exist.
        \end{rem}
    \subsection{Properties of $ \schemeLocalFun{\pfVelocity{}{}} $ and 
      	$ \schemeLocalFun{\pfPressure{}{}} $ to leading order}
    	\label{sec:formal asymptotics:inner expansion:properties of vel to leading order}
    	Let $ S \in \{\interfOne{},\interfTwo{}\} $. To obtain insight on the higher 
        order coefficients in the expansion of the velocity and the pressure, we 
        exploit the structure of the Navier-Stokes equations 
        \eqref{equ:modelling:phase field resulting PDEs:momentum balance},
        \eqref{equ:modelling:phase field resulting PDEs:incompressibility}
        following \cite[p.~486, Section~A.1.2]{AL2018}.
        
        Due to
        \refAss{it:formal asymptotics:inner expansion struct:cond for phase fields},
        $ \grad{\cortexPhaseFieldSym}{L^2}{}\pfGenWillmoreEnergyFunct{}{} \in 
        \landauBigo{\phaseFieldParam^{-3}} $.
        Thus, for $ N \geq 3 $, we have from 
        \eqref{equ:modelling:phase field resulting PDEs:momentum balance},
        at order $ \phaseFieldParam^{\minExpInd-2} $,
        \begin{equation*}
        	-\viscosity
            \simpleDeriv{\prime\prime}{\schemeLocalFun{\pfVelocity{}{}}_{\minExpInd}}
            \concat\interfCoords[\phaseFieldParam]{}{} 
            =
            0.
        \end{equation*}
        With \eqref{equ:formal asymptotics:inner expansion:negative order z limit}, it
        further follows
        $ \simpleDeriv{\prime}{\schemeLocalFun{\pfVelocity{}{}}_{\minExpInd}} = 0 $. 
        From
        $ \simpleDeriv{\prime}{\schemeLocalFun{\pfVelocity{}{}}_{\minExpInd}} = 0 $
        with \eqref{equ:formal asymptotics:inner expansion:negative order z limit}, we conclude 
        analogously
        $ \schemeLocalFun{\pfVelocity{}{}}_{\minExpInd} = 0 $.        
        
        At order $ \phaseFieldParam^{\minExpInd-1} $, the equation is
        \begin{equation}
            \label{equ:formal asymptotics:properties velocity leading order:navier stokes:%
            new leading order}
            -\viscosity
        	\simpleDeriv{\prime\prime}{\schemeLocalFun{\pfVelocity{}{}}_{\minExpInd+1}}
            \concat\interfCoords[\phaseFieldParam]{}{} 
            +
            \simpleDeriv{\prime}{\schemeLocalFun{\pfPressure{}{}}_{\minExpInd}}
            \concat\interfCoords[\phaseFieldParam]{}{} 
            \extNormal{}
            =
            0.
        \end{equation}
        From \eqref{equ:modelling:phase field resulting PDEs:incompressibility} we have,
        using 
        Remark~\ref{lemma:formal asymptotics:transformation diff op}%
        \eqref{equ:formally matched asymptotics:div rescaled distance}, to leading order
        $ \phaseFieldParam^{\minExpInd} $:
        \begin{equation}
            \label{equ:formal asymptotics:properties velocity leading order:norm comp deriv}
        	\simpleDeriv{\prime}{\schemeLocalFun{\pfVelocity{}{}}_{\minExpInd+1}}
            \concat \interfCoords[\phaseFieldParam]{}{}
            \cdot \extNormal{}
            = 0.
        \end{equation}
        Multiplying 
        \eqref{equ:formal asymptotics:properties velocity leading order:navier stokes:%
        new leading order}
        by $ \extNormal{}$, we find with 
        \eqref{equ:formal asymptotics:properties velocity leading order:norm comp deriv}
        \begin{equation}
            \label{equ:formal asymptotics:properties velocity leading order:pressure deriv zero}
            \simpleDeriv{\prime}{\schemeLocalFun{\pfPressure{}{}}_{\minExpInd}}
            \concat\interfCoords[\phaseFieldParam]{}{} 
            =
            0.
        \end{equation}
        In turn, inserting \eqref{equ:formal asymptotics:properties velocity leading order:%
        pressure deriv zero}
        back into \eqref{equ:formal asymptotics:properties velocity leading order:navier stokes:%
        new leading order},
        we obtain 
        $ \simpleDeriv{\prime\prime}{\schemeLocalFun{\pfVelocity{}{}}_{\minExpInd+1}} = 0 $, and 
        with \eqref{equ:formal asymptotics:inner expansion:negative order z limit}
        further
        $ \simpleDeriv{\prime}{\schemeLocalFun{\pfVelocity{}{}}_{\minExpInd+1}} = 0 $. 
        From
        $ \simpleDeriv{\prime}{\schemeLocalFun{\pfVelocity{}{}}_{\minExpInd+1}} = 0 $
        with \eqref{equ:formal asymptotics:inner expansion:negative order z limit}, we conclude 
        analogously
        $ \schemeLocalFun{\pfVelocity{}{}}_{\minExpInd+1} = 0 $. Arguing verbatim with
        \eqref{equ:formal asymptotics:inner expansion:negative order z limit},
        \eqref{equ:formal asymptotics:properties velocity leading order:pressure deriv zero}
        implies $ \schemeLocalFun{\pfPressure{}{}}_{\minExpInd} = 0 $.

        Repeating the arguments of the previous paragraph, we may from now on  assume w.l.o.g.
        $ \schemeLocalFun{\pfVelocity{}{}}_{\ell} = 0 $ for all $ \ell \leq-3 $
        and
        $ \schemeLocalFun{\pfPressure{}{}}_{\ell} = 0 $ for all $ \ell \leq-4 $.

        At order $ \phaseFieldParam^{-4} $, we have an additional right hand side term
        \begin{equation*}
            \phaseFieldParam^{-4}
            \left(
                -
              	\viscosity
                \simpleDeriv{\prime\prime}{
              	\schemeLocalFun{\pfVelocity{}{}}_{-2}
            }
            \concat\interfCoords[\phaseFieldParam]{}{} 
            +
            \simpleDeriv{\prime}{\schemeLocalFun{\pfPressure{}{}}_{-3}}
            \concat\interfCoords[\phaseFieldParam]{}{} 
            \extNormal{}
            \right)
            =
            \phaseFieldParam^{-1}
            \left(
              	\grad{\membrPhaseFieldSym}{L^2}{}\freeEnergyFunct{}{}
                \simpleDeriv{\prime}{\schemeLocalFun{\membrPhaseFieldSym}_0}
                \concat\interfCoords[\phaseFieldParam]{}{}
                \extNormal{}
                +
              	\grad{\cortexPhaseFieldSym}{L^2}{}\freeEnergyFunct{}{}
                \simpleDeriv{\prime}{\schemeLocalFun{\cortexPhaseFieldSym}_0}
                \concat\interfCoords[\phaseFieldParam]{}{}
                \extNormal{}
            \right)
            -
            \pDiff{\pfSpeciesOne{}{}}{}{}\pfCouplingIntMembr
            \meanCurv{\schemeInnerExp{\cortexPhaseFieldSym}_0}{}
            \normal[\schemeInnerExp{\cortexPhaseFieldSym}_0]{}
            \schemeInnerExp{\pfSpeciesOneIdx{-1}{}{}}.
        \end{equation*}
        Multiplying again by $ \extNormal{} $ and noting that due to the previous considerations
        $ \simpleDeriv{\prime}{\schemeLocalFun{\pfVelocity{}{}}_{-2}}
        \concat \interfCoords[\phaseFieldParam]{}{}
        \cdot\extNormal{} = 0 $ \eqref{equ:formal asymptotics:properties velocity leading order:norm comp deriv},
        we have
        \begin{align*}
            \phaseFieldParam^{-4}
            \simpleDeriv{\prime}{\schemeLocalFun{\pfPressure{}{}}_{-3}}
            \concat\interfCoords[\phaseFieldParam]{}{} 
            =
            \phaseFieldParam^{-1}
            \left(
              	\grad{\membrPhaseFieldSym}{L^2}{}\freeEnergyFunct{}{}
                \simpleDeriv{\prime}{\schemeLocalFun{\membrPhaseFieldSym}_0}
                \concat\interfCoords[\phaseFieldParam]{}{}
                +
              	\grad{\cortexPhaseFieldSym}{L^2}{}\freeEnergyFunct{}{}
                \simpleDeriv{\prime}{\schemeLocalFun{\cortexPhaseFieldSym}_0}
                \concat\interfCoords[\phaseFieldParam]{}{}
            \right)
            -
            \pDiff{\pfSpeciesOne{}{}}{}{}\pfCouplingIntMembr
            \meanCurv{\schemeInnerExp{\cortexPhaseFieldSym}_0}{}
            \normal[\schemeInnerExp{\cortexPhaseFieldSym}_0]{}
            \cdot\extNormal{}
            \schemeInnerExp{\pfSpeciesOneIdx{-1}{}{}};
        \end{align*}
        hence,
        $
        	\simpleDeriv{\prime\prime}{
              	\schemeLocalFun{\pfVelocity{}{}}_{-2}
            } = 0
        $ and we may conclude 
        $
              	\schemeLocalFun{\pfVelocity{}{}}_{-2} = 0
        $ as before.

        We cannot go further now.
        However, in Section~\ref{sec:formal asymptotics:optimal profile cortex phase field}
        we show that actually 
        $ \grad{\rho}{L^2}{}\pfGenWillmoreEnergyFunct{}{} \in \landauBigo{\phaseFieldParam^{-2}} $
        and in 
        Section~\ref{sec:formal asymptotics:inner expansion:properties of species to leading order}        
        that $ \schemeInnerExp{\pfSpeciesOne{}{}} \in \landauBigo{1} $---using only 
        the results on velocity and pressure we've derived here---, which gives
        $
            \simpleDeriv{\prime}{\schemeLocalFun{\pfPressure{}{}}_{-3}}
            \concat\interfCoords[\phaseFieldParam]{}{} 
            =
            0,
        $
        and further
        \begin{align*}
            \phaseFieldParam^{-3}
            \simpleDeriv{\prime}{\schemeLocalFun{\pfPressure{}{}}_{-2}}
            \concat\interfCoords[\phaseFieldParam]{}{} 
            =
            \phaseFieldParam^{-1}
            \left(
              	\grad{\membrPhaseFieldSym}{L^2}{}\freeEnergyFunct{}{}
                \simpleDeriv{\prime}{\schemeLocalFun{\membrPhaseFieldSym}_0}
                \concat\interfCoords[\phaseFieldParam]{}{}
                +
              	\grad{\cortexPhaseFieldSym}{L^2}{}\freeEnergyFunct{}{}
                \simpleDeriv{\prime}{\schemeLocalFun{\cortexPhaseFieldSym}_0}
                \concat\interfCoords[\phaseFieldParam]{}{}
            \right)
            -
            \pDiff{\pfSpeciesOne{}{}}{}{}\pfCouplingIntMembr
            \meanCurv{\schemeInnerExp{\cortexPhaseFieldSym}_0}{}
            \normal[\schemeInnerExp{\cortexPhaseFieldSym}_0]{}
            \cdot\extNormal{}
            \schemeInnerExp{\pfSpeciesOneIdx{0}{}{}}
        \end{align*}
        resulting in 
        $
            \simpleDeriv{\prime\prime}{
                \schemeLocalFun{\pfVelocity{}{}}_{-1}
            } = 0
        $ and 
        $
                \schemeLocalFun{\pfVelocity{}{}}_{-1} = 0
        $. 
        All together, we can thus state that 
        \begin{equation} 
            \label{equ:formal asymptotics:inner expansion:velocity and pressure zeroing}
            \schemeLocalFun{\pfVelocity{}{}}_{\ell} = 0 
            \;\text{for all}\; \ell\in\{\minExpInd,\dots,-1\}, \;\text{and}\;
            \schemeLocalFun{\pfPressure{}{}}_\ell = 0 \;\text{for all}\; \ell\in\{\minExpInd,\dots,-3\}.
        \end{equation}
	\subsection{Optimal profiles of $ \schemeLocalFun{\cortexPhaseFieldSym} $
      and $ \schemeLocalFun{\membrPhaseFieldSym} $
      to leading order}
    	\label{sec:formal asymptotics:optimal profile cortex phase field}
        Leading order of $ \grad{\cortexPhaseFieldSym}{L^2}{}\ginzburgLandauEnergyFunct{}{} $ and
        $ \grad{\cortexPhaseFieldSym}{L^2}{}\pfCouplingEnergy{}{} $ is at most $ \phaseFieldParam^{-2} $.
        We consider the evolution law
        \eqref{equ:modelling:phase field:resulting PDEs:second interface evolution}:
        \begin{equation*}
           	\pDiff{t}{}{}\cortexPhaseFieldSym
            +
            \pfVelocity{}{}
            \cdot
            \grad{}{}{}\cortexPhaseFieldSym
            =
            \phaseFieldParam^{\alpha}
            \laplacian{}{}{
                \grad{\cortexPhaseFieldSym}{L^2}{}\pfGenWillmoreEnergyFunct{}{}
                +
                \grad{\cortexPhaseFieldSym}{L^2}{}\ginzburgLandauEnergyFunct{}{}
                +
                \grad{\cortexPhaseFieldSym}{L^2}{}\pfCouplingEnergy{}{}
            }.
        \end{equation*}
        The left hand side is at most of order $ \phaseFieldParam^{-2} $
        (since the velocity is at most of order $ \phaseFieldParam^{-1} $, see the previous
        Section~\ref{sec:formal asymptotics:inner expansion:properties of vel to leading order}).
        So requiring $ \alpha \leq 2 $,
        the leading order terms of 
        $ 
			\phaseFieldParam^\alpha
        	\laplacian{}{}{
                \grad{\cortexPhaseFieldSym}{L^2}{}\pfGenWillmoreEnergyFunct{}{}
        	} 
        $ are of order $ \phaseFieldParam^{-3} $ and must be zero, which 
        is equivalent to the equation
        \begin{equation*}
            \simpleDeriv{\prime\prime}{
            \left(
                \simpleDeriv{\prime\prime}{\left(
                    \simpleDeriv{\prime\prime}{\cortexPhaseFieldIFCIdx{0}{}{}}
                    - 
                    \doubleWellPot[\schemeSupscr{\prime}]{\cortexPhaseFieldIFCIdx{0}{}{}}
                \right)}
                -
                \left(
                    \simpleDeriv{\prime\prime}{\cortexPhaseFieldIFCIdx{0}{}{}}
                    - 
                    \doubleWellPot[\schemeSupscr{\prime}]{\schemeLocalFun{\cortexPhaseFieldSym}_0}
                \right)
                \doubleWellPot[\schemeSupscr{\prime\prime}]{\schemeLocalFun{\cortexPhaseFieldSym}_0}
            \right)
            }
            =
            0.
        \end{equation*}
        We pose the additional condition $ \cortexPhaseFieldIFCIdx{0}{t}{\projVar,0} = 0 $ 
        (otherwise, we had infinitely
        many solutions by shifting along the abscissa).
        Further, we set 
        $ 
        	g \colonequals 
            \simpleDeriv{\prime\prime}{\left(
                \simpleDeriv{\prime\prime}{\cortexPhaseFieldIFCIdx{0}{}{}}
                - 
                \doubleWellPot[\schemeSupscr{\prime}]{\cortexPhaseFieldIFCIdx{0}{}{}}
            \right)}
            -
            \left(
                \simpleDeriv{\prime\prime}{\cortexPhaseFieldIFCIdx{0}{}{}}
                - 
                \doubleWellPot[\schemeSupscr{\prime}]{\schemeLocalFun{\cortexPhaseFieldSym}_0}
            \right)
            \doubleWellPot[\schemeSupscr{\prime\prime}]{\schemeLocalFun{\cortexPhaseFieldSym}_0}
        $ 
        and observe that thanks to
        the counterparts of
        \eqref{equ:formal asymptotics:lim matching principle:membr phase field:outer},
        \eqref{equ:formal asymptotics:lim matching principle:membr phase field:inner} for $ \cortexPhaseFieldSym $,
        $ \lim_{\abs{z}\to\infty} g = 0 $.
        By integration, we obtain
        $$
        	0 = \simpleDeriv{\prime}{g}(z) - \simpleDeriv{\prime}{g}(0),
        $$
        and sending $\abs{z} \to \infty $ gives $ \simpleDeriv{\prime}{g}(0) = 0 $. Conclusively, 
        $ \simpleDeriv{\prime}{g}(z) = 0 $ for all $ z \in \reals $. 
        Repeating the argument, we obtain 
        $$ 
        	0 = g(z) - g(0),
        $$ 
        send $ \abs{z} \to \infty $, conclude $ g(0) = 0 $ and thus have
        $ g(z)= 0 $ for all $ z \in \reals $.
        Setting $ f \colonequals 
            \left(
              	\simpleDeriv{\prime\prime}{\cortexPhaseFieldIFCIdx{0}{}{}}
                \concat\interfCoords[\phaseFieldParam]{}{}
            	- \doubleWellPot[\schemeSupscr{\prime}]{\cortexPhaseFieldSym_0}
            \right) 
        $, a solution to 
        $$
            \simpleDeriv{\prime\prime}{\left(
                \simpleDeriv{\prime\prime}{\cortexPhaseFieldIFCIdx{0}{}{}}
                - 
                \doubleWellPot[\schemeSupscr{\prime}]{\cortexPhaseFieldIFCIdx{0}{}{}}
            \right)}
            -
            \left(
                \simpleDeriv{\prime\prime}{\cortexPhaseFieldIFCIdx{0}{}{}}
                - 
                \doubleWellPot[\schemeSupscr{\prime}]{\schemeLocalFun{\cortexPhaseFieldSym}_0}
            \right)
            \doubleWellPot[\schemeSupscr{\prime\prime}]{\schemeLocalFun{\cortexPhaseFieldSym}_0}
            =
        	g = 0
        $$
        is obviously given by $ f = 0 $. From
        \begin{equation}
            \label{equ:formal asymptotics:leading order props cortex phase field:optimal profile equ}
            f = \simpleDeriv{\prime\prime}{\cortexPhaseFieldIFCIdx{0}{}{}}
            \concat\interfCoords[\phaseFieldParam]{}{}
            - \doubleWellPot[\schemeSupscr{\prime}]{\cortexPhaseFieldSym_0}
            =
            0,
        \end{equation}
        we further conclude with the counterparts of 
        \eqref{equ:formal asymptotics:lim matching principle:membr phase field:outer},
        \eqref{equ:formal asymptotics:lim matching principle:membr phase field:inner} 
        for $ \cortexPhaseFieldSym $ 
        that
        $ 
	        \cortexPhaseFieldIFCIdx{0}{\NOARG}{\fastVar}
            =
            \tanh\left(\frac{\fastVar}{\sqrt{2}}\right).
        $

        The very same argument applies verbatim for $ \membrPhaseFieldSym $.
    \subsection{Properties of $ \schemeLocalFun{\pfSpeciesOne{}{}} $ and $
        \schemeLocalFun{\pfSpeciesTwo{}{}} $ to leading order}
      	\label{sec:formal asymptotics:inner expansion:properties of species to leading order}
        The following analysis is conducted on the example of $ \schemeLocalFun{\pfSpeciesOne{}{}} $,
        but the arguments are the same for $ \schemeLocalFun{\pfSpeciesTwo{}{}} $.
    	We consider the equation \eqref{equ:modelling:phase field:resulting PDEs:%
        reaction-diffusion of species one} on $ \tubNeighbourhood{\ifCoordsDelta}{\interfTwo{}} $:
        \begin{equation*}
            \ginzburgLandauEnergyDens{\cortexPhaseField{}{}}{}
            \pDiff{t}{}{}\pfSpeciesOne{}{}
            -
            \symVelocityNormal{\cortexPhaseField{}{}}
            \meanCurv{\cortexPhaseField{}{}}{}
            \pfSpeciesOne{}{}
            -
            \diver{}{}{
                \ginzburgLandauEnergyDens{\cortexPhaseField{}{}}{}
                \speciesOneDiffusiv{}
                \grad{}{}{}\speciesOne{}{}
            }
            +
            \diver{}{}{
              	\ginzburgLandauEnergyDens{\cortexPhaseFieldSym}{}
                \pfSpeciesOne{}{}
              	\pfVelocity{}{}_\tau
            }
            =
            \ginzburgLandauEnergyDens{\cortexPhaseField{}{}}{}
            \reactionOne{\speciesOne{}{},\speciesTwo{}{}}{\membrPhaseField{}{},
              \normal[\cortexPhaseFieldSym]{}}.
        \end{equation*}      
        Using
        \eqref{equ:formal asymptotics:expansion of the L2 gradient of the coupling energy:%
        helping expansions:ginzburg landau density},
        the results from Section~\ref{sec:formal asymptotics:inner expansion:%
        properties of vel to leading order}, \eqref{equ:formal asymptotics:inner expansion:velocity and pressure zeroing},
        the optimal profile found for $ \cortexPhaseFieldSym $ in
        Section~\ref{sec:formal asymptotics:optimal profile cortex phase field} together with
        \eqref{equ:formal asymptotics:expansion of the L2 gradient of the coupling energy:%
        helping expansions:phase field mean curv}, we have
        $$
            \ginzburgLandauEnergyDens{\cortexPhaseField{}{}}{}
            \pDiff{t}{}{}\schemeInnerExp{\pfSpeciesOne{}{}},
            \symVelocityNormal{\cortexPhaseField{}{}}
            \meanCurv{\cortexPhaseField{}{}}{}
            \schemeInnerExp{\pfSpeciesOne{}{}},
            \diver{}{}{
              	\ginzburgLandauEnergyDens{\schemeInnerExp{\cortexPhaseFieldSym}}{}
                \schemeInnerExp{\pfSpeciesOne{}{}}
              	\schemeInnerExp{\pfVelocity{}{}}_\tau
            },
            \ginzburgLandauEnergyDens{\cortexPhaseField{}{}}{}
            \reactionOne{\speciesOne{}{},\speciesTwo{}{}}{\membrPhaseField{}{},
            \normal[\cortexPhaseFieldSym]{}}
        	\in\landauBigo{\phaseFieldParam^{\minExpInd-2}},
        $$
        so to leading order only the terms 
        at $ \phaseFieldParam^{\minExpInd-3} $ of the diffusion term matter:
        \begin{align*}
            \diver{x}{}{
               	\phaseFieldParam^{\minExpInd-2}
                \left(
                    \frac{1}{2}
                    \left(
                        \simpleDeriv{\prime}{\cortexPhaseFieldIFCIdx{0}{}{}}
                    \right)^2
                    +
                    \doubleWellPot{\cortexPhaseFieldIFCIdx{0}{}{}}
                \right)
                \concat\interfCoords[\phaseFieldParam]{}{}
                \activeLinkersDiffusiv{}
                \simpleDeriv{\prime}{\schemeLocalFun{\pfSpeciesOne{}{}}_{\minExpInd}}
                \concat\interfCoords[\phaseFieldParam]{}{}
                \extNormal{}
            }
            &=
            \phaseFieldParam^{\minExpInd-3}
            \simpleDeriv{\prime}{\left(
                \left(
                    \frac{1}{2}
                    \left(
                        \simpleDeriv{\prime}{\cortexPhaseFieldIFCIdx{0}{}{}}
                    \right)^2
                    +
                    \doubleWellPot{\cortexPhaseFieldIFCIdx{0}{}{}}
                \right)
                \activeLinkersDiffusiv{}
                \simpleDeriv{\prime}{\schemeLocalFun{\pfSpeciesOne{}{}}_{\minExpInd}}
                \hat{\nu}
            \right)}
            \cdot \extNormal{}
            \\
            +
            \landauBigo{
            \phaseFieldParam^{\minExpInd-2}
            }
            &=
            0.
        \end{align*}
        Thus, 
        $
            \left(
              	\frac{1}{2}
                \left(
                    \simpleDeriv{\prime}{\cortexPhaseFieldIFCIdx{0}{}{}}
                \right)^2
                +
                \doubleWellPot{\cortexPhaseFieldIFCIdx{0}{}{}}
            \right)
            \activeLinkersDiffusiv{}
            \simpleDeriv{\prime}{\schemeLocalFun{\pfSpeciesOne{}{}}_{\minExpInd}}
        $ has to be constant in $ \fastVar $. However, 
        $
            \frac{1}{2}
            \left(
                \simpleDeriv{\prime}{\cortexPhaseFieldIFCIdx{0}{}{}}
            \right)^2
            +
            \doubleWellPot{\cortexPhaseFieldIFCIdx{0}{}{}}
        $
        decays due to \eqref{equ:formal asymptotics:lim matching principle:membr phase field:outer},
        \eqref{equ:formal asymptotics:lim matching principle:membr phase field:inner}.
        Simultaneously, 
        $ \simpleDeriv{\prime}{\schemeLocalFun{\pfSpeciesOne{}{}}_{\minExpInd}} $ decays
        as $ \abs{\fastVar} \to \infty $, see 
        \eqref{equ:formal asymptotics:inner expansion:negative order z limit}.
        Thus, it must even hold
        $$
            \left(
              	\frac{1}{2}
                \left(
                    \simpleDeriv{\prime}{\cortexPhaseFieldIFCIdx{0}{}{}}
                \right)^2
                +
                \doubleWellPot{\cortexPhaseFieldIFCIdx{0}{}{}}
            \right)
            \activeLinkersDiffusiv{}
            \simpleDeriv{\prime}{\schemeLocalFun{\pfSpeciesOne{}{}}_{\minExpInd}} = 0,
        $$
        and so
        $
            \simpleDeriv{\prime}{\schemeLocalFun{\pfSpeciesOne{}{}}_{\minExpInd}}(\projVar,\fastVar)
            = 0
        $
        for all $ \projVar,\fastVar $. Consequently,
        $ \schemeLocalFun{\pfSpeciesOne{}{}}_{\minExpInd} $ is constant in $ \fastVar $.
        Leveraging 
        \eqref{equ:formal asymptotics:inner expansion:negative order z limit} again,
        it follows
        $ \schemeLocalFun{\pfSpeciesOne{}{}}_{\minExpInd} = 0 $.
        We may repeat this argument and find
        \begin{equation}
            \label{equ:formal asymptotics:inner expansion:species to leading order:%
            negative order zero}
            \schemeLocalFun{\pfSpeciesOne{}{}}_{\ell} = 0 
            \quad\quad \text{for\;all}\;\ell \in \{\minExpInd,\dots,-1\}.
        \end{equation}

        Finally, we have to leading order:
		\begin{align*}
            \simpleDeriv{\prime}{\left(
                \left(
                    \frac{1}{2}
                    \left(
                        \simpleDeriv{\prime}{\cortexPhaseFieldIFCIdx{0}{}{}}
                    \right)^2
                    +
                    \doubleWellPot{\cortexPhaseFieldIFCIdx{0}{}{}}
                \right)
                \activeLinkersDiffusiv{}
                \simpleDeriv{\prime}{\schemeLocalFun{\pfSpeciesOne{}{}}_{0}}
                \hat{\nu}
            \right)}
            \cdot \extNormal{}
            =
            0,
		\end{align*}         
        and conclude
        \begin{equation}
            \label{equ:formal asymptotics:inner expansion:species to leading order:%
            constant in normal direction}
            \simpleDeriv{\prime}{\schemeLocalFun{\pfSpeciesOne{}{}}_{0}}(\projVar,\fastVar)
            = 0
            \quad\quad
            \text{for\;all}\;\projVar,\fastVar.
        \end{equation}

    \subsection{Further properties of $\schemeLocalFun{\membrPhaseFieldSym} $ and
      $ \schemeLocalFun{\cortexPhaseFieldSym}$}
    	\label{sec:Expanding the L2 gradient of the canham-helfrich energy}
        The expansion of the Willmore-Energy gradient in interfacial coordinates shall be
        \begin{equation*}
            \widehat{\grad{\phaseFieldProto}{L^2}{}
            \pfGenWillmoreEnergyFunct{}{}}
            =
            \sum_{k=-3}^\infty \phaseFieldParam^{k} \schemeLocalFun{e}_k(\projVar,\fastVar),
        \end{equation*}
        and in original coordinates
        \begin{equation}
            \label{equ:formal asymptotics:inner expansion energy gradient}
            \grad{\phaseFieldProto}{L^2}{}
            \pfGenWillmoreEnergyFunct{}{}
            =
            \sum_{k=-3}^\infty \phaseFieldParam^{k} 
            e_k\left(
                \orthProj{\surfaceProto}{x},\frac{\sdist{\surfaceProto}{x}}{\phaseFieldParam}
            \right).
        \end{equation}            
        We are going to show that $ \schemeLocalFun{e}_{-3} = \schemeLocalFun{e}_{-2} = \schemeLocalFun{e}_{-1} = 0 $ 
        by dint of the energy inequality.
        Afterwards, we are going to see that important properties of $ \schemeLocalFun{\phaseFieldProto}_0 $,
        $ \schemeLocalFun{\phaseFieldProto}_1 $, and $ \schemeLocalFun{\phaseFieldProto}_2 $ follow from these
        equations that
        we will use when passing to the limit in the next
        Section~\ref{sec:sharp interface limit}.
        Before going on, let us calculate
        \begin{equation*}
            \grad{\phaseFieldProto}{L^2}{}\pfGenWillmoreEnergyFunct{\phaseFieldProto}{}
            =
            -\laplacian{}{}{\genGlChemPot{\phaseFieldProto}{}}
            +
            \phaseFieldParam^{-2}
            \genGlChemPot{\phaseFieldProto}{}
            \doubleWellPot[\schemeSupscr{\prime\prime}]{\phaseFieldProto}.
        \end{equation*}
            
        Thanks to the optimal profiles at leading order for both phase fields, cf.
        Section~\ref{sec:formal asymptotics:optimal profile cortex phase field}, we have
        $ \grad{\phaseFieldProto}{L^2}{}\pfGenWillmoreEnergyFunct{\phaseFieldProto}{}
        \in \landauBigo{\phaseFieldParam^{-2}} $,
        and also
        $ \grad{\phaseFieldProto}{L^2}{}\ginzburgLandauEnergyFunct{\phaseFieldProto}{} \in \landauBigo{1} $.
        We note further
        \begin{equation*}
            \grad{\membrPhaseField{}{}}{L^2}{}\pfCouplingEnergy{}{}(y)
            =
            \glGenChemPot{\membrPhaseField{}{}}{}(y)
            \pfCouplingIntCortex(y)
            -
            \phaseFieldParam
            \integral{\domain}{}{
                \ginzburgLandauEnergyDens{\cortexPhaseField{}{}}{\phaseFieldParam}
                (x)
                \grad{y}{}{}\membrPhaseField{}{}
                \cdot
                \grad{y}{}{}\pfCouplingEnergyDens{x, y,\pfSpeciesOne{}{},
                  \normal[\cortexPhaseFieldSym]{}}{}
            }{\lebesgueM{3}(x)} \in \landauBigo{1},
        \end{equation*}
        which follows from the optimal profile of $ \membrPhaseFieldSym $ and 
        $ \cortexPhaseFieldSym $ to leading order combined with Lemma~\ref{lemma:addendum formal asymptotics:concentration}.
        (The optimal profiles allow for showing the decaying condition that is the main prerequisite of
        Lemma~\ref{lemma:addendum formal asymptotics:concentration}.)
        For
        \begin{equation*}
            \begin{split}
                \grad{\cortexPhaseField{}{}}{L^2}{}
                \pfCouplingEnergy{}{}(x)
                &=
                \genGlChemPot{\cortexPhaseField{}{}}{}(x)
                \pfCouplingIntMembr(x)
                -\integral{\domain}{}{
                    \phaseFieldParam
                    \ginzburgLandauEnergyDens{\membrPhaseField{}{}}{}(y)
                    \grad{x}{}{
                        \pfCouplingEnergyDens{x,y}{\pfSpeciesOne{}{},\normal[\cortexPhaseFieldSym]{}}
                    }
                    \cdot
                    \grad{x}{}{}\cortexPhaseField{}{}
                }{\lebesgueM{3}(y)}
                \\
                &-
                \integral{\domain}{}{
                  	\ginzburgLandauEnergyDens{\membrPhaseField{}{}}{}(y)
                    \diver{x}{}{
                      	\ginzburgLandauEnergyDens{\cortexPhaseField{}{}}{}(x)
                      	\transposed{
                        \grad{\normal{}}{}{}\pfCouplingEnergyDens{x,y}{\pfSpeciesOne{}{},
                        \normal[\cortexPhaseFieldSym]{}}
                        }
                      	\frac{1}{\abs{\grad{}{}{}\cortexPhaseFieldSym}}
                        \orthProjMat{\normal[\cortexPhaseFieldSym]{}}{}
                    }
                }{\lebesgueM{3}(y)} \in \landauBigo{1},
            \end{split}
        \end{equation*}
        we have to additionally consider
        \eqref{equ:formal asymptotics:consequ bound L2 grad coupling energy:dxdnc}, and note
        \eqref{equ:formal asymptotics:consequ bound L2 grad coupling energy:dnc with proj %
                lower order},
        as well as $ \frac{1}{\abs{\grad{}{}{}\cortexPhaseFieldSym}}
        \orthProjMat{\normal[\cortexPhaseFieldSym]{}}{} \in \landauBigo{\phaseFieldParam} $.
        We conclude 
        $ 
            \avgIntegral{\tubNeighbourhood{\ifCoordsDelta}{S}}{}{
                \grad{\phaseFieldProto}{}{}\freeEnergyFunct{}{}
            }{\lebesgueM{3}}
            \in \landauBigo{\phaseFieldParam^{-1}}
        $
        (again, leveraging Lemma~\ref{lemma:addendum formal asymptotics:concentration} and using the optimal
        profile of $ \membrPhaseField{}{} $ and $ \cortexPhaseField{}{} $ to verify the prerequisites).
        The energy inequality \eqref{equ:formal asymptotics:energy inequ:pre} gives us additionally
        \begin{equation*}
            \integral{0}{\finTime}{
                \integral{\tubNeighbourhood{\ifCoordsDelta}{S}}{}{
                    \phaseFieldParam^\alpha
                    \abs{
                        \grad{}{}{}\grad{\phaseFieldProto}{}{}\freeEnergyFunct{}{}
                    }^2
                }{\lebesgueM{3}}
            }{\lebesgueM{1}}
            \in
            \landauBigo{1}.
        \end{equation*}
        Note that we can restrict to $ \tubNeighbourhood{\ifCoordsDelta}{S} $ since
        the energies and their $L^2 $-gradients are zero outside to leading order.
        Applying the Poincar\'{e}-Wirtinger inequality, we deduce,
        \begin{equation*}
            \left(
              	\integral{0}{\finTime}{
                \integral{\tubNeighbourhood{\ifCoordsDelta}{S}}{}{
                    \phaseFieldParam^\alpha
                    \left(
                        \grad{\phaseFieldProto}{}{}\freeEnergyFunct{}{}
                        -
                        \avgIntegral{\tubNeighbourhood{\ifCoordsDelta}{S}}{}{
                        	\grad{\phaseFieldProto}{}{}\freeEnergyFunct{}{}
                        }{\lebesgueM{3}}
                    \right)^2
                }{\lebesgueM{3}}
                }{\lebesgueM{1}}
            \right)^{\frac{1}{2}}
            \leq
            \left(
              	\integral{0}{\finTime}{
                \integral{\tubNeighbourhood{\ifCoordsDelta}{S}}{}{
                    \phaseFieldParam^\alpha
                    \abs{
                        \grad{}{}{}\grad{\phaseFieldProto}{}{}\freeEnergyFunct{}{}
                    }^2
                }{\lebesgueM{3}}
                }{\lebesgueM{1}}
            \right)^{\frac{1}{2}},
        \end{equation*}
        which implies, using the reversed triangle inequality for 
        $ \LTNorm{\tubNeighbourhood{\ifCoordsDelta}{S}}{\cdot} $,
        \begin{equation}
            \label{equ:formal asymptotics:order of cont equ chem pot L2 norm}
            \begin{split}
                \left(
                    \integral{0}{\finTime}{
                    \integral{\tubNeighbourhood{\ifCoordsDelta}{S}}{}{
                        \phaseFieldParam^\alpha
                        \abs{\grad{\phaseFieldProto}{}{}\freeEnergyFunct{}{}}^2
                    }{\lebesgueM{3}}
                    }{\lebesgueM{1}}
                \right)^{\frac{1}{2}}
                -
                \sqrt{\abs{\tubNeighbourhood{\ifCoordsDelta}{S}}}
                \phaseFieldParam^{\frac{\alpha}{2}}
                \avgIntegral{\tubNeighbourhood{\ifCoordsDelta}{S}}{}{
                	\grad{\phaseFieldProto}{}{}\freeEnergyFunct{}{}
                }{\lebesgueM{3}}
                &\in 
                \landauBigo{1},
           \end{split}
		\end{equation}
        thus
        \begin{equation}
            \begin{split}
                \left(
                    \integral{0}{\finTime}{
                    \integral{\tubNeighbourhood{\ifCoordsDelta}{S}}{}{
                        \abs{\grad{\phaseFieldProto}{}{}\freeEnergyFunct{}{}}^2
                    }{\lebesgueM{3}}
                    }{\lebesgueM{1}}
                \right)^{\frac{1}{2}}
                -
                \sqrt{\abs{\tubNeighbourhood{\ifCoordsDelta}{S}}}
                \avgIntegral{\tubNeighbourhood{\ifCoordsDelta}{S}}{}{
                	\grad{\phaseFieldProto}{}{}\freeEnergyFunct{}{}
                }{\lebesgueM{3}}
                &\in 
                \landauBigo{\phaseFieldParam^{-\frac{\alpha}{2}}},
            \end{split}
        \end{equation}
        so
        \begin{equation*}
            \integral{0}{\finTime}{
            \integral{\tubNeighbourhood{\ifCoordsDelta}{S}}{}{
               	\abs{\grad{\phaseFieldProto}{}{}\freeEnergyFunct{}{}}^2
            }{\lebesgueM{3}}
            }{\lebesgueM{1}}
            \in \landauBigo{\phaseFieldParam^{-2}}
        \end{equation*}
        for $ \alpha \leq 2 $.
        By applying Young's inequality, we can deduce further
        \begin{equation*}
            \begin{split}
                \integral{\tubNeighbourhood{\ifCoordsDelta}{S}}{}{
                  	\abs{\grad{\phaseFieldProto}{}{}\freeEnergyFunct{}{}}^2
                }{\lebesgueM{3}}
                &=
                \integral{\tubNeighbourhood{\ifCoordsDelta}{S}}{}{
                    \abs{
                        \grad{\phaseFieldProto}{L^2}{}
                        \pfCouplingEnergy{}{}
                        +
                        \grad{\phaseFieldProto}{L^2}{}
                        \ginzburgLandauEnergyFunct{}{}
                    }^2
                    +
                    2
                    \grad{\phaseFieldProto}{L^2}{}
                    \pfGenWillmoreEnergyFunct{}{}
                    \cdot
                    \left(
                        \grad{\phaseFieldProto}{L^2}{}
                        \pfCouplingEnergy{}{}
                        +
                        \grad{\phaseFieldProto}{L^2}{}
                        \ginzburgLandauEnergyFunct{}{}
                    \right)
                    +
                    \abs{
                        \grad{\phaseFieldProto}{L^2}{}
                        \pfGenWillmoreEnergyFunct{}{}
                    }^2
                }{\lebesgueM{3}},
            \end{split}
        \end{equation*}
        which in turn implies
        \begin{equation*}
            \begin{split}
                \frac{1}{2}
                \integral{\tubNeighbourhood{\ifCoordsDelta}{S}}{}{
                    \abs{
                        \grad{\phaseFieldProto}{L^2}{}
                        \pfGenWillmoreEnergyFunct{}{}
                    }^2
                }{\lebesgueM{3}}                
                &\leq
                \integral{\tubNeighbourhood{\ifCoordsDelta}{S}}{}{
                  	\abs{\grad{\phaseFieldProto}{}{}\freeEnergyFunct{}{}}^2
                    +
                    3
                    \abs{
                        \grad{\phaseFieldProto}{L^2}{}
                        \pfCouplingEnergy{}{}
                        +
                        \grad{\phaseFieldProto}{L^2}{}
                        \ginzburgLandauEnergyFunct{}{}
                    }^2
                }{\lebesgueM{3}},
            \end{split}
        \end{equation*}
        so with the co-area formula, it follows
        \begin{equation}
            \label{equ:formal asymptotics:inner expansion:helfrich gradient:energy inequ}
            \phaseFieldParam
            \integral{0}{\finTime}{
                \integral{
                    -\frac{\ifCoordsDelta}{\phaseFieldParam}
                }{
                    \frac{\ifCoordsDelta}{\phaseFieldParam}
                }{
                    \integral{\surfaceProto_{\phaseFieldParam\fastVar}}{}{
                        \abs{
                          	\widehat{
                            \grad{\phaseFieldProto}{L^2}{}
                            \pfGenWillmoreEnergyFunct{}{}}
                        }^2
                    }{\hausdorffM{2}}
                }{\lebesgueM{1}(\fastVar)}
            }{\lebesgueM{1}}
            \in \landauBigo{\phaseFieldParam^{-2}}.
        \end{equation}
        From \eqref{equ:formal asymptotics:inner expansion energy gradient}, the 
        expansion
        \begin{equation*}
            \begin{split}
                \abs{
                  	\widehat{
                    \grad{\phaseFieldProto}{L^2}{}
                    \pfGenWillmoreEnergyFunct{}{}}
                }^2 
                &=
                \phaseFieldParam^{-6}
                \schemeLocalFun{e}_{-3}^2
                +
                \phaseFieldParam^{-5}
                2\schemeLocalFun{e}_{-3}\schemeLocalFun{e}_{-2}
                +
                \phaseFieldParam^{-4}
                \left(
                    \schemeLocalFun{e}_{-2}^2
                    +
                    2\schemeLocalFun{e}_{-3}\schemeLocalFun{e}_{-1}
                \right)
                +
                \phaseFieldParam^{-3}
                \left(
                    2\schemeLocalFun{e}_{-3}\schemeLocalFun{e}_0
                    +
                    2\schemeLocalFun{e}_{-2}\schemeLocalFun{e}_{-1}
                \right)
                \\
                &+
                \phaseFieldParam^{-2}
                \left(
                    \schemeLocalFun{e}_{-1}^2
                    +
                    \schemeLocalFun{e}_{-2}
                    \schemeLocalFun{e}_0
                    +
                    2\schemeLocalFun{e}_{-3}\schemeLocalFun{e}_1
                \right)
                +
                \landauBigo{\phaseFieldParam^{-1}}
                \\
                &=
                \sum_{k=-6}^{-2} \phaseFieldParam^{k} f_k(\projVar,\fastVar) 
                +
                \landauBigo{\phaseFieldParam^{-1}}
            \end{split}  
        \end{equation*}
        of the integrand follows directly.
        Equation
        \eqref{equ:formal asymptotics:inner expansion:helfrich gradient:energy inequ}
        then requires 
        \begin{equation*}
            \integral{0}{\finTime}{
                \integral{-\infty}{\infty}{
                    \integral{\surfaceProto}{}{
                        f_k
                    }{\hausdorffM{2}}
                }{\lebesgueM{1}}
            }{\lebesgueM{1}} = 0 
        \end{equation*} 
        up to $ k \leq -4 $, so
        \begin{equation*}
            \begin{split}
                \schemeLocalFun{e}_{-3}^2 = f_{-6} = 0 \;\text{a.e.}
                &\implies
                \schemeLocalFun{e}_{-3} = 0 \;\text{a.e.} \\
                &\implies
                \schemeLocalFun{e}_{-2}^2 = f_{-4} = 0 \;\text{a.e.} \\
                &\implies
                \schemeLocalFun{e}_{-2} = 0 \; \text{a.e.} 
            \end{split}
        \end{equation*}
        This in turn gives
        $ \grad{\phaseFieldProto}{}{}\freeEnergyFunct{}{}\in\landauBigo{\phaseFieldParam^{-1}} $,
        so $ \avgIntegral{\tubNeighbourhood{\ifCoordsDelta}{S}}{}{
          \grad{\phaseFieldProto}{}{}\freeEnergyFunct{}{}
        }{\lebesgueM{3}} \in \landauBigo{1} $,
        which we insert into 
        \eqref{equ:formal asymptotics:order of cont equ chem pot L2 norm},
        and choose $ \alpha < 1 $ to
        obtain            
        \begin{equation*}
            \integral{\tubNeighbourhood{\ifCoordsDelta}{S}}{}{
              	\abs{\grad{\phaseFieldProto}{}{}\freeEnergyFunct{}{}}^2
            }{\lebesgueM{3}}
            \in \landauSmallO{\phaseFieldParam^{-1}};
        \end{equation*}
        hence,
        \begin{equation*}
            \integral{0}{\finTime}{
                \integral{
                    -\frac{\ifCoordsDelta}{\phaseFieldParam}
                }{
                    \frac{\ifCoordsDelta}{\phaseFieldParam}
                }{
                    \integral{\surfaceProto_{\phaseFieldParam\fastVar}}{}{
                        \abs{
                          	\widehat{
                            \grad{\phaseFieldProto}{L^2}{}
                            \pfGenWillmoreEnergyFunct{}{}}
                        }^2
                    }{\hausdorffM{2}}
                }{\lebesgueM{1}(\fastVar)}
            }{\lebesgueM{1}}
            \in \landauSmallO{\phaseFieldParam^{-2}},
        \end{equation*}        
        so that $ \schemeLocalFun{e}_{-1} = 0 $.

        Now that we have found equations 
        $ \schemeLocalFun{e}_{-1} = 0 $, $ \schemeLocalFun{e}_{-2} = 0$, and
        $\schemeLocalFun{e}_{-3} = 0 $, we may derive information on $ \schemeLocalFun{\phaseFieldProto} $
        from them. 

        \subsubsection{Expansion of the $L^2$-gradient of the Willmore energy}
            We recall that
            \begin{equation}
                \grad{\phaseFieldProto}{L^2}{}
                \pfGenWillmoreEnergyFunct{}{}
                =
                -\laplacian{}{}{ 
                    \genGlChemPot{\phaseFieldProto}{}
                }	
                +
                \phaseFieldParam^{-2}
                \genGlChemPot{\phaseFieldProto}{}
                \doubleWellPot[\schemeSupscr{\prime\prime}]{\phaseFieldProto}{}.
            \end{equation}
            First, we expand the chemical potential,
            \begin{equation*}
                \begin{split}
                    \genGlChemPot{\phaseFieldProto_{0} + \phaseFieldParam e_1}{}
                    &=
                    -\phaseFieldParam
                    \laplacian{}{}{\phaseFieldProto_{0} + \phaseFieldParam e_1}
                    +
                    \phaseFieldParam^{-1}
                    \doubleWellPot[\schemeSupscr{\prime}]{
                      	\phaseFieldProto_{0} 
                        + \phaseFieldParam e_1
                    },
                \end{split}
            \end{equation*}
            by expanding the Laplacian term:
            \begin{equation*}
                \begin{split}
                    \phaseFieldParam
                    \laplacian{}{}{\phaseFieldProto_{0} + \phaseFieldParam e_1}
                    &=
                    \phaseFieldParam^{-1}
                    \simpleDeriv{\prime\prime}{\schemeLocalFun{\phaseFieldProto}_{0}}
                    \concat
                    \interfCoords[\phaseFieldParam]{}{}
                    -
                    \simpleDeriv{\prime}{\schemeLocalFun{\phaseFieldProto}_{0}}                    
                    \concat
                    \interfCoords[\phaseFieldParam]{}{}
                    \extMeanCurv{}
                    +
                    \phaseFieldParam
                    \laplacian{\interfOne{}_{\sdist{}{x}}}{}{}\phaseFieldProto_{0}
                    \\
                    &\quad+
                    \simpleDeriv{\prime\prime}{\schemeLocalFun{\phaseFieldProto}_{1}}
                    \concat
                    \interfCoords[\phaseFieldParam]{}{}
                    -
                    \phaseFieldParam
                    \simpleDeriv{\prime}{\schemeLocalFun{\phaseFieldProto}_{1}}                    
                    \concat
                    \interfCoords[\phaseFieldParam]{}{}
                    \extMeanCurv{}
                    +
                    \phaseFieldParam^2
                    \laplacian{\interfOne{}_{\sdist{}{x}}}{}{}\phaseFieldProto_{1}
                    \\
                    &\quad+
                    \phaseFieldParam
                    \simpleDeriv{\prime\prime}{\schemeLocalFun{\phaseFieldProto}_{2}}
                    \concat
                    \interfCoords[\phaseFieldParam]{}{}
                    -
                    \phaseFieldParam^2
                    \simpleDeriv{\prime}{\schemeLocalFun{\phaseFieldProto}_{2}}                    
                    \concat
                    \interfCoords[\phaseFieldParam]{}{}
                    \extMeanCurv{}
                    \\
                    &\quad+
                    \phaseFieldParam^2
                    \simpleDeriv{\prime\prime}{\schemeLocalFun{\phaseFieldProto}_{3}}
                    \concat
                    \interfCoords[\phaseFieldParam]{}{}
                    +
                    \landauBigo{\phaseFieldParam^3}
                \end{split}
            \end{equation*}
            and the double well potential's first derivative:
            \begin{equation*}
                \begin{split}
                    \phaseFieldParam^{-1}
                    \doubleWellPot[\schemeSupscr{\prime}]{
                      	\phaseFieldProto_{0}
                        +
                        \phaseFieldParam e_1
                    }
                    &=
                    \phaseFieldParam^{-1}
                    \left(
                        \doubleWellPot[\schemeSupscr{\prime}]{\phaseFieldProto_{0}}
                        +
                        \phaseFieldParam
                        \doubleWellPot[\schemeSupscr{\prime\prime}]{
                            \phaseFieldProto_{0}
                        }
                        e_1
                        +
                        \phaseFieldParam^2
                        \doubleWellPot[\schemeSupscr{(3)}]{
                            \phaseFieldProto_{0}
                        }                    
                        e_1^2
                        +
                        \phaseFieldParam^3
                        \doubleWellPot[\schemeSupscr{(4)}]{
                            \phaseFieldProto_{0}
                        }
                        e_1^3
                    \right)
                    \\
                    &=
                    \phaseFieldParam^{-1}
                    \doubleWellPot[\schemeSupscr{\prime}]{\phaseFieldProto_{0}}
                    +
                    \doubleWellPot[\schemeSupscr{\prime\prime}]{
                      	\phaseFieldProto_{0}
                    }
                    \left(
                        \phaseFieldProto_{1}
                        +
                        \phaseFieldParam
                        \phaseFieldProto_{2}
                        +
                        \phaseFieldParam^2
                        \phaseFieldProto_{3}
                    \right)
                    +
                    \phaseFieldParam
                    \doubleWellPot[\schemeSupscr{(3)}]{
                      	\phaseFieldProto_{0}
                    }                    
                    \left(
                        \phaseFieldProto_{1}^2
                        +
                        2
                        \phaseFieldParam
                        \phaseFieldProto_{1}                        
                        \phaseFieldProto_{2}                        
                    \right)
                    +
                    \phaseFieldParam^2
                    \doubleWellPot[\schemeSupscr{(4)}]{
                      	\phaseFieldProto_{0}
                    }
                    \phaseFieldProto_{1}^3
                    \\
                    &+
                    \landauBigo{\phaseFieldParam^3}
                    \\
                    &=
                    \phaseFieldParam^{-1}
                    \doubleWellPot[\schemeSupscr{\prime}]{\phaseFieldProto_{0}}
                    +
                    \doubleWellPot[\schemeSupscr{\prime\prime}]{
                      	\phaseFieldProto_{0}
                    }
                    \phaseFieldProto_{1}
                    +
                    \phaseFieldParam
                    \left(
                        \doubleWellPot[\schemeSupscr{\prime\prime}]{
                            \phaseFieldProto_{0}
                        }
                        \phaseFieldProto_{2}
                        +
                        \doubleWellPot[\schemeSupscr{(3)}]{
                            \phaseFieldProto_{0}
                        }                             
                        \phaseFieldProto_{1}^2
                    \right)
                    \\
                    &+
                    \phaseFieldParam^2
                    \left(
                        \doubleWellPot[\schemeSupscr{\prime\prime}]{
                            \phaseFieldProto_{0}
                        }                    
                        \phaseFieldProto_{3}
                        +
                        2
                        \doubleWellPot[\schemeSupscr{(3)}]{
                            \phaseFieldProto_{0}
                        }                    
                        \phaseFieldProto_{1}                        
                        \phaseFieldProto_{2}                                                
                        +
                        \doubleWellPot[\schemeSupscr{(4)}]{\phaseFieldProto_{0}}
                        \phaseFieldProto_{1}^3
                    \right)
                    +
                    \landauBigo{\phaseFieldParam^3}.
                \end{split}
            \end{equation*}
            The expansion of the chemical potential then reads
            \begin{equation}
                \label{equ:formal asymptotics:CH-L2 gradient expansion:chem pot expansion}
                \begin{split}
                    \genGlChemPot{\phaseFieldProto_{0} + \phaseFieldParam e_1}{}
                    =
                    \phaseFieldParam^{-1}
                    \chemPotIdx{-1}{\phaseFieldProto}{}
                    +
                    \chemPotIdx{0}{\phaseFieldProto}{}
                    +
                    \phaseFieldParam
                    \chemPotIdx{1}{\phaseFieldProto}{}
                    +
                    \phaseFieldParam^2
                    \chemPotIdx{2}{\phaseFieldProto}{}
                    +
                    \landauBigo{\phaseFieldParam^3}
                \end{split}
            \end{equation}
            with
            \begin{equation}
                \label{equ:formal asymptotics:expansion of the L2 gradient of the canham-helfrich energy:%
                chem pot expansion}
                \begin{split}
                    \chemPotIdx{-1}{\phaseFieldProto}{}
                    &=
                        -\simpleDeriv{\prime\prime}{\schemeLocalFun{\phaseFieldProto}_{0}}
                        \concat
                        \interfCoords[\phaseFieldParam]{}{}
                        +
                        \doubleWellPot[\schemeSupscr{\prime}]{
                          	\phaseFieldProto_{0}
                        },
                    \\
                    \chemPotIdx{0}{\phaseFieldProto}{}
                    &=
                        \simpleDeriv{\prime}{\schemeLocalFun{\phaseFieldProto}_{0}}
                        \concat
                        \interfCoords[\phaseFieldParam]{}{}
                        \extMeanCurv{}
                        -
                        \simpleDeriv{\prime\prime}{\schemeLocalFun{\phaseFieldProto}_{1}}
                        \concat
                        \interfCoords[\phaseFieldParam]{}{}
                        +
                        \doubleWellPot[\schemeSupscr{\prime\prime}]{\phaseFieldProto_{0}}
                        \phaseFieldProto_{1},
                    \\
                    \chemPotIdx{1}{\phaseFieldProto}{}
                    &=
                        -\laplacian{\interfOne{}_{\sdist{}{x}}}{}{}\phaseFieldProto_{0}
                        +
                        \simpleDeriv{\prime}{\schemeLocalFun{\phaseFieldProto}_{1}}
                        \concat
                        \interfCoords[\phaseFieldParam]{}{}
                        \extMeanCurv{}
                        -
                        \simpleDeriv{\prime\prime}{\schemeLocalFun{\phaseFieldProto}_{2}}
                        \concat
                        \interfCoords[\phaseFieldParam]{}{}
                        +
                        \doubleWellPot[\schemeSupscr{\prime\prime}]{
                            \phaseFieldProto_{0}
                        }
                        \phaseFieldProto_{2}
                        +
                        \doubleWellPot[\schemeSupscr{(3)}]{
                            \phaseFieldProto_{0}
                        }\phaseFieldProto_{1}^2,
                    \\
                    \chemPotIdx{2}{\phaseFieldProto}{}
                    &=
                        -\laplacian{\interfOne{}_{\sdist{}{x}}}{}{}\phaseFieldProto_{1}
                        +
                        \simpleDeriv{\prime}{\schemeLocalFun{\phaseFieldProto}_{2}}
                        \concat
                        \interfCoords[\phaseFieldParam]{}{}
                        \extMeanCurv{}
                        -
                        \simpleDeriv{\prime\prime}{\schemeLocalFun{\phaseFieldProto}_{3}}
                        \concat
                        \interfCoords[\phaseFieldParam]{}{}
                        +
                        \doubleWellPot[\schemeSupscr{\prime\prime}]{
                            \phaseFieldProto_{0}
                        }                    
                        \phaseFieldProto_{3}
                        +
                        2
                        \doubleWellPot[\schemeSupscr{(3)}]{
                            \phaseFieldProto_{0}
                        }                    
                        \phaseFieldProto_{1}                        
                        \phaseFieldProto_{2}                                                
                        +
                        \doubleWellPot[\schemeSupscr{(4)}]{\phaseFieldProto_{0}}
                        \phaseFieldProto_{1}^3.
                \end{split}
            \end{equation}

            \paragraph{Expansion of $ \laplacian{}{}{\chemPot{\phaseFieldProto}{}} $}
            We may rewrite
            $ 
            	\chemPotIdx{i}{\phaseFieldProto}{} =
                \chemPotIdxPFWD{i}{
                    \schemeLocalFun{\phaseFieldProto}
                }{}
                \concat
                \interfCoords[\phaseFieldParam]{}{} 
            $
            and treat the Laplacian terms 
            $ 
            	\laplacian{}{}{
              		\chemPotIdx{i}{\phaseFieldProto}{}
                } 
            $
            with
            \eqref{equ:appendix:formally matched asymptotics:%
            laplacian rescaled distance}:
            \begin{equation*}
                \laplacian{}{}{\chemPotIdx{i}{\phaseFieldProto}{}}                
                =
                \phaseFieldParam^{-2}
                \simpleDeriv{\prime\prime}{
                	\chemPotIdxPFWD{i}{\schemeLocalFun{\phaseFieldProto}}{}
                }
                \concat
                \interfCoords[\phaseFieldParam]{}{}
                -
                \phaseFieldParam^{-1}
                \simpleDeriv{\prime}{
                	\chemPotIdxPFWD{i}{\schemeLocalFun{\phaseFieldProto}}{}
                }
                \concat
                \interfCoords[\phaseFieldParam]{}{}
                \extMeanCurv{}
                +
                \laplacian{\interfOne{}_{\sdist{}{x}}}{}{
                  	\chemPotIdx{i}{\phaseFieldProto}{}
                }
            \end{equation*}
            giving
            \begin{equation*}
                \begin{split}
                    \laplacian{}{}{
                        \chemPot{\phaseFieldProto}{}
                    } 
                    =
                    \phaseFieldParam^{-3}&
                    \left(
                        \simpleDeriv{\prime\prime}{
                          	\chemPotIdxPFWD{-1}{\schemeLocalFun{\phaseFieldProto}}{}
                        }
                        \concat
                        \interfCoords[\phaseFieldParam]{}{}
                    \right)
                    +
                    \\
                    \phaseFieldParam^{-2}&
                    \left(
                        -
                        \simpleDeriv{\prime}{
                          	\chemPotIdxPFWD{-1}{\schemeLocalFun{\phaseFieldProto}}{}
                        }
                        \concat
                        \interfCoords[\phaseFieldParam]{}{}
                        \extMeanCurv{}
                        +
                        \simpleDeriv{\prime\prime}{
                          	\chemPotIdxPFWD{0}{\schemeLocalFun{\phaseFieldProto}}{}
                        }
                        \concat
                        \interfCoords[\phaseFieldParam]{}{}                        
                    \right)
                    +
                    \\
                    \phaseFieldParam^{-1}&
                    \left(
                        \laplacian{
                          	\interfOne{}_{\sdist{}{x}}
                        }{}{
                          	\chemPotIdx{-1}{\phaseFieldProto}{}
                        }
                        -
                        \simpleDeriv{\prime}{
                          	\chemPotIdxPFWD{0}{\schemeLocalFun{\phaseFieldProto}}{}
                        }
                        \concat
                        \interfCoords[\phaseFieldParam]{}{}
                        \extMeanCurv{}
                        +
                        \simpleDeriv{\prime\prime}{
                          	\chemPotIdxPFWD{1}{\schemeLocalFun{\phaseFieldProto}}{}
                        }
                        \concat
                        \interfCoords[\phaseFieldParam]{}{}
                    \right)
                    +
                    \\
                    &
                    \laplacian{
                      	\interfOne{}_{\sdist{}{x}}
                    }{}{
                      	\chemPotIdx{0}{\phaseFieldProto}{}
                    }
                    -
                    \simpleDeriv{\prime}{
                        \chemPotIdxPFWD{1}{\schemeLocalFun{\phaseFieldProto}}{}
                    }
                    \concat
                    \interfCoords[\phaseFieldParam]{}{}
                    \extMeanCurv{}
                    +
                    \simpleDeriv{\prime\prime}{
                      	\chemPotIdxPFWD{2}{\schemeLocalFun{\phaseFieldProto}}{}
                    }
                    \concat
                    \interfCoords[\phaseFieldParam]{}{}
                    +
                    \\
                    &\landauBigo{\phaseFieldParam}.
                \end{split}
            \end{equation*}

            \paragraph{Expansion of
                            $ \phaseFieldParam^{-2}\doubleWellPot[\schemeSupscr{\prime\prime}]{\phaseFieldProto}
                            \chemPot{\phaseFieldProto}{} $
                       }
            For obtaining the expansion of
            $ 
            	\phaseFieldParam^{-2}
                \doubleWellPot[\schemeSupscr{\prime\prime}]{\phaseFieldProto}
                \chemPot{\phaseFieldProto}{} 
            $,
            the remaining ingredient is an expansion of 
            $ \doubleWellPot[\schemeSupscr{\prime\prime}]{\phaseFieldProto} $:
            \begin{equation}
                \label{equ:formal asymptotics:CH-L2 gradient expansion:double well pp expansion}
                \begin{split}
                    \phaseFieldParam^{-2}
                    \doubleWellPot[\schemeSupscr{\prime\prime}]{\phaseFieldProto_{0}+
                    \phaseFieldParam e_1}
                    &=
                    \phaseFieldParam^{-2}
                    \doubleWellPot[\schemeSupscr{\prime\prime}]{\phaseFieldProto_{0}}
                    +
                    \phaseFieldParam^{-1}
                    \doubleWellPot[\schemeSupscr{(3)}]{
                      	\phaseFieldProto_{0}
                    }
                    \phaseFieldProto_{1}
                    +
                    \doubleWellPot[\schemeSupscr{(3)}]{
                        \phaseFieldProto_{0}
                    }
                    \phaseFieldProto_{2}
                    +
                    \doubleWellPot[\schemeSupscr{(4)}]{
                        \phaseFieldProto_{0}
                    }                             
                    \phaseFieldProto_{1}^2
                    \\
                    &\quad+
                    \phaseFieldParam
                    \left(
                        \doubleWellPot[\schemeSupscr{(3)}]{
                            \phaseFieldProto_{0}
                        }                    
                        \phaseFieldProto_{3}
                        +
                        2
                        \doubleWellPot[\schemeSupscr{(4)}]{
                            \phaseFieldProto_{0}
                        }                    
                        \phaseFieldProto_{1}                        
                        \phaseFieldProto_{2}                                                
                        +
                        \doubleWellPot[\schemeSupscr{(5)}]{\phaseFieldProto_{0}}
                        \phaseFieldProto_{1}^3
                    \right)                    
                    +
                    \landauBigo{\phaseFieldParam^2}.
                \end{split}
            \end{equation}
            Multiplication of 
            \eqref{equ:formal asymptotics:CH-L2 gradient expansion:chem pot expansion}
            and 
            \eqref{equ:formal asymptotics:CH-L2 gradient expansion:double well pp expansion} gives
            \begin{equation*}
                \begin{split}
                    \phaseFieldParam^{-2}\doubleWellPot[\schemeSupscr{\prime\prime}]{\phaseFieldProto}
                    \chemPot{\phaseFieldProto}{}
                    =
                    \phaseFieldParam^{-3}&
                    \left(
                        \chemPotIdx{-1}{\phaseFieldProto}{}
                        \doubleWellPot[\schemeSupscr{\prime\prime}]{\phaseFieldProto_{0}}
                    \right)
                    +
                    \\
                    \phaseFieldParam^{-2}&
                    \left(
                        \chemPotIdx{0}{\phaseFieldProto}{}
                        \doubleWellPot[\schemeSupscr{\prime\prime}]{\phaseFieldProto_{0}}
                        +
                        \chemPotIdx{-1}{\phaseFieldProto}{}
                        \doubleWellPot[\schemeSupscr{(3)}]{\phaseFieldProto_{0}}
                        \phaseFieldProto_{1}
                    \right)
                    +
                    \\
                    \phaseFieldParam^{-1}&
                    \left(
                        \chemPotIdx{1}{\phaseFieldProto}{}                    
                        \doubleWellPot[\schemeSupscr{\prime\prime}]{\phaseFieldProto_{0}}
                        +
                        \chemPotIdx{0}{\phaseFieldProto}{}                                            
                        \doubleWellPot[\schemeSupscr{(3)}]{\phaseFieldProto_{0}}
                        \phaseFieldProto_{1}                        
                        +
                        \chemPotIdx{-1}{\phaseFieldProto}{}
                        \left(
                            \doubleWellPot[\schemeSupscr{(3)}]{\phaseFieldProto_{0}}
                            \phaseFieldProto_{2}                                                
                            +
                            \doubleWellPot[\schemeSupscr{(4)}]{\phaseFieldProto_{0}}
                            \phaseFieldProto_{1}^2
                        \right)
                    \right)
                    +
                    \\
                    &\quad
                    \chemPotIdx{2}{\phaseFieldProto}{}                    
                    \doubleWellPot[\schemeSupscr{\prime\prime}]{\phaseFieldProto_{0}}
                    +
                    \chemPotIdx{1}{\phaseFieldProto}{}                                            
                    \doubleWellPot[\schemeSupscr{(3)}]{\phaseFieldProto_{0}}
                    \phaseFieldProto_{1}                        
                    +
                    \chemPotIdx{0}{\phaseFieldProto}{}
                    \left(
                        \doubleWellPot[\schemeSupscr{(3)}]{\phaseFieldProto_{0}}
                        \phaseFieldProto_{2}                                                
                        +
                        \doubleWellPot[\schemeSupscr{(4)}]{\phaseFieldProto_{0}}
                        \phaseFieldProto_{1}^2
                    \right)
                    +
                    \\
                    &\quad
                    \chemPotIdx{-1}{\phaseFieldProto}{}
                    \left(
                        \doubleWellPot[\schemeSupscr{(3)}]{
                            \phaseFieldProto_{0}
                        }                    
                        \phaseFieldProto_{3}
                        +
                        2
                        \doubleWellPot[\schemeSupscr{(4)}]{
                            \phaseFieldProto_{0}
                        }                    
                        \phaseFieldProto_{1}                        
                        \phaseFieldProto_{2}                                                
                        +
                        \doubleWellPot[\schemeSupscr{(5)}]{\phaseFieldProto_{0}}
                        \phaseFieldProto_{1}^3
                    \right)                    
                    +
                    \\
                    &\quad\quad\landauBigO{\phaseFieldParam}.
                \end{split}
            \end{equation*}

            Finally, we draw the following conclusions for $ \schemeLocalFun{\phaseFieldProto}_0 $,
            $ \schemeLocalFun{\phaseFieldProto}_1 $, and $ \schemeLocalFun{\phaseFieldProto}_2 $
            by evaluating the equations $ \schemeLocalFun{e}_i = 0 $, for $ i \in \{-1,-2,-3\} $:
            \begin{itemize}
            	\item $\schemeLocalFun{e}_{-3}=0$:
                    This is an equation we have already encountered in Section~\ref{sec:formal asymptotics:%
                    optimal profile cortex phase field}, and it reassures the optimal profile
                    $ \schemeLocalFun{\phaseFieldProto}_0(\projVar, \fastVar) = 
                    \tanh\left(\frac{\fastVar}{\sqrt{2}}\right) $.                  
                \item $\schemeLocalFun{e}_{-2}=0$: We use
                    \begin{equation}
                        \label{equ:formal asymtptotics:gradient expansion:em3}
                        0
                        =
                        -\simpleDeriv{\prime\prime}{
                          	\chemPotIdxPFWD{0}{\schemeLocalFun{\phaseFieldProto}}{}
                        }
                        \concat
                        \interfCoords[\phaseFieldParam]{}{}                        
                        +
                         \chemPotIdx{0}{\phaseFieldProto}{}
                        \doubleWellPot[\schemeSupscr{\prime\prime}]{\phaseFieldProto_0},        
                    \end{equation}

                    and compute 
                    \begin{equation}
                        \label{equ:formal asymptotics:inner expansion:consequ bound on L2 grad %
                        gen willm:finer scale analysis}
                        \begin{split}
                            -\simpleDeriv{\prime\prime}{
                                \chemPotIdxPFWD{0}{\phaseFieldProto}{}
                            }
                            &=
                            -\simpleDeriv{\prime\prime}{
                                \left(
                                \simpleDeriv{\prime}{\schemeLocalFun{\phaseFieldProto}_0}
                                \meanCurv[\hat]{}{}
                                \right)
                            }
                            +
                            \simpleDeriv{\prime\prime}{
                                \left(
                                    \simpleDeriv{\prime\prime}{
                                        \schemeLocalFun{\phaseFieldProto}_1
                                    }
                                    -
                                    \doubleWellPot[\schemeSupscr{\prime\prime}]{
                                        \schemeLocalFun{\phaseFieldProto}_0
                                    }
                                    \schemeLocalFun{\phaseFieldProto}_1
                                \right)
                            }
                            \\
                            &=
                            -\simpleDeriv{(3)}{\schemeLocalFun{\phaseFieldProto}_0}
                            \meanCurv[\hat]{}{}
                            -
                            2
                            \simpleDeriv{\prime\prime}{\schemeLocalFun{\phaseFieldProto}_0}
                            \simpleDeriv{\prime}{
                                \meanCurv[\hat]{}{}
                            }
                            -
                            \simpleDeriv{\prime}{\schemeLocalFun{\phaseFieldProto}_0}
                            \simpleDeriv{\prime\prime}{
                                \meanCurv[\hat]{}{}
                            }                            
                            +
                            \simpleDeriv{\prime\prime}{
                                \left(
	                              	\simpleDeriv{\prime\prime}{
                                        \schemeLocalFun{\phaseFieldProto}_1
                                    }
                                    -
                                    \doubleWellPot[\schemeSupscr{\prime\prime}]{
                                        \schemeLocalFun{\phaseFieldProto}_0
                                    }
                                    \schemeLocalFun{\phaseFieldProto}_1
                                \right)
                            }
                        \end{split}
                    \end{equation}
                    with
                    \begin{equation}
                        \label{equ:formal asymptotics:inner expansion:consequ bound on L2 grad %
                        gen willm:exp of mean curv deriv on -2}
                        \simpleDeriv{\prime}{
                          	\meanCurv[\hat]{}{}
                        }(\projVar,\fastVar)
                        =
                        \simpleDeriv{\prime}{
                            \extMeanCurv{
                                \projVar + \phaseFieldParam \fastVar \normal[S]{\projVar}
                            }
                        }    
                        =
                        \phaseFieldParam
                        \grad{}{}{}\extMeanCurv{}
                        \left(\projVar + \phaseFieldParam \fastVar \normal[S]{s}\right)
                        \cdot
                        \normal[S]{\projVar}
                    \end{equation}
                    and
                    \begin{equation}
                        \label{equ:formal asymptotics:inner expansion:consequ bound on L2 grad %
                        gen willm:exp of mean curv dderiv on -2}
                        \begin{split}
                            \simpleDeriv{\prime\prime}{
                              	\meanCurv[\hat]{}{}
                            }
                            (\projVar,\fastVar)
                            =
                            \simpleDeriv{\prime\prime}{
                                \extMeanCurv{
                                    \projVar + \phaseFieldParam \fastVar \normal[S]{\projVar}
                                }
                            }     
                            =
                            \phaseFieldParam^2
                            \grad{}{2}{}\extMeanCurv{
                            	\projVar + \phaseFieldParam \fastVar \normal[S]{s}
                            }
                            \frobProd
                            \normal[S]{\projVar}
                            \tensorProd
                            \normal[S]{\projVar}.               
                        \end{split}
                    \end{equation}
                    Passing the last two terms to lower scales, we obtain from
                    \eqref{equ:formal asymtptotics:gradient expansion:em3},
                    \begin{equation*}
                        -
                        \extMeanCurv{}
                        \simpleDeriv{(3)}{\schemeLocalFun{\phaseFieldProto}_0}
                        \concat
                        \interfCoords[\phaseFieldParam]{}{}
                        +
                        \simpleDeriv{\prime\prime}{
                            \left(
                                \simpleDeriv{\prime\prime}{
                                    \schemeLocalFun{\phaseFieldProto}_1
                                }
                                +
                                \doubleWellPot[\schemeSupscr{\prime\prime}]{
                                    \schemeLocalFun{\phaseFieldProto}_0
                                }
                                \schemeLocalFun{\phaseFieldProto}_1
                            \right)
                        }          
                        \concat
                        \interfCoords[\phaseFieldParam]{}{}
                        -
                        \left(
                            -
                            \extMeanCurv{}
                            \simpleDeriv{\prime}{\schemeLocalFun{\phaseFieldProto}_0}
                            \concat
                            \interfCoords[\phaseFieldParam]{}{}
                            +
                            \simpleDeriv{\prime\prime}{\schemeLocalFun{\phaseFieldProto}_1}
                            \concat
                            \interfCoords[\phaseFieldParam]{}{}
                            -
                            \doubleWellPot[\schemeSupscr{\prime\prime}]{
                                \phaseFieldProto_0
                            }
                            \phaseFieldProto_1
                        \right)
                        \doubleWellPot[\schemeSupscr{\prime\prime}]{
                          	\phaseFieldProto_{0}
                        }=0.
                    \end{equation*}
                    We note that from the optimal profile property
                    $
                        \simpleDeriv{\prime\prime}{\schemeLocalFun{\phaseFieldProto}_0}
                        \concat
                        \interfCoords[\phaseFieldParam]{}{}
                        -
                        \doubleWellPot[\schemeSupscr{\prime}]{\phaseFieldProto_0}
                        =0,
                    $
                    the relation
                    $
                    	\simpleDeriv{(3)}{
                            \schemeLocalFun{\phaseFieldProto}_0
                        }
                        =
                        \simpleDeriv{\prime}{\schemeLocalFun{\phaseFieldProto}_0}
                        \doubleWellPot[\schemeSupscr{\prime\prime}]{
                            \schemeLocalFun{\phaseFieldProto}_0
                        }
                    $ directly follows, so we may further simplify:
                    \begin{equation*}
                        \simpleDeriv{\prime\prime}{
                            \left(
                                \simpleDeriv{\prime\prime}{
                                    \schemeLocalFun{\phaseFieldProto}_1
                                }
                                +
                                \doubleWellPot[\schemeSupscr{\prime\prime}]{
                                    \schemeLocalFun{\phaseFieldProto}_0
                                }
                                \schemeLocalFun{\phaseFieldProto}_1
                            \right)
                        }          
                        \concat
                        \interfCoords[\phaseFieldParam]{}{}
                        -
                        \left(
                            \simpleDeriv{\prime\prime}{\schemeLocalFun{\phaseFieldProto}_1}
                            \concat
                            \interfCoords[\phaseFieldParam]{}{}
                            -
                            \doubleWellPot[\schemeSupscr{\prime\prime}]{
                                \phaseFieldProto_0
                            }
                            \phaseFieldProto_1
                        \right)
                        \doubleWellPot[\schemeSupscr{\prime\prime}]{
                            \phaseFieldProto_0
                        }=0,
                    \end{equation*}
                    which is solved by 
                    \begin{equation}
                        \label{equ:formal asymptotics:inner expansion:L2 grad gen willm:%
                        scale -2:result}
                    	\schemeLocalFun{\phaseFieldProto}_1 = 0.
                    \end{equation}
                \item $\schemeLocalFun{e}_{-1}=0$:
                	Equation
                    \begin{equation*}
                        \begin{split}
                            0
                            =
                            \extMeanCurv{}
                            \simpleDeriv{\prime}{
                                \chemPotIdxPFWD{0}{\schemeLocalFun{\phaseFieldProto}}{}
                            }
                            \concat
                            \interfCoords[\phaseFieldParam]{}{}
                            -
                            \simpleDeriv{\prime\prime}{
                                \chemPotIdxPFWD{1}{\schemeLocalFun{\phaseFieldProto}}{}
                            }
                            \concat
                            \interfCoords[\phaseFieldParam]{}{}
                            +
                            \chemPotIdx{1}{\phaseFieldProto}{}                    
                            \doubleWellPot[\schemeSupscr{\prime\prime}]{\phaseFieldProto_{0}}
                            -2
                            \simpleDeriv{\prime\prime}{\schemeLocalFun{\phaseFieldProto}_{0}}
                            \concat
                            \interfCoords[\phaseFieldParam]{}{}
                            \grad{}{}{}\extMeanCurv{}{}
                            \cdot
                            \extNormal{x}
                        \end{split}
                    \end{equation*}
                    is equivalent to
                    \begin{equation}
                        \label{equ:formal asymtptotics:inner expansion:consequ bound on L2 %
                        grad gen willm:scale -1 main}
                        \begin{split}
                            0
                            &=
                            \extMeanCurv{}{}^2
                            \simpleDeriv{\prime\prime}{\schemeLocalFun{\phaseFieldProto}_{0}}
                            \concat
                            \interfCoords[\phaseFieldParam]{}{}
                            -
                            \simpleDeriv{\prime\prime}{
                                \chemPotIdxPFWD{1}{\schemeLocalFun{\phaseFieldProto}}{}
                            }
                            \concat
                            \interfCoords[\phaseFieldParam]{}{}
                            +
                            \chemPotIdx{1}{\phaseFieldProto}{}                    
                            \doubleWellPot[\schemeSupscr{\prime\prime}]{\phaseFieldProto_{0}}
                            -2
                            \simpleDeriv{\prime\prime}{\schemeLocalFun{\phaseFieldProto}_{0}}
                            \concat
                            \interfCoords[\phaseFieldParam]{}{}
                            \grad{}{}{}\extMeanCurv{}
                            \cdot
                            \extNormal{x}
                            \\
                            &\quad+
                            \phaseFieldParam
                            \extMeanCurv{}{}
                            \simpleDeriv{\prime}{\schemeLocalFun{\phaseFieldProto}_{0}}
                            \concat
                            \interfCoords[\phaseFieldParam]{}{}   
                            \grad{}{}{}
                            \extMeanCurv{}{}
                            \cdot\extNormal{}
                        \end{split}
                    \end{equation}
                    using $ \schemeLocalFun{\phaseFieldProto}_1 = 0 $ so  
                    that $ \genGlChemPot{\phaseFieldProto}{0} = 
                    \simpleDeriv{\prime}{\schemeLocalFun{\phaseFieldProto}_0} 
                    \concat\interfCoords[\phaseFieldParam]{}{}
                    \extMeanCurv{} $.
                    We use Lemma~\ref{lemma:formal asymptotics:expansions}
                    abbreviating
                    $ \meanCurvPFWD{}{\projVar,0} \equalscolon \meanCurvPFWD{}{}\vert_{S}(\projVar) $:
                    \begin{equation}
                        \label{equ:formal asymtptotics:inner expansion:consequ bound on L2 %
                        grad gen willm:scale -1:exp of normal deriv mean curv}
                        \begin{split}
                            \grad[x]{}{}{\extMeanCurv{}{}}\cdot\extNormal{x}
                            &=
                            \left(
                                \sum_{i=1}^2
                                \pCurv[\hat]{i}{}{}^2
                                +
                                2\phaseFieldParam \fastVar
                                \pCurv[\hat]{i}{}{}^3
                                +
                                \landauBigo{\phaseFieldParam^2}
                            \right)
                            \concat
                            \interfCoords[\phaseFieldParam]{\NOARG}{\orthProj{S}{x}}
                            \\
                            &=
                            \left(
                                \meanCurvPFWD{}{}\vert_{S}^2 - 
                                2\gaussCurvPFWD{}{}\vert_{S}
                                +
                                2\phaseFieldParam\fastVar
                                \left(
                                    \meanCurvPFWD{}{}\vert_{S}^3
                                    -
                                    3\meanCurvPFWD{}{}\vert_{S}
                                    \gaussCurvPFWD{}{}\vert_{S}
                                \right)
                                +
                                \landauBigo{\phaseFieldParam^2}
                            \right)\concat\interfCoords[\phaseFieldParam]{\NOARG}{\orthProj{S}{x}}
                        \end{split}
                    \end{equation}
                    and
                    \begin{equation}
                        \label{equ:formal asymtptotics:inner expansion:consequ bound on L2 %
                        grad gen willm:scale -1:exp of mean curv squared}
                        \begin{split}
                            \extMeanCurv{x}^2
                            &=
                            \left(
                                \meanCurvPFWD{}{}\vert_{S}
                                +
                                \phaseFieldParam \fastVar
                                \left(
                                    \meanCurvPFWD{}{}\vert_{S}^2
                                    -
                                    2\gaussCurvPFWD{}{}\vert_{S}
                                \right)
                                +
                                \landauBigo{\phaseFieldParam^2}
                            \right)^2
                            \concat
                            \interfCoords[\phaseFieldParam]{\NOARG}{\orthProj{S}{x}}
                            \\
                            &=
                            \left(
                                \meanCurvPFWD{}{}\vert_{S}^2
                                +
                                2\phaseFieldParam 
                                \fastVar
                                \meanCurvPFWD{}{}\vert_{S}
                                \left(
                                    \meanCurvPFWD{}{}\vert_{S}^2
                                    -
                                    2\gaussCurvPFWD{}{}\vert_{S}
                                \right)                   
                                +
                                \landauBigo{\phaseFieldParam^2}
                            \right)
                            \concat
                            \interfCoords[\phaseFieldParam]{\NOARG}{\orthProj{S}{x}}.
                        \end{split}
                    \end{equation}
                    Passing all terms of lower order to the lower scales,
                    we obtain
                    \begin{equation}
                        \label{equ:formal asymptotics:expansion of the L2 gradient of the canham-helfrich %
                        energy:vanishing terms:order minus one}
                        \begin{split}
                            0
                            &=
                            \restrFun{\meanCurvPFWD{}{}}{S}^2
                            \simpleDeriv{\prime\prime}{\schemeLocalFun{\phaseFieldProto}_{0}}
                            -
                            \simpleDeriv{\prime\prime}{
                                \chemPotIdxPFWD{1}{\schemeLocalFun{\phaseFieldProto}}{}
                            }
                            +
                            \chemPotIdx{1}{\schemeLocalFun{\phaseFieldProto}}{}                    
                            \doubleWellPot[\schemeSupscr{\prime\prime}]{
                              	\schemeLocalFun{\phaseFieldProto}_{0}
                            }
                            -
                            2
                            \simpleDeriv{\prime\prime}{\schemeLocalFun{\phaseFieldProto}_{0}}
                            \left(
                                \meanCurvPFWD{}{}\vert_{S}^2 
                                - 
                                2
                                \gaussCurvPFWD{}{}\vert_{S}
                            \right)
                            \\
                            &=
                            -\simpleDeriv{\prime\prime}{
                                \chemPotIdxPFWD{1}{\schemeLocalFun{\phaseFieldProto}}{}
                            }
                            +
                            \chemPotIdx{1}{\schemeLocalFun{\phaseFieldProto}}{}                    
                            \doubleWellPot[\schemeSupscr{\prime\prime}]{
                              	\schemeLocalFun{\phaseFieldProto}_{0}
                            }
                            -
                            \simpleDeriv{\prime\prime}{\schemeLocalFun{\phaseFieldProto}_{0}}
                            \left(
                                \meanCurvPFWD{}{}\vert_{S}^2 
                                - 
                                4\gaussCurvPFWD{}{}\vert_{S}
                            \right).
                        \end{split}
                    \end{equation}
                    We make the ansatz
                    $ 
                    	\chemPotIdxPFWD{1}{\schemeLocalFun{\phaseFieldProto}}{}(\projVar,\fastVar) = 
                        -\left(
                            \meanCurvPFWD{}{}\vert_{S}^2 
                            - 4\gaussCurvPFWD{}{}\vert_{S}
                        \right)(\projVar)
                        s_1(\fastVar).
                    $ 
                    Then \eqref{equ:formal asymptotics:expansion of the L2 gradient of the canham-helfrich %
                    energy:vanishing terms:order minus one} becomes
                    \begin{equation*}
                        0 = 
                        \left(
                            \meanCurvPFWD{}{}\vert_{S}^2 
                            - 4\gaussCurvPFWD{}{}\vert_{S}
                        \right)
                        \left(
                            \simpleDeriv{\prime\prime}{
                   				s_1
                            }
                            -
                            s_1
                            \doubleWellPot[\schemeSupscr{\prime\prime}]{
                                \schemeLocalFun{\phaseFieldProto}_0
                            }
                            -
                            \simpleDeriv{\prime\prime}{\schemeLocalFun{\phaseFieldProto}_0}
                        \right).
                    \end{equation*}
                    Substituting $ \simpleDeriv{\prime\prime}{\schemeLocalFun{\phaseFieldProto}_0}  
                    = \doubleWellPot[\schemeSupscr{\prime}]{\schemeLocalFun{\phaseFieldProto}_0} $,
                    and solving for
                    \begin{equation*}
                        0 = 
                        \simpleDeriv{\prime\prime}{s_1}
                        -
                        s_1
                        \doubleWellPot[\schemeSupscr{\prime\prime}]{
                          \schemeLocalFun{\phaseFieldProto}_0
                        }
                        -
                        \simpleDeriv{\prime\prime}{\schemeLocalFun{\phaseFieldProto}_0}
                    \end{equation*}
                    gives
                    $ 
                    	s_1(\fastVar) 
                        = 
                        \frac{1}{2}
                        \simpleDeriv{\prime}{
                          	\schemeLocalFun{\phaseFieldProto}_{0}
                        }(\fastVar) \fastVar 
                    $ as in \cite[Theorem~2.13, Equation~(2.29)]{Wa2008},
                    so
                    \begin{equation}
                        \label{equ:formal asymptotics:expansion of the L2 gradient of the canham-helfrich %
                        energy:order minus one}
                        \chemPotIdxPFWD{1}{\schemeLocalFun{\phaseFieldProto}}{}
                        (\projVar,\fastVar) 
                        = 
                        -\frac{1}{2}
                        \left(
                            \meanCurvPFWD{}{}\vert_{S}^2 
                            - 4\gaussCurvPFWD{}{}\vert_{S}
                        \right)(\projVar)
                        \simpleDeriv{\prime}{\schemeLocalFun{\phaseFieldProto}_{0}}(\fastVar)
                        \fastVar.
                    \end{equation}
            \end{itemize}
	\subsection{Revisiting $\pfVelocityIFC{}{}$ and $\schemeLocalFun{\pfPressure{}{}}$ at leading order}
        The incompressibility \eqref{equ:modelling:phase field resulting PDEs:incompressibility}
        gives with 
        \eqref{equ:formally matched asymptotics:div rescaled distance}
        to leading order $ \phaseFieldParam^{-1} $
        \begin{equation}
            \label{equ:formal asymptotics:mass conservation navier-stokes:velocity:zero normal deriv}
            \simpleDeriv{\prime}{\left(\pfVelocityIFCIdx{0}{}{}\cdot\extNormal{}\right)} = 0.
        \end{equation}
        We have shown in Section~\ref{sec:Expanding the L2 gradient of the canham-helfrich energy}
        that $ \grad{\phaseFieldProto}{L^2}{}\freeEnergyFunct{}{} \in \landauBigo{1} $.
        This gives, by repeating the arguments used in Section~\ref{sec:formal asymptotics:%
        inner expansion:properties of vel to leading order}, $ \schemeLocalFun{\pfPressure{}{}}_{-2}
        = 0 $.
        Using 
        \eqref{equ:formally matched asymptotics:time derivative rescaled distance},
        \eqref{equ:formally matched asymptotics:jac rescaled distance}, and
        \eqref{equ:formally matched asymptotics:laplacian vec fun rescaled distance}, we compute
        for the inner expansion on $ \tubNeighbourhood{\ifCoordsDelta}{\interfOne{}}
        \cup \tubNeighbourhood{\ifCoordsDelta}{\interfTwo{}} $, $ S \in\{\interfOne{},\interfTwo{}\} $,
        \begin{align*}
            \pDiff{t}{}{}\pfVelocity{}{} 
            &= 
            \pDiff{t}{}{}\schemeInnerExp{\pfVelocity{}{}}_0
            -
            \phaseFieldParam^{-1}
            \simpleDeriv{\prime}{\pfVelocityIFCIdx{0}{}{}}
            \concat
            \interfCoords[\phaseFieldParam]{}{}
            \shapeTransfVelNormal^\surfaceProto
            +
            \simpleDeriv{\prime}{\pfVelocityIFCIdx{1}{}{}}
            \concat\interfCoords[\phaseFieldParam]{}{}
            +
            \landauBigo{\phaseFieldParam},
            \\
            \grad{}{}{}\pfVelocity{}{} 
            &=
            \phaseFieldParam^{-1}
            \simpleDeriv{\prime}{\pfVelocityIFCIdx{0}{}{}}
            \concat
            \interfCoords[\phaseFieldParam]{}{}
            \tensorProd
            \extNormal{}
            +
            \grad{\interfProto{}_{\sdist{}{}}}{}{}
            \schemeInnerExp{\pfVelocity{}{}}_0
            +
            \simpleDeriv{\prime}{\pfVelocityIFCIdx{1}{}{}}
            \concat\interfCoords[\phaseFieldParam]{}{}
            \tensorProd
            \extNormal{}
            +
            \landauBigo{\phaseFieldParam}
        \end{align*}
        resulting in
        \begin{equation*}
            (\pfVelocity{}{}\cdot\grad{}{}{})\pfVelocity{}{}
            =
            \phaseFieldParam^{-1}
            \simpleDeriv{\prime}{\pfVelocityIFCIdx{0}{}{}}
            \concat
            \interfCoords[\phaseFieldParam]{}{}
            \schemeInnerExp{\pfVelocity{}{}}_0
            \cdot
            \extNormal{}
            +
            \simpleDeriv{\prime}{\pfVelocityIFCIdx{0}{}{}}
            \concat
            \interfCoords[\phaseFieldParam]{}{}
            \schemeInnerExp{\pfVelocity{}{}}_1
            \cdot\extNormal{}
            +
            \left(
                \grad{\interfProto{}_{\sdist{}{}}}{}{}
                \schemeInnerExp{\pfVelocity{}{}}_0
                +
                \simpleDeriv{\prime}{\pfVelocityIFCIdx{1}{}{}}
                \concat\interfCoords[\phaseFieldParam]{}{}
                \tensorProd
                \extNormal{}                
            \right)
            \schemeInnerExp{\pfVelocity{}{}}_0
            +\landauBigo{\phaseFieldParam},
        \end{equation*}
        and
        \begin{equation*}
            \laplacian{}{}{}\schemeInnerExp{\pfVelocity{}{}}
            =
            \phaseFieldParam^{-2}\simpleDeriv{\prime\prime}{
            \pfVelocityIFCIdx{0}{}{}}
            \concat\interfCoords[\phaseFieldParam]{}{}
            +
            \phaseFieldParam^{-1}
            \left(
                -
                \simpleDeriv{\prime}{\pfVelocityIFCIdx{0}{}{}}
                \extMeanCurv{}
                +
                \simpleDeriv{\prime\prime}{\pfVelocityIFCIdx{1}{}{}}
            \right)
            +
            \laplacian{\surfaceProto_{\sdist{}{}}}{}{}
            \schemeInnerExp{\pfVelocity{}{}}_0
            -
            \simpleDeriv{\prime}{\pfVelocityIFCIdx{1}{}{}}
            \extMeanCurv{}
            +
            \simpleDeriv{\prime\prime}{\schemeLocalFun{\pfVelocity{}{}}_2}
            \concat
            \interfCoords[\phaseFieldParam]{}{}
            +
            \landauBigo{\phaseFieldParam}.
        \end{equation*}

        At leading order $ \phaseFieldParam^{-2} $ of
        \eqref{equ:modelling:phase field resulting PDEs:momentum balance},
        we thus find
        \begin{equation*}
            -\viscosity
            \simpleDeriv{\prime\prime}{\schemeLocalFun{\pfVelocity{}{}}_0}
            \concat \interfCoords[\phaseFieldParam]{}{}
            +
            \simpleDeriv{\prime}{\schemeLocalFun{\pfPressure{}{}}_{-1}}
            \concat \interfCoords[\phaseFieldParam]{}{}
            \extNormal{}
            =
            0.
        \end{equation*}
        Multiplication by $ \extNormal{} $ and using 
        \eqref{equ:formal asymptotics:mass conservation navier-stokes:velocity:zero normal deriv} 
        gives
        \begin{equation}
            \label{equ:formal asymptotics:composite expansion:pressure}
            \simpleDeriv{\prime}{\schemeLocalFun{\pfPressure{}{}}_{-1}}
            =
            0.
        \end{equation}
        By matching \eqref{equ:formal asymptotics:inner expansion:negative order z limit},
        $ \schemeLocalFun{\pfPressure{}{}}_{-1} = 0 $.
        Inserting back again, we obtain
        $
            \simpleDeriv{\prime\prime}{\schemeLocalFun{\pfVelocity{}{}}_0}
            =
            0.
        $
        Conclusively, $ \simpleDeriv{\prime}{\schemeLocalFun{\pfVelocity{}{}}_0} $ 
        is constant in $ \fastVar $. Matching with the outer expansion 
        $$ 
            \left(
            \lim_{\fastVar\nearrow\infty} 
            \schemeLocalFun{\pfVelocity{}{}}_0
            \right)
            \concat\interfCoords[\phaseFieldParam]{\NOARG}{x}
            = 
            \lim_{\delta\searrow0} 
            \schemeOuterExp{\pfVelocity{}{}}_0 (\orthProj{S}{x}+\delta\normal[S]{
            \orthProj{S}{x}})
        $$
        and
        $$ 
            \left(
            \lim_{\fastVar\searrow-\infty} 
            \schemeLocalFun{\pfVelocity{}{}}_0
            \right)
            \concat\interfCoords[\phaseFieldParam]{\NOARG}{x}
            = 
            \lim_{\delta\nearrow0} 
            \schemeOuterExp{\pfVelocity{}{}}_0 (\orthProj{S}{x}+\delta\normal[S]{
            \orthProj{S}{x}})
        $$
        indicates that $ \schemeLocalFun{\pfVelocity{}{}}_0 $ is bounded. Thus,
        it must hold 
        \begin{equation}
            \label{equ:formal asymptotics:sharp interface limit:composite expansion:%
            leading order normal deriv vel zero}
            \simpleDeriv{\prime}{\schemeLocalFun{\pfVelocity{}{}}_0} = 0.
        \end{equation}

\section{Sharp Interface limit}
    \label{sec:sharp interface limit}
    By inserting the expansions in interfacial coordinates 
    of the components of the solution of \eqref{equ:modelling:phase field:resulting PDEs} 
    into the systems' equations, we have managed to 
    \begin{itemize}
      \item eliminate the velocity expansion's summands
        up to (and including) order $ \phaseFieldParam^{-1} $, 
      \item eliminate the pressure expansion's summands
        up to (and including) order $ \phaseFieldParam^{-3} $,
      \item show that both phase fields assume the optimal profile
        at leading order,
      \item show that $ \phaseFieldProto_1 = 0 $, 
      \item and derive \eqref{equ:formal asymptotics:expansion of the L2 gradient of the canham-helfrich %
      energy:order minus one}.
    \end{itemize}
    Before we can make use of these findings and pass to the limit $ \phaseFieldParam \to 0 $, we compute
    the expansions of the remaining terms in $ K $ (see the right hand side of 
    \eqref{equ:modelling:phase field resulting PDEs:momentum balance}). 

    \subsection{Expansion of  $ \grad{}{L^2}{}\pfCouplingEnergy{}{}$ and remaining force terms}
        \label{sec:formal asymptotics:expanding the L2 gradient of the coupling energy}
        We compute the asymptotic expansions of
        $
            \grad{\membrPhaseField{}{}}{L^2}{}{}
            \pfCouplingEnergy{}{},
        $
        $ 
            \grad{\cortexPhaseField{}{}}{L^2}{}{}
            \pfCouplingEnergy{}{}
        $, 
        $$ 
        G_\phaseFieldParam\colonequals
        -\pDiff{\pfSpeciesOne{}{}}{}{}\pfCouplingIntMembr \meanCurv{\cortexPhaseFieldSym}{}
        \pfSpeciesOne{}{}\normal[\cortexPhaseFieldSym]{}
        =
        -
        \pfSpeciesOne{}{}
        \meanCurv{\cortexPhaseField{}{}}{}
        \integral{\domain}{}{
            \ginzburgLandauEnergyDens{\membrPhaseField{}{}}{}(y)
            \pDiff{\pfSpeciesOne{}{}}{}{}\pfCouplingEnergyDens{}{}(x,y,\pfSpeciesOne{}{},
            \normal[\cortexPhaseFieldSym]{})
        }{\lebesgueM{3}(y)}
        \normal[\cortexPhaseFieldSym]{},
        $$
        and
        $$
            H_\phaseFieldParam\colonequals
            -
            \ginzburgLandauEnergyDens{\cortexPhaseFieldSym}{}
            \orthProjMat{\normal[\cortexPhaseFieldSym]{}}{}
            \grad{}{}{}
            \pDiff{\pfSpeciesOne{}{}}{}{}\pfCouplingIntMembr
            \pfSpeciesOne{}{}
            =
            -\integral{\domain}{}{
                \ginzburgLandauEnergyDens{\membrPhaseField{}{}}{}(y)
                \orthProjMat{\normal[\cortexPhaseFieldSym]{}}{}
                \grad{x}{}{\pDiff{\pfSpeciesOne{}{}}{}{} \pfCouplingEnergyDens{}{}
                (\cdot,y,\pfSpeciesOne{}{},\normal[\cortexPhaseFieldSym]{})}
                \ginzburgLandauEnergyDens{\cortexPhaseFieldSym}{}
                \pfSpeciesOne{}{}
            }{\lebesgueM{3}(y)}.
        $$

        \paragraph{Expansion of $ \grad{\membrPhaseField{}{}}{L^2}{}\pfCouplingEnergy{}{} $}
        We recall from \eqref{equ:modelling:grad mempf coupling energy} that
        \begin{subequations}
            \label{equ:formal asymtptotics:expanding the L2 gradient of the coupling energy:%
            L2 gradient membrane}
            \begin{equation}
                \begin{split}
                    \grad{\membrPhaseField{}{}}{L^2}{}
                    \pfCouplingEnergy{}{}
                    &=
                    A_\phaseFieldParam + B_\phaseFieldParam
                \end{split}
            \end{equation}
            with
            \begin{equation}
                \begin{split}
                    A_\phaseFieldParam(y) &=
                    \left(
                        -\phaseFieldParam
                        \laplacian{y}{}{}\membrPhaseField{}{}
                        +
                        \phaseFieldParam^{-1}
                        \doubleWellPot[\schemeSupscr{\prime}]{\membrPhaseField{}{}}
                    \right)(y)
                    \integral{\tubNeighbourhood{\ifCoordsDelta}{\interfTwo{}}}{}{
                        \ginzburgLandauEnergyDens{\cortexPhaseField{}{}}{}(x)
                        \pfCouplingEnergyDens{x,y}{\pfSpeciesOne{}{},\normal[\cortexPhaseField{}{}]{}}
                    }{\lebesgueM{3}(x)}
                    +
                    \landauBigo{\phaseFieldParam},
                    \\
                    B_\phaseFieldParam(y) 
                    &=
                    -\phaseFieldParam
                    \grad{y}{}{}\membrPhaseField{}{}
                    \cdot
                    \integral{\tubNeighbourhood{\ifCoordsDelta}{\interfTwo{}}}{}{
                        \ginzburgLandauEnergyDens{\cortexPhaseField{}{}}{}(x)
                        \grad{y}{}{
                            \pfCouplingEnergyDens{x,y}{
                                \pfSpeciesOne{}{},\normal[\cortexPhaseField{}{}]{}
                            }
                        }
                    }{\lebesgueM{3}(x)}
                    +
                    \landauBigo{\phaseFieldParam}.
                \end{split}
            \end{equation}
        \end{subequations}
        We further expand
        \begin{equation*}
            \pfCouplingEnergyDens{x,y}{
                \pfSpeciesOneIdx{0}{}{} + \phaseFieldParam r_1,
                \normal[\cortexPhaseFieldIdx{0}{}{}+\phaseFieldParam s_1]{}
            }
            =
            \pfCouplingEnergyDens{x,y}{\pfSpeciesOneIdx{0}{}{},\normal[\cortexPhaseFieldIdx{0}{}{}]{}}
            +
            \phaseFieldParam
            \grad{\pfSpeciesOne{}{},\normal{}}{}{}
            \pfCouplingEnergyDens{}{}
            \cdot
            \transposed{\left(
                r_1,
                \diff[0]{\phaseFieldParam}{}{
                    \normal[\cortexPhaseFieldIdx{0}{}{}+\phaseFieldParam s_1]{}
                }
            \right)}
            +
            \landauBigo{\phaseFieldParam^2},
        \end{equation*}
        and note
        \begin{equation}
            \label{equ:formal asymptotics:inner expansion:consequ bound on coupl energy:%
            variation of normal}
            \diff[0]{\phaseFieldParam}{}{
                \normal[\cortexPhaseFieldIdx{0}{}{}+\phaseFieldParam s_1]{}
            }          
            =
            \frac{1}{\abs{\grad{}{}{}\cortexPhaseFieldIdx{0}{}{}}}
            \orthProjMat{\normal[\cortexPhaseFieldIdx{0}{}{}]{}}{} 
            \grad{}{}{}s_1
            \in \landauBigo{1}, 
        \end{equation}
        and thus
        \begin{equation*}
            \grad{\pfSpeciesOne{}{},\normal{}}{}{}
            \pfCouplingEnergyDens{}{}
            \cdot
            \transposed{
            \left(
                r_1,
                \diff[0]{\phaseFieldParam}{}{
                    \normal[\cortexPhaseFieldIdx{0}{}{}+\phaseFieldParam s_1]{}
                }
            \right)}       
            \in
            \landauBigo{1}.
        \end{equation*}
        By employing
        \eqref{equ:formal asymptotics:expansion of the L2 gradient of the coupling energy:%
        helping expansions:ginzburg landau density},
        we expand $ A_\phaseFieldParam $:
        \begin{equation*}
            \begin{split}
                &\integral{\tubNeighbourhood{\ifCoordsDelta}{\interfTwo{}}}{}{
                    \ginzburgLandauEnergyDens{\cortexPhaseField{}{}}{}(x)
                    \pfCouplingEnergyDens{x,y}{\pfSpeciesOne{}{},\normal[\cortexPhaseField{}{}]{}}
                }{\lebesgueM{3}(x)}            
                =
                \\
                &\integral{\tubNeighbourhood{\ifCoordsDelta}{\interfTwo{}}}{}{
                    \phaseFieldParam^{-1}
                    \left(
                        \frac{1}{2}
                        \left(\simpleDeriv{\prime}{\cortexPhaseFieldIFCIdx{0}{}{}}\right)^2
                        \concat
                        \interfCoords[\phaseFieldParam]{\NOARG}{x}
                        +
                        \doubleWellPot{\cortexPhaseFieldIdx{0}{\NOARG}{x}}
                    \right)
                    \pfCouplingEnergyDens{x,y}{\pfSpeciesOneIdx{0}{}{},\normal[\cortexPhaseFieldIdx{0}{}{}]{}}
                    +
                    \landauBigo{1}
                }{\lebesgueM{3}(x)}.
            \end{split}
        \end{equation*}
        Multiplication with 
        $ 
            \left(
                -\phaseFieldParam
                \laplacian{y}{}{}\membrPhaseField{}{}
                +
                \phaseFieldParam^{-1}
                \doubleWellPot[\schemeSupscr{\prime}]{\membrPhaseField{}{}}
            \right)
        $ yields
        \begin{equation*}
            \begin{split}
                A_\phaseFieldParam(y)
                &=
                \simpleDeriv{\prime}{\membrPhaseFieldIFCIdx{0}{}{}}
                \concat\interfCoords[\phaseFieldParam]{\NOARG}{y}
                \extMeanCurv{y}
                \integral{\tubNeighbourhood{\ifCoordsDelta}{\interfTwo{}}}{}{
                    \phaseFieldParam^{-1}
                    \left(
                        \frac{1}{2}
                        \left(\simpleDeriv{\prime}{\cortexPhaseFieldIFCIdx{0}{}{}}\right)^2
                        \concat
                        \interfCoords[\phaseFieldParam]{\NOARG}{x}
                        +
                        \doubleWellPot{\cortexPhaseFieldIdx{0}{\NOARG}{x}}
                    \right)
                    \pfCouplingEnergyDens{x,y}{\pfSpeciesOneIdx{0}{}{},\normal[\cortexPhaseFieldIdx{0}{}{}]{}}
                    +
                    \landauBigo{1}
                }{\lebesgueM{3}(x)}
            \end{split}
        \end{equation*}
        thanks to $ \cortexPhaseFieldIFCIdx{0}{}{} $ being the optimal profile,
        and
        \eqref{equ:formal asymptotics:inner expansion:L2 grad gen willm:scale -2:result}.

        In order to expand $ B_\phaseFieldParam $, we first compute
        \begin{equation*}
            \grad{y}{}{}\pfCouplingEnergyDens{x,y}{
                \pfSpeciesOneIdx{0}{}{}+\phaseFieldParam r_1,
                \normal[\cortexPhaseFieldIdx{0}{}{}+\phaseFieldParam s_1]{}
            }
            =
            \grad{y}{}{}\pfCouplingEnergyDens{x,y}{
                \pfSpeciesOneIdx{0}{}{},
                \normal[\cortexPhaseFieldIdx{0}{}{}]{}
            }          
            +
            \phaseFieldParam
            \grad{\pfSpeciesOne{}{},\normal{}}{}{}
            \grad{y}{}{}\pfCouplingEnergyDens{}{}
            \begin{pmatrix}
                r_1\\
                \frac{1}{\abs{\cortexPhaseFieldIdx{0}{}{}}}
                \orthProjMat{\normal[\cortexPhaseFieldIdx{0}{}{}]{}}{}
                \grad{}{}{}s_1
            \end{pmatrix}
            +
            \landauBigo{\phaseFieldParam^2}.
        \end{equation*}%
        Therefore, using 
        \eqref{equ:formal asymptotics:expansion of the L2 gradient of the coupling energy:%
            helping expansions:ginzburg landau density}, we find
        \begin{equation*}
            \begin{split}
                &\integral{\tubNeighbourhood{\ifCoordsDelta}{\interfTwo{}}}{}{
                    \ginzburgLandauEnergyDens{\cortexPhaseField{}{}}{}
                    \grad{y}{}{}
                    \pfCouplingEnergyDens{x,y}{\pfSpeciesOne{}{},\normal[\cortexPhaseField{}{}]{}}
                }{\lebesgueM{3}(x)}
                =
                \\
                &\integral{\tubNeighbourhood{\ifCoordsDelta}{\interfTwo{}}}{}{
                    \phaseFieldParam^{-1}
                    \left(
                        \frac{1}{2}
                        \left(\simpleDeriv{\prime}{\cortexPhaseFieldIFCIdx{0}{}{}}\right)^2
                        \concat
                        \interfCoords[\phaseFieldParam]{\NOARG}{x}
                        +
                        \doubleWellPot{\cortexPhaseFieldIdx{0}{\NOARG}{x}}
                    \right)
                    \grad{y}{}{}
                    \pfCouplingEnergyDens{x,y}{\pfSpeciesOneIdx{0}{}{},\normal[\cortexPhaseFieldIdx{0}{}{}]{}}
                    +
                    \landauBigo{1}
                }{\lebesgueM{3}(x)}.
            \end{split}
        \end{equation*}
        Multiplication with $ -\phaseFieldParam\grad{}{}{}\membrPhaseField{}{} $ gives
        \begin{equation*}
            B_\phaseFieldParam(y) =
            -\simpleDeriv{\prime}{\membrPhaseFieldIFCIdx{0}{}{}}
            \concat\interfCoords[\phaseFieldParam]{\NOARG}{y}
            \integral{\tubNeighbourhood{\ifCoordsDelta}{\interfTwo{}}}{}{
                \phaseFieldParam^{-1}
                \left(
                    \frac{1}{2}
                    \left(\simpleDeriv{\prime}{\cortexPhaseFieldIFCIdx{0}{}{}}\right)^2
                    \concat
                    \interfCoords[\phaseFieldParam]{\NOARG}{x}
                    +
                    \doubleWellPot{\cortexPhaseFieldIdx{0}{\NOARG}{x}}
                \right)
                \extNormal{y}\cdot
                \grad{y}{}{}
                \pfCouplingEnergyDens{x,y}{\pfSpeciesOneIdx{0}{}{},\normal[\cortexPhaseFieldIdx{0}{}{}]{}}
                +
                \landauBigo{1}
            }{\lebesgueM{3}(x)}.
        \end{equation*}
    \paragraph{Expansion of $\grad{\cortexPhaseField{}{}}{L^2}{}\pfCouplingEnergy{}{}$}
    We recall \eqref{equ:modelling:grad cortpf coupling energy},
    \begin{subequations}
        \label{equ:formal asymptotics:expanding the L2 gradient of the coupling energy:%
        cortex L2 gradient}
        \begin{equation}
            \grad{\cortexPhaseField{}{}}{L^2}{}
            \pfCouplingEnergy{}{}
            =
            C_\phaseFieldParam + D_\phaseFieldParam + E_\phaseFieldParam
        \end{equation}
        with
        \begin{equation}
            \begin{split}
                C_\phaseFieldParam(x) 
                &=
                \left(
                    -\phaseFieldParam
                    \laplacian{x}{}{}\cortexPhaseField{}{}
                    +
                    \phaseFieldParam^{-1}
                    \doubleWellPot[\schemeSupscr{\prime}]{\cortexPhaseField{}{}}
                \right)(x)
                \integral{\tubNeighbourhood{\ifCoordsDelta}{\interfOne{}}}{}{
                    \ginzburgLandauEnergyDens{\membrPhaseField{}{}}{}(y)
                    \pfCouplingEnergyDens{x,y}{\pfSpeciesOne{}{},\normal[\cortexPhaseField{}{}]{}}
                }{\lebesgueM{3}(y)},
                \\
                D_\phaseFieldParam(x) 
                &=
                -
                \phaseFieldParam
                \grad{x}{}{}\cortexPhaseField{}{}
                \cdot
                \integral{\tubNeighbourhood{\ifCoordsDelta}{\interfOne{}}}{}{
                    \ginzburgLandauEnergyDens{\membrPhaseField{}{}}{}(y)
                    \grad{x}{}{
                        \pfCouplingEnergyDens{x,y}{\pfSpeciesOne{}{},\normal[\cortexPhaseField{}{}]{}}
                    }            
                }{\lebesgueM{3}(y)},
                \\
                E_\phaseFieldParam(x) 
                &=
                -
                \integral{\tubNeighbourhood{\ifCoordsDelta}{\interfOne{}}}{}{
                    \ginzburgLandauEnergyDens{\membrPhaseField{}{}}{}(y)
                    \diver{x}{}{
                        \ginzburgLandauEnergyDens{\cortexPhaseField{}{}}{}
                        \transposed{\grad{\normal{}}{}{}\pfCouplingEnergyDens{}{}}
                        \frac{1}{\abs{\grad{}{}{}\cortexPhaseField{}{}}}
                        \orthProjMat{\normal[\cortexPhaseField{}{}]{}}{}
                    }
                }{\lebesgueM{3}(y)}.
            \end{split}
        \end{equation}
    \end{subequations}
    Before we start expanding these terms, we prove the following formulae:
    \begin{lem}
        \label{lemma:formal asymptotics:inner expansion:consequ bound on coupling energy L2 grad:%
        further expansions}
        It holds,
        \begin{equation}
            \label{equ:formal asymptotics:expansion of the L2 gradient of the coupling energy:%
            expansion of the normal}
            \frac{
                \grad{}{}{}\cortexPhaseField{}{}
            }{
                \abs{\grad{}{}{}\cortexPhaseField{}{}}
            }
            =
            \extNormal{}
            +
            \landauBigo{\phaseFieldParam^2},
        \end{equation}
        \begin{equation}
            \label{equ:formal asymptotics:expansion of the L2 gradient of the coupling energy:%
            expansion of the phase field hessian}
            \grad{}{2}{}\cortexPhaseField{}{}
            =
            \phaseFieldParam^{-2}
            \simpleDeriv{\prime\prime}{\cortexPhaseFieldIFCIdx{0}{}{}}
            \concat
            \interfCoords[\phaseFieldParam]{}{}
            \extNormal{}\tensorProd\extNormal{}
            +
            \phaseFieldParam^{-1}
            \simpleDeriv{\prime}{\cortexPhaseFieldIFCIdx{0}{}{}}
            \concat
            \interfCoords[\phaseFieldParam]{}{}
            \grad{}{}{}\extNormal{}
            +
            \simpleDeriv{\prime\prime}{\cortexPhaseFieldIFCIdx{2}{}{}}
            \concat
            \interfCoords[\phaseFieldParam]{}{}
            \extNormal{}\tensorProd\extNormal{}
            +
            \landauBigo{\phaseFieldParam}.
        \end{equation}
    \end{lem}
    \begin{proof}
        Ad~\eqref{equ:formal asymptotics:expansion of the L2 gradient of the coupling energy:%
        expansion of the normal}:
        Due to $ \cortexPhaseFieldIFCIdx{0}{}{} $ being the optimal profile and it thus being independent
        of the tangential variable $ \projVar $,
        and considering
        \eqref{equ:formal asymptotics:inner expansion:L2 grad gen willm:%
        scale -2:result}, we have
        \begin{equation*}
            \grad{}{}{}\cortexPhaseField{}{}
            =
            \phaseFieldParam^{-1}
            \simpleDeriv{\prime}{\cortexPhaseFieldIFCIdx{0}{}{}}\concat\interfCoords[\phaseFieldParam]{}{}
            \extNormal{}
            +
            \grad{\interfTwo{}_{\sdist{}{}}}{}{}
            \cortexPhaseFieldIFCIdx{0}{}{}
            +
            \simpleDeriv{\prime}{\cortexPhaseFieldIFCIdx{1}{}{}}\concat\interfCoords[\phaseFieldParam]{}{}
            \extNormal{}
            +
            \landauBigo{\phaseFieldParam}
            =
            \phaseFieldParam^{-1}
            \simpleDeriv{\prime}{\cortexPhaseFieldIFCIdx{0}{}{}}\concat\interfCoords[\phaseFieldParam]{}{}
            \extNormal{}                
            +
            \landauBigo{\phaseFieldParam}.
        \end{equation*}
        This observation brings the claimed expansion
        for the product of $ \grad{}{}{}\cortexPhaseField{}{} $ 
        and $ \abs{\grad{}{}{}\cortexPhaseField{}{}}^{-1} $
        using
        \eqref{equ:formal asymptotics:expanding L2 grad coupling energy cortex:inv grad phase field cortex}.

        Ad~\eqref{equ:formal asymptotics:expansion of the L2 gradient of the coupling energy:%
        expansion of the phase field hessian}:
        \begin{equation*}
            \begin{split}
                \grad{}{}{\pDiff{i}{}{}\cortexPhaseField{}{}}
                &=
                \phaseFieldParam^{-2}
                \simpleDeriv{\prime\prime}{\cortexPhaseFieldIFCIdx{0}{}{}}
                \concat\interfCoords[\phaseFieldParam]{}{}
                \extNormal{}\extNormal[i]{}
                +
                \phaseFieldParam^{-1}
                \grad{\interfTwo{}_{\sdist{}{}}}{}{
                    \simpleDeriv{\prime}{\cortexPhaseFieldIFCIdx{0}{}{}}
                    \concat\interfCoords[\phaseFieldParam]{}{}
                }
                \extNormal[i]{}
                +
                \phaseFieldParam^{-1}
                \simpleDeriv{\prime}{\cortexPhaseFieldIFCIdx{0}{}{}}                
                \concat
                \interfCoords[\phaseFieldParam]{}{}
                \grad{}{}{}\extNormal[i]{}
                \\
                &+
                \grad{\interfTwo{}_{\sdist{}{}}}{}{
                    \phaseFieldParam^{-1}
                    \simpleDeriv{\prime}{\cortexPhaseFieldIFCIdx{0}{}{}}
                    \concat
                    \interfCoords[\phaseFieldParam]{}{}
                    \extNormal[i]{}
                    +
                    \grad{\interfTwo{}_{\sdist{}{}}}{}{
                        \cortexPhaseFieldIdx{0}{}{}
                    }(i)
                }
                +
                \grad{}{}{
                    \simpleDeriv{\prime}{\cortexPhaseFieldIFCIdx{1}{}{}}
                    \concat
                    \interfCoords[\phaseFieldParam]{}{}
                    \grad{}{}{}\normal[i]{}
                    +
                    \phaseFieldParam
                    \grad{\interfTwo{}_{\sdist{}{}}}{}{}
                    \cortexPhaseFieldIFCIdx{1}{}{}(i)
                }
                \\
                &+
                \simpleDeriv{\prime\prime}{\cortexPhaseFieldIFCIdx{2}{}{}}
                \concat
                \interfCoords[\phaseFieldParam]{}{}
                \extNormal{}
                \extNormal[i]{}
                +
                \landauBigo{\phaseFieldParam}.
            \end{split}
        \end{equation*}
        We again use the optimal profile and
        \eqref{equ:formal asymptotics:inner expansion:L2 grad gen willm:scale -2:result}
        to conclude
        \begin{equation*}
            \phaseFieldParam^{-1}
            \grad{\interfTwo{}_{\sdist{}{}}}{}{
                \simpleDeriv{\prime}{\cortexPhaseFieldIFCIdx{0}{}{}}
                \concat\interfCoords[\phaseFieldParam]{}{}
            }
            \normal[i]{}
            =
            \grad{\interfTwo{}_{\sdist{}{}}}{}{
                \phaseFieldParam^{-1}
                \simpleDeriv{\prime}{\cortexPhaseFieldIFCIdx{0}{}{}}
                \concat
                \interfCoords[\phaseFieldParam]{}{}
                \normal[i]{}
                +
                \grad{\interfTwo{}_{\sdist{}{}}}{}{
                    \cortexPhaseFieldIdx{0}{}{}
                }(i)
            }
            =
            \grad{}{}{
                \simpleDeriv{\prime}{\cortexPhaseFieldIFCIdx{1}{}{}}
                \concat
                \interfCoords[\phaseFieldParam]{}{}
                \grad{}{}{}\normal[i]{}
                +
                \phaseFieldParam
                \grad{\interfTwo{}_{\sdist{}{}}}{}{}
                \cortexPhaseFieldIFCIdx{1}{}{}(i)
            }
            =
            0,
        \end{equation*}
        and the claim follows.
    \end{proof}

    $ C_\phaseFieldParam $ is expanded just like $ A_\phaseFieldParam $:
    \begin{equation*}
    \begin{split}
        C_\phaseFieldParam(x)
        &=
        \simpleDeriv{\prime}{\cortexPhaseFieldIFCIdx{0}{}{}}
        \concat\interfCoords[\phaseFieldParam]{\NOARG}{x}
        \extMeanCurv{x}
        \integral{\tubNeighbourhood{\ifCoordsDelta}{\interfTwo{}}}{}{
            \phaseFieldParam^{-1}
            \left(
                \frac{1}{2}
                \left(\simpleDeriv{\prime}{\membrPhaseFieldIFCIdx{0}{}{}}\right)^2
                \concat
                \interfCoords[\phaseFieldParam]{\NOARG}{x}
                +
                \doubleWellPot{\membrPhaseFieldIdx{0}{\NOARG}{x}}
            \right)
            \pfCouplingEnergyDens{x,y}{\pfSpeciesOneIdx{0}{}{},\normal[\cortexPhaseFieldIdx{0}{}{}]{}}
            +
            \landauBigo{1}
        }{\lebesgueM{3}(y)}.
    \end{split}
    \end{equation*}

    We continue by expanding $ D_\phaseFieldParam $: First note
    \begin{equation*}
        \begin{split}
            \phaseFieldParam
            \grad{x}{}{}\cortexPhaseField{}{}
            \cdot
            \grad{x}{}{
                \pfCouplingEnergyDens{\cdot,y}{\pfSpeciesOne{}{},\normal[\cortexPhaseField{}{}]{}}
            }
            =
            \phaseFieldParam
            \grad{x}{}{}\cortexPhaseField{}{}
            \cdot
            \left(
            \grad{x}{}{}\pfCouplingEnergyDens{}{}
            +
            \pDiff{\pfSpeciesOne{}{}}{}{}\pfCouplingEnergyDens{}{}
            \grad{x}{}{}\pfSpeciesOne{}{}
            +
            \transposed{
                \grad{}{}{}\normal[\cortexPhaseField{}{}]{}
            }
            \grad{\normal{}}{}{}\pfCouplingEnergyDens{}{}
            \right).
        \end{split}
    \end{equation*}
    Then we observe
    $
        \phaseFieldParam
        \transposed{\grad{x}{}{}\cortexPhaseField{}{}}
        \transposed{
            \grad{}{}{}\normal[\cortexPhaseField{}{}]{}
        } 
        =
        \landauBigo{\phaseFieldParam},
    $
    so
    \begin{equation*}
        \phaseFieldParam
        \grad{x}{}{}\cortexPhaseField{}{}
        \cdot
        \grad{x}{}{}\pfCouplingEnergyDens{}{}
        =
        (
            \simpleDeriv{\prime}{\cortexPhaseFieldIFCIdx{0}{}{}}
            \concat\interfCoords[\phaseFieldParam]{}{}
            \extNormal{}
            +
            \landauBigo{\phaseFieldParam}
        )
        \cdot
        \left(
            \grad{x}{}{}\pfCouplingEnergyDens{}{}
            +
            \pDiff{\pfSpeciesOne{}{}}{}{}
            \pfCouplingEnergyDens{}{}
            \grad{}{}{}\pfSpeciesOne{}{}
        \right)
        =
        \simpleDeriv{\prime}{\cortexPhaseFieldIFCIdx{0}{}{}}
        \concat\interfCoords[\phaseFieldParam]{}{}
        \extNormal{}
        \cdot
        \left(
            \grad{x}{}{}\pfCouplingEnergyDens{}{}
            +
            \pDiff{\pfSpeciesOne{}{}}{}{}
            \pfCouplingEnergyDens{}{}
            \grad{}{}{}\pfSpeciesOne{}{}
        \right)
        +
        \landauBigo{\phaseFieldParam},
    \end{equation*}
    where the last equality is justified by
    \eqref{equ:formal asymptotics:inner expansion:species to leading order:%
        constant in normal direction}.
    Then,
    \begin{equation*}
        \begin{split}
            D_\phaseFieldParam(x) 
            &=
            -
            \integral{\tubNeighbourhood{\ifCoordsDelta}{\interfOne{}}}{}{
                \phaseFieldParam^{-1}
                \left(
                    \frac{1}{2}
                    \left(\simpleDeriv{\prime}{\membrPhaseFieldIFCIdx{0}{}{}}\right)^2
                    \concat
                    \interfCoords[\phaseFieldParam]{\NOARG}{y}
                    +
                    \doubleWellPot{\membrPhaseFieldIdx{0}{\NOARG}{y}}
                \right)
                \left(
                \simpleDeriv{\prime}{\cortexPhaseFieldIFCIdx{0}{}{}}
                \concat\interfCoords[\phaseFieldParam]{}{}
                \extNormal{}
                \cdot
                \left(
                    \grad{x}{}{}\pfCouplingEnergyDens{}{}
                    +
                    \pDiff{\pfSpeciesOne{}{}}{}{}
                    \pfCouplingEnergyDens{}{}
                    \grad{}{}{}\pfSpeciesOne{}{}
                \right)
                +
                \landauBigo{\phaseFieldParam}
                \right)
                +
                \landauBigo{\phaseFieldParam}
            }{\lebesgueM{3}(y)}.
        \end{split}
    \end{equation*}

    At last, we turn to $ E_\phaseFieldParam $ and compute
    \begin{equation}
        \label{equ:formal asymptotics:consequ bound L2 grad coupling energy:dxdnc}
        \begin{split}
            \diver{x}{}{
                \ginzburgLandauEnergyDens{\cortexPhaseField{}{}}{}
                \transposed{\grad{\normal{}}{}{}\pfCouplingEnergyDens{}{}}
                \frac{1}{\abs{\grad{}{}{}\cortexPhaseField{}{}}}
                \orthProjMat{\normal[\cortexPhaseField{}{}]{}}{}
            }
            &=
            \grad{x}{}{
                \ginzburgLandauEnergyDens{\cortexPhaseField{}{}}{}
            }
            \cdot
            \frac{1}{\abs{\grad{}{}{}\cortexPhaseField{}{}}}
            \transposed{\orthProjMat{\normal[\cortexPhaseField{}{}]{}}{}}
            \grad{\normal{}}{}{}\pfCouplingEnergyDens{}{}
            +
            \ginzburgLandauEnergyDens{\cortexPhaseField{}{}}{}
            \diver{}{}{
                \frac{1}{\abs{\grad{}{}{}\cortexPhaseField{}{}}}
                \transposed{\orthProjMat{\normal[\cortexPhaseField{}{}]{}}{}}
                \grad{\normal{}}{}{}\pfCouplingEnergyDens{}{}
            }            
            \\
            &=
            \transposed{\grad{\normal{}}{}{}\pfCouplingEnergyDens{}{}}
            \frac{1}{\abs{\grad{}{}{}\cortexPhaseField{}{}}}
            \orthProjMat{\normal[\cortexPhaseField{}{}]{}}{}
            \grad{x}{}{
                \ginzburgLandauEnergyDens{\cortexPhaseField{}{}}{}
            }
            \\
            &+
            \ginzburgLandauEnergyDens{\cortexPhaseField{}{}}{}
            \left(
                \frac{1}{\abs{\grad{}{}{}\cortexPhaseField{}{}}}
                \grad{}{}{
                    \grad{\normal{}}{}{}\pfCouplingEnergyDens{}{}
                }
                \frobProd
                \orthProjMat{\normal[\cortexPhaseField{}{}]{}}{}
                +
                \grad{\normal{}}{}{}\pfCouplingEnergyDens{}{}
                \cdot
                \diver{}{}{
                    \frac{1}{\abs{\grad{}{}{}\cortexPhaseField{}{}}}
                    \orthProjMat{\normal[\cortexPhaseField{}{}]{}}{}
                }
            \right).
        \end{split}
    \end{equation}
    On the first term, we use 
    \eqref{equ:formal asymptotics:%
    expansion of the L2 gradient of the coupling energy:%
    expansion of the phase field hessian} and \eqref{equ:formal asymptotics:inner expansion:%
    L2 grad gen willm:scale -2:result} (for the expansion of the double well potential) to obtain
     \begin{equation*}
        \begin{split}
            \grad{}{}{\ginzburgLandauEnergyDens{\cortexPhaseField{}{}}{}}
            &=
            \left(
                \phaseFieldParam
                \grad{}{2}{}\cortexPhaseField{}{}
                +
                \phaseFieldParam^{-1}
                \doubleWellPot[\schemeSupscr{\prime}]{\cortexPhaseField{}{}}
            \right)
            \grad{}{}{}\cortexPhaseField{}{}
            \\
            &=
            \left(
                \phaseFieldParam^{-1}
                \simpleDeriv{\prime\prime}{
                    \cortexPhaseFieldIFCIdx{0}{}{}
                }
                \concat\interfCoords[\phaseFieldParam]{}{}
                \extNormal{}\tensorProd\extNormal{}
                +
                \simpleDeriv{\prime}{\cortexPhaseFieldIFCIdx{0}{}{}}
                \concat
                \interfCoords[\phaseFieldParam]{}{}
                \grad{}{}{}\extNormal{}
                +
                \phaseFieldParam^{-1}
                \doubleWellPot[\schemeSupscr{\prime}]{\cortexPhaseFieldIdx{0}{}{}}
                +
                \landauBigo{\phaseFieldParam}
            \right)
            \left(
                \phaseFieldParam^{-1}
                \simpleDeriv{\prime}{\cortexPhaseFieldIFCIdx{0}{}{}}
                \concat\interfCoords[\phaseFieldParam]{}{}
                \extNormal{}
                +
                \landauBigo{\phaseFieldParam}
            \right)
            \\
            &=
            \phaseFieldParam^{-2}
            \left(
                \simpleDeriv{\prime\prime}{
                    \cortexPhaseFieldIFCIdx{0}{}{}
                }
                \concat \interfCoords[\phaseFieldParam]{}{}
                \simpleDeriv{\prime}{\cortexPhaseFieldIFCIdx{0}{}{}}
                \concat \interfCoords[\phaseFieldParam]{}{}
                \extNormal{}
                +
                \doubleWellPot[\schemeSupscr{\prime}]{\cortexPhaseFieldIdx{0}{}{}}
                \simpleDeriv{\prime}{\cortexPhaseFieldIFCIdx{0}{}{}}
                \concat \interfCoords[\phaseFieldParam]{}{}
                \extNormal{}
            \right)
            +
            \landauBigo{1}.
        \end{split}
    \end{equation*}
    Since
    $ 
        \orthProjMat{\normal[\cortexPhaseField{}{}]{}}{}
        \extNormal{}
        \in
        \landauBigo{\phaseFieldParam^2}
    $ thanks to
    \eqref{equ:formal asymptotics:expansion of the L2 gradient of the coupling energy:%
        expansion of the normal}, we have 
    \begin{equation}
        \label{equ:formal asymptotics:consequ bound L2 grad coupling energy:dnc with proj %
        lower order}
        \begin{split}
            \transposed{\grad{\normal{}}{}{}\pfCouplingEnergyDens{}{}}
            \frac{1}{\abs{\grad{}{}{}\cortexPhaseField{}{}}}
            \orthProjMat{\normal[\cortexPhaseField{}{}]{}}{}
            \grad{}{}{\ginzburgLandauEnergyDens{\cortexPhaseField{}{}}{}}            
            \in \landauBigo{\phaseFieldParam}.
        \end{split}
    \end{equation}
    Second, we calculate
    \begin{equation*}
        \begin{split}
            \ginzburgLandauEnergyDens{\cortexPhaseField{}{}}{}
            \frac{1}{\abs{\grad{}{}{}\cortexPhaseField{}{}}}
            \grad{}{}{
                \grad{\normal{}}{}{}\pfCouplingEnergyDens{}{}
            }
            \frobProd
            \orthProjMat{\normal[\cortexPhaseField{}{}]{}}{}
            &=
            \left(
                \phaseFieldParam^{-1}
                \left(
                    \frac{1}{2}
                    \left(\simpleDeriv{\prime}{\cortexPhaseFieldIFCIdx{0}{}{}}\right)^2
                    \concat
                    \interfCoords[\phaseFieldParam]{}{}
                    +
                    \doubleWellPot{\cortexPhaseFieldIdx{0}{}{}}
                \right)
                +
                \landauBigo{\phaseFieldParam}
            \right)
            \frac{1}{\abs{\grad{}{}{}\cortexPhaseField{}{}}}
            \left(
                \grad{}{}{
                    \grad{\normal{}}{}{}\pfCouplingEnergyDens{}{}
                }
                \frobProd
                \orthProjMat{\extNormal{}}{}
                +
                \landauBigo{\phaseFieldParam^2}
            \right)
            \\
            &=
            \phaseFieldParam^{-1}
            \frac{1}{\abs{\grad{}{}{}\cortexPhaseField{}{}}}
            \left(
                \frac{1}{2}
                \left(\simpleDeriv{\prime}{\cortexPhaseFieldIFCIdx{0}{}{}}\right)^2
                \concat
                \interfCoords[\phaseFieldParam]{}{}
                +
                \doubleWellPot{\cortexPhaseFieldIdx{0}{}{}}
            \right)
            \diver{\interfTwo{}_{\sdist{}{x}}}{}{\grad{\normal{}}{}{}\pfCouplingEnergyDens{}{}}
            +
            \landauBigo{\phaseFieldParam^2}.
        \end{split}
    \end{equation*}

    Third, 
    \begin{equation*}
        \diver{}{}{
            \frac{1}{\abs{\grad{}{}{}\cortexPhaseField{}{}}}
            \orthProjMat{\normal[\cortexPhaseField{}{}]{}}{}
        }
        =
        \orthProjMat{\normal[\cortexPhaseField{}{}]{}}{}
        \grad{}{}{
            \frac{1}{\abs{\grad{}{}{}\cortexPhaseField{}{}}}                          	
        }
        +
        \frac{1}{\abs{\grad{}{}{}\cortexPhaseField{}{}}}            
        \diver{}{}{}
        \orthProjMat{\normal[\cortexPhaseField{}{}]{}}{}.
    \end{equation*}
    We further compute
    \begin{equation*}
        \grad{}{}{
            \frac{1}{\abs{\grad{}{}{}\cortexPhaseField{}{}}}
        }
        =
        -
        \frac{1}{\abs{\grad{}{}{}\cortexPhaseField{}{}}^3}
        \grad{}{2}{}\cortexPhaseField{}{}
        \grad{}{}{}\cortexPhaseField{}{}
        =
        -
        \frac{1}{\abs{\grad{}{}{}\cortexPhaseField{}{}}^2}
        \grad{}{2}{}\cortexPhaseField{}{}
        \extNormal{},
    \end{equation*}
    and using 
    \eqref{equ:formal asymptotics:expansion of the L2 gradient of the coupling energy:%
    expansion of the phase field hessian}, we find
    \begin{equation*}
        \grad{\normal{}}{}{}\pfCouplingEnergyDens{}{}
        \cdot
        \diver{}{}{
            \frac{1}{\abs{\grad{}{}{}\cortexPhaseField{}{}}}
            \orthProjMat{\normal[\cortexPhaseField{}{}]{}}{}
        }
        =
        \grad{\normal{}}{}{}\pfCouplingEnergyDens{}{}
        \cdot
        \orthProjMat{\normal[\cortexPhaseField{}{}]{}}{}
        \grad{}{}{
            \frac{1}{\abs{\grad{}{}{}\cortexPhaseField{}{}}}                          	
        }                
        +
        \grad{\normal{}}{}{}\pfCouplingEnergyDens{}{}
        \cdot
        \frac{1}{\abs{\grad{}{}{}\cortexPhaseField{}{}}}
        \diver{}{}{}
        \orthProjMat{\normal[\cortexPhaseField{}{}]{}}{}
        =
        \grad{\normal{}}{}{}\pfCouplingEnergyDens{}{}
        \cdot
        \frac{1}{\abs{\grad{}{}{}\cortexPhaseField{}{}}}
        \diver{}{}{}
        \orthProjMat{\normal[\cortexPhaseField{}{}]{}}{}                
        +\landauBigo{\phaseFieldParam^3}.
    \end{equation*}
    Finally,
    \begin{equation*}
        \diver{}{}{}\orthProjMat{\normal[\cortexPhaseField{}{}]{}}{}
        =
        -(
        \grad{}{}{}\extNormal{}
        \extNormal{}
        +
        \extNormal{}\diver{}{}{}\extNormal{}
        +
        \landauBigo{\phaseFieldParam}
        )
        =
        \extMeanCurv{}{}\extNormal{}
        +
        \landauBigo{\phaseFieldParam},
    \end{equation*}
    so
    \begin{equation*}
        \begin{split}
            E_\phaseFieldParam(x) &=
            -\frac{\phaseFieldParam^{-1}}{\abs{\grad{x}{}{}\cortexPhaseField{}{}}}
            \left(
                \frac{1}{2}
                \left(\simpleDeriv{\prime}{\cortexPhaseFieldIFCIdx{0}{}{}}\right)^2
                \concat
                \interfCoords[\phaseFieldParam]{\NOARG}{x}
                +
                \doubleWellPot{\cortexPhaseFieldIdx{0}{\NOARG}{x}}
            \right)
            \cdot
            \\
            &\cdot\integral{\tubNeighbourhood{\ifCoordsDelta}{\interfOne{}}}{}{
                \phaseFieldParam^{-1}
                \left(
                    \frac{1}{2}
                    \left(\simpleDeriv{\prime}{\membrPhaseFieldIFCIdx{0}{}{}}\right)^2
                    \concat
                    \interfCoords[\phaseFieldParam]{\NOARG}{y}
                    +
                    \doubleWellPot{\membrPhaseFieldIdx{0}{\NOARG}{y}}
                \right)
                \left(
                    \diver{\interfTwo{}_{\sdist{}{x}}}{}{\grad{\normal{}}{}{}\pfCouplingEnergyDens{}{}}
                    +
                    \grad{\normal{}}{}{}\pfCouplingEnergyDens{}{}
                    \cdot
                    \extMeanCurv{x}
                    \extNormal{x}
                \right)
            }{\lebesgueM{3}(y)}
            +\landauBigo{1}.
        \end{split}
    \end{equation*}

    \paragraph{Expansion of $ G_\phaseFieldParam $}
    We use 
    \eqref{equ:formal asymptotics:expansion of the L2 gradient of the coupling energy:%
    helping expansions:ginzburg landau density}, 
    \eqref{equ:formal asymptotics:expansion of the L2 gradient of the coupling energy:%
    helping expansions:phase field mean curv}, and
    \eqref{equ:formal asymptotics:expansion of the L2 gradient of the coupling energy:%
    expansion of the normal}
    in connection with \eqref{equ:formal asymptotics:inner expansion:L2 grad gen willm:%
    scale -2:result}
    to obtain
    \begin{align*}
        G_\phaseFieldParam =
        &-\integral{\domain}{}{
            \ginzburgLandauEnergyDens{\membrPhaseField{}{}}{}(y)
            \pfSpeciesOne{}{}
            \meanCurv{\cortexPhaseField{}{}}{}
            \pDiff{\pfSpeciesOne{}{}}{}{}\pfCouplingEnergyDens{
                \cdot,y,\pfSpeciesOne{}{},\normal[\cortexPhaseFieldSym]{}
            }{}
            \normal[\cortexPhaseFieldSym]{}
        }{\lebesgueM{3}(y)}
        =
        \\
        &-\phaseFieldParam^{-1}
        \pfSpeciesOneIdx{0}{}{}
        \left(
            \simpleDeriv{\prime}{\schemeLocalFun{\cortexPhaseFieldSym}_0}
        \right)^2
        \extMeanCurv{}
        \extNormal{}
        \integral{\tubNeighbourhood{\ifCoordsDelta}{\interfOne{}}}{}{
            \phaseFieldParam^{-1}
            \left(
            \frac{1}{2}\left(\simpleDeriv{\prime}{\membrPhaseFieldIFCIdx{0}{}{}}\right)^2
            \concat\interfCoords[\phaseFieldParam]{\NOARG}{y}
            +
            \doubleWellPot{\membrPhaseFieldIdx{0}{\NOARG}{y}}
            \right)
            \pDiff{\pfSpeciesOne{}{}}{}{}\pfCouplingEnergyDens{\cdot,y,\pfSpeciesOneIdx{0}{}{},
            \normal[\cortexPhaseFieldIdx{0}{}{}]{}}{}
        }{\lebesgueM{3}(y)}
        +
        \landauBigo{1}.
    \end{align*}
\paragraph{Expansion of $ H_\phaseFieldParam$}
    As before, we employ
    \eqref{equ:formal asymptotics:expansion of the L2 gradient of the coupling energy:%
    helping expansions:ginzburg landau density}, 
    \eqref{equ:formal asymptotics:expansion of the L2 gradient of the coupling energy:%
    expansion of the normal} to obtain
    \begin{align*}
        H_\phaseFieldParam =
        &-\integral{\domain}{}{
            \ginzburgLandauEnergyDens{\membrPhaseField{}{}}{}(y)
            \orthProjMat{\normal[\cortexPhaseFieldSym]{}}{}
            \grad{x}{}{\pDiff{\pfSpeciesOne{}{}}{}{} \pfCouplingEnergyDens{}{}
            (\cdot,y,\pfSpeciesOne{}{},\normal[\cortexPhaseFieldSym]{})}
            \ginzburgLandauEnergyDens{\cortexPhaseFieldSym}{}
            \pfSpeciesOne{}{}
        }{\lebesgueM{3}(y)}
        =
        -\phaseFieldParam^{-1}
        \left(
            \frac{1}{2}
            \left(
                \simpleDeriv{\prime}{\membrPhaseFieldIFCIdx{0}{}{}}
            \right)^2
            \concat\interfCoords[\phaseFieldParam]{}{}
            +
            \doubleWellPot{\cortexPhaseFieldIdx{0}{}{}}
        \right)\cdot
        \\
        &\cdot\integral{\domain}{}{
            \phaseFieldParam^{-1}
            \left(
                \frac{1}{2}
                \left(
                    \simpleDeriv{\prime}{\cortexPhaseFieldIFCIdx{0}{\NOARG}{}}
                \right)^2
                \concat\interfCoords[\phaseFieldParam]{\NOARG}{y}
                +
                \doubleWellPot{\membrPhaseFieldIdx{0}{\NOARG}{y}}
            \right)
            \grad{\interfTwo{}_{\sdist{}{}}}{}{\pDiff{\pfSpeciesOne{}{}}{}{} \pfCouplingEnergyDens{}{}
            (\cdot,y,\pfSpeciesOneIdx{0}{}{},\normal[\cortexPhaseFieldIdx{0}{}{}]{})}
            \pfSpeciesOneIdx{0}{}{}
        }{\lebesgueM{3}(y)}   
        +\landauBigo{1}.
    \end{align*}

    We now show, using the results of
    the analysis in the previous sections,
    that classical solutions of \eqref{equ:modelling:phase field:resulting PDEs}
    converge formally to solutions of \eqref{equ:modelling:sharp interface system}
    for $ \phaseFieldParam \searrow 0 $.
    In the following, we will often use that
    \begin{equation*}
        \left(\simpleDeriv{\prime}{\schemeLocalFun{\phaseFieldProto}_{0}}(\fastVar)\right)^2
        =
        \left(\simpleDeriv{\prime}{\cortexPhaseFieldIFCIdx{0}{}{}}(\fastVar)\right)^2
        =
        \left(
        \simpleDeriv{\prime}{
          	\tanh\left(
                \frac{\fastVar}{\sqrt{2}}
            \right)
        }
        \right)^2
        =
        \frac{1}{2}\left(1-\tanh\left(\frac{\fastVar}{\sqrt{2}}\right)^2\right)^2
    \end{equation*}
    is integrable, and we will abbreviate
    \begin{equation*}
        \tanhIntConst
        \colonequals
        \frac{1}{2}
        \integral{-\infty}{\infty}{
            \left(1-\tanh\left(\frac{\fastVar}{\sqrt{2}}\right)^2\right)^2
        }{\lebesgueM{1}(\fastVar)}.
    \end{equation*}
    We also partition
    all integrals over $ \domain $ into an interal over $ 
    \tubNeighbourhood{\ifCoordsDelta}{S} $ and one over $ \domain_\ifCoordsDelta = \domain \setminus \tubNeighbourhood{
      \ifCoordsDelta}{S} $.
    In the latter region, the outer expansions hold, and thus
    the integrands are all of lower order vanishing in the limit, so we can 
    neglect them.
    \subsection{Momentum balance and mass conservation}
    	\subsubsection{Outer region}
            At order $ \phaseFieldParam^0 $, we find with the results of 
            Section~\ref{sec:formal asymptotics:outer expansion:phase fields}
            (causing all the energy gradient terms
            on the right to vanish)
            \begin{equation*}
                \fluidMassDens\left(
                \pDiff{t}{}{}\pfVelocityOuter{}{}_0
                +
                \left(\pfVelocityOuter{}{}_0\cdot\grad{}{}{}\right)\pfVelocityOuter{}{}_0
                \right)
                -
                \diver{}{}{
                \viscosity
                \left(
                \grad{}{}{}\pfVelocityOuter{}{}_0
                +
                \transposed{\grad{}{}{}\pfVelocityOuter{}{}_0}
                \right)
                }
                +
                \grad{}{}{}\schemeOuterExp{\pfPressure{}{}}_0 = 0,
            \end{equation*}
            and for the incompressibility condition
            \begin{equation*}
                \diver{}{}{}\pfVelocityOuter{}{}_0
                = 0.
            \end{equation*}
            This gives \eqref{equ:modelling:sharp interface classical PDE model:momentum balance}
            and \eqref{equ:modelling:sharp interface classical PDE model:mass cons}.
        \subsubsection{Inner region}
            Let $ \interfProto{} \in \{\interfOne{},\interfTwo{}\} $.
            We note that the matching conditions for the velocity 
            state the no-jump conditions 
            \eqref{equ:modelling:sharp interface classical PDE model:no jump if one},
            \eqref{equ:modelling:sharp interface classical PDE model:no jump if two}.

            Plugging further the inner expansion into
            \eqref{equ:modelling:phase field resulting PDEs:momentum balance} and using
            \eqref{equ:formal asymptotics:composite expansion:pressure},
            \eqref{equ:formal asymptotics:sharp interface limit:composite expansion:%
            leading order normal deriv vel zero}, we find
            \begin{equation*}
                \begin{split}
                -\phaseFieldParam^{-1}\left(
                \simpleDeriv{\prime}{\left(
                    \hat{\viscosity}
                  	\simpleDeriv{\prime}{\schemeLocalFun{\pfVelocity{}{}}_{1}}
                    \right)
                }
                +
                \simpleDeriv{\prime}{
                \left(
                    \transposed{
                        \widehat{\grad{\surfaceProto_{\sdist{}{}}}{}{}
                        \pfVelocity{}{}_0}
                    }
                    \transposed{\hat{\viscosity}}
                \right)
                }
                \hat{\nu}
                +
                \hat{\nu}
                \tensorProd\simpleDeriv{\prime\prime}{\schemeLocalFun{\pfVelocity{}{}}_1}
                \hat{\nu}
                \right)
                \concat
                \interfCoords[\phaseFieldParam]{}{}
                +
                \phaseFieldParam^{-1}
                \simpleDeriv{\prime}{\schemeLocalFun{\pfPressure{}{}}_0}
                \concat\interfCoords[\phaseFieldParam]{}{}
                \extNormal{}
                +
                r
                =
                \\
                \phaseFieldParam^{-1}
                \left(
                    \grad{\membrPhaseFieldSym}{L^2}{}\freeEnergyFunct{}{}
                    \simpleDeriv{\prime}{\schemeLocalFun{\membrPhaseFieldSym}_0}
                    \concat\interfCoords[\phaseFieldParam]{}{}
                    +
                    \grad{\cortexPhaseFieldSym}{L^2}{}\freeEnergyFunct{}{}
                    \simpleDeriv{\prime}{\schemeLocalFun{\cortexPhaseFieldSym}_0}
                    \concat\interfCoords[\phaseFieldParam]{}{}
                \right)
                \extNormal{}
                -
                \pDiff{\pfSpeciesOne{}{}}{}{}\pfCouplingIntMembr
                \meanCurv{\schemeInnerExp{\cortexPhaseFieldSym}_0}{}
                \normal[\schemeInnerExp{\cortexPhaseFieldSym}_0]{}
                \schemeInnerExp{\pfSpeciesOneIdx{0}{}{}},
                \end{split}
            \end{equation*}
            where $ r = \hat{r}\concat\interfCoords[\phaseFieldParam]{}{} $ 
            with $ \schemeLocalFun{r} \in \landauBigo{1} $. 
            For understanding the limit 
            of this equation, let us consider its variational formulation
            with test functions $ \testFunVelocity{} \in [\sobolevHSet{1}{\domain}]^3 $.
            The left hand side then reads
            \begin{equation}
                \label{equ:formal asymptotics:sharp interface limit:momentum balance:%
                lhs integrated}
                \begin{split}
                    &\phaseFieldParam^{-1}
                    \integral{
                        \tubNeighbourhood{\ifCoordsDelta}{\interfOne{}}
                        \cup
                        \tubNeighbourhood{\ifCoordsDelta}{\interfTwo{}}
                    }{}{
                        \left(
                        -
                        \simpleDeriv{\prime}{\left(
                            \hat{\viscosity}
                            \simpleDeriv{\prime}{
                                \schemeLocalFun{\pfVelocity{}{}}_1
                            }
                        \right)}
                        \concat\interfCoords[\phaseFieldParam]{}{} 
                        +
                        \simpleDeriv{\prime}{
                          \left(
                            \transposed{
                              \widehat{\grad{\surfaceProto_{\sdist{}{}}}{}{}
                                \pfVelocity{}{}_0}
                            }
                            \transposed{\hat{\viscosity}}
                          \right)
                        }
                        \extNormal{}
                        \concat\interfCoords[\phaseFieldParam]{}{}\extNormal{}
                        +
                        \extNormal{}\tensorProd\simpleDeriv{\prime\prime}{\schemeLocalFun{\pfVelocity{}{}}_1}
                        \concat\interfCoords[\phaseFieldParam]{}{}\extNormal{}
                        +
                        \simpleDeriv{\prime}{\schemeLocalFun{\pfPressure{}{}}_0}
                        \concat\interfCoords[\phaseFieldParam]{}{} 
                        \extNormal{}
                        \right)\cdot\testFunVelocity{}
                    }{\lebesgueM{3}}
                    =
                    \\
                    &\integral{
                        -\frac{\ifCoordsDelta}{\phaseFieldParam}
                    }{
                        \frac{\ifCoordsDelta}{\phaseFieldParam}
                    }{
                        \integral{
                            \interfOne{}_{\phaseFieldParam\fastVar} 
                            \cup 
                            \interfTwo{}_{\phaseFieldParam\fastVar}
                        }{}{
                            \left(
                            \left(
                            -
                            \simpleDeriv{\prime}{\left(
                                \hat{\viscosity}
                                \simpleDeriv{\prime}{
                                    \schemeLocalFun{\pfVelocity{}{}}_1
                                }
                            \right)}
                            +
                            \simpleDeriv{\prime}{
                              \left(
                                \transposed{
                                  \widehat{\grad{\surfaceProto_{\phaseFieldParam\fastVar}}{}{}
                                    \pfVelocity{}{}_0}
                                }
                                \transposed{\hat{\viscosity}}
                              \right)
                            }
                            \hat{\nu}
                            +
                            \hat{\nu}
                            \tensorProd\simpleDeriv{\prime\prime}{\schemeLocalFun{\pfVelocity{}{}}_1}
                            \hat{\nu}
                            \right)
                            (\orthProj{\interfTwo{}\cup\interfOne{}}{\sigma},\fastVar)
                            +
                            \simpleDeriv{\prime}{\schemeLocalFun{\pfPressure{}{}}_0}
                            (\orthProj{\interfTwo{}\cup\interfOne{}}{\sigma},\fastVar)
                            \extNormal{\sigma}
                            \right)\cdot\testFunVelocity{}
                        }{\hausdorffM{2}(\sigma)}
                    }{\fastVar}
                    =
                    \\
                    &\integral{
                        \interfOne{}_{\phaseFieldParam\fastVar} 
                        \cup 
                        \interfTwo{}_{\phaseFieldParam\fastVar}
                    }{}{
                        \left(
                        \left[
                          	-\schemeLocalFun{\viscosity}
                          	\simpleDeriv{\prime}{\schemeLocalFun{\pfVelocity{}{}}_1}
                            +
                            \left(
                              \transposed{
                                \widehat{\grad{\surfaceProto_{\phaseFieldParam\fastVar}}{}{}
                                  \pfVelocity{}{}_0}
                              }
                              \transposed{\hat{\viscosity}}
                            \right)
                            \hat{\nu}
	                        +
                            \hat{\nu}
	                        \tensorProd
	                        \simpleDeriv{\prime}{\schemeLocalFun{\pfVelocity{}{}}_1}
                            \hat{\nu}
                        \right]_{-\frac{\ifCoordsDelta}{\phaseFieldParam}}^{\frac{\ifCoordsDelta}{\phaseFieldParam}}
                        (\orthProj{\interfTwo{}\cup\interfOne{}}{\sigma})
                        +
                        \left[
                            \schemeLocalFun{\pfPressure{}{}}_0
                        \right]_{-\frac{\ifCoordsDelta}{\phaseFieldParam}}^{\frac{\ifCoordsDelta}{\phaseFieldParam}}
                        (\orthProj{\interfTwo{}\cup\interfOne{}}{\sigma})
                        \extNormal{\sigma}
                        \right)\cdot\testFunVelocity{}
                    }{\hausdorffM{2}(\sigma)}.
                \end{split}
            \end{equation}
            We can rewrite the integral of
            $
                -\schemeLocalFun{\viscosity}
                \simpleDeriv{\prime}{\schemeLocalFun{\pfVelocity{}{}}_1}
                +
                \left(
                  \transposed{
                    \widehat{\grad{\surfaceProto_{\phaseFieldParam\fastVar}}{}{}
                      \pfVelocity{}{}_0}
                  }
                  \transposed{\hat{\viscosity}}
                \right)
                \hat{\nu}
                +
                \hat{\nu}
                \tensorProd
                \simpleDeriv{\prime}{\schemeLocalFun{\pfVelocity{}{}}_1}
                \hat{\nu}
            $ 
            in $ \fastVar $ 
            by looking at the expansions of 
            $
            	\grad{}{}{}\pfVelocity{}{}\extNormal{}
            $ 
            and of
            $
            	\transposed{\grad{}{}{}\pfVelocity{}{}}\extNormal{}
            $ 
            in interfacial coordinates:
            \begin{equation*}
                \begin{split}
            	\grad{}{}{
                  	\schemeLocalFun{\pfVelocity{}{}}\concat\interfCoords[\phaseFieldParam]{}{}
                }
                \extNormal{}
                &=
                \phaseFieldParam^{-1}\simpleDeriv{\prime}{\schemeLocalFun{\pfVelocity{}{}}_0}
                \concat\interfCoords[\phaseFieldParam]{}{}
                +
                \simpleDeriv{\prime}{\schemeLocalFun{\pfVelocity{}{}}_1}
                \concat\interfCoords[\phaseFieldParam]{}{}
                +
                \landauBigo{\phaseFieldParam}
                \\
                &=
                \simpleDeriv{\prime}{\schemeLocalFun{\pfVelocity{}{}}_1}
                \concat\interfCoords[\phaseFieldParam]{}{}
                +
                \landauBigo{\phaseFieldParam},
                \end{split}
            \end{equation*}
            and
            \begin{equation*}
                \begin{split}
                \transposed{
            	\grad{}{}{
                  	\schemeLocalFun{\pfVelocity{}{}}\concat\interfCoords[\phaseFieldParam]{}{}
                }}
                \extNormal{}
                &=
                \phaseFieldParam^{-1}
                \extNormal{}
                \tensorProd
                \simpleDeriv{\prime}{\schemeLocalFun{\pfVelocity{}{}}_0}
                \concat\interfCoords[\phaseFieldParam]{}{}
                \extNormal{}
                +
                \transposed{\grad{\surfaceProto_{\sdist{}{}}}{}{}\pfVelocity{}{}_0}
                \extNormal{}
                +
                \extNormal{}
                \tensorProd
                \simpleDeriv{\prime}{\schemeLocalFun{\pfVelocity{}{}}_1}
                \concat\interfCoords[\phaseFieldParam]{}{}
                \extNormal{}
                +
                \landauBigo{\phaseFieldParam}
                \\
                &=
                \transposed{\grad{\surfaceProto_{\sdist{}{}}}{}{}\pfVelocity{}{}_0}
                \extNormal{}
                +
                \extNormal{}
                \tensorProd
                \simpleDeriv{\prime}{\schemeLocalFun{\pfVelocity{}{}}_1}
                \concat\interfCoords[\phaseFieldParam]{}{}
                \extNormal{}
                +
                \landauBigo{\phaseFieldParam},
                \end{split}
            \end{equation*}
            respectively
            (following with 
            \eqref{equ:formally matched asymptotics:jac rescaled distance}
            and 
            \eqref{equ:formal asymptotics:sharp interface limit:composite expansion:%
            leading order normal deriv vel zero}).
            With the
            matching conditions for
            $ \grad{}{}{}\pfVelocity{}{}\extNormal{} $ and $ \transposed{\grad{}{}{}\pfVelocity{}{}}
            \extNormal{} $,
            we further obtain
            \begin{equation}
                \label{equ:formal asymptotics:sharp interface limit:momentum balance:matching normal %
                deriv velocity}
                \begin{split}
            	\lim_{\alpha\searrow0}
                (
                \grad{}{}{}\pfVelocityIdx{0}{\NOARG}{}
                \extNormal{}                
                )
                (\orthProj{S}{x}+
                \alpha\normal[S]{\orthProj{S}{x}})
                &=
                (\lim_{\fastVar\nearrow\infty}
                \simpleDeriv{\prime}{\schemeLocalFun{\pfVelocity{}{}}_1})
                \concat\interfCoords[\phaseFieldParam]{\NOARG}{x},
                \\
            	\lim_{\alpha\searrow0}
                (
                \transposed{\grad{}{}{}\pfVelocityIdx{0}{\NOARG}{}}
                \extNormal{}                
                )
                (\orthProj{S}{x}+
                \alpha\normal[S]{\orthProj{S}{x}})
                &=         
                \lim_{\fastVar\nearrow\infty}(
                \transposed{
                \widehat{
                \grad{\surfaceProto_{\sdist{}{}}}{}{}\pfVelocity{}{}_0
                }})
                \concat\interfCoords[\phaseFieldParam]{}{}
                \extNormal{}
                +
                \extNormal{}
                \tensorProd
                \lim_{\fastVar\nearrow\infty}(
                \simpleDeriv{\prime}{\schemeLocalFun{\pfVelocity{}{}}_1})
                \concat\interfCoords[\phaseFieldParam]{}{}
                \extNormal{}.
                \end{split}
            \end{equation}
            The computations go analogously for the limits $ \alpha \nearrow 0 $ and $ \fastVar \searrow -\infty $.
            Together, both limits form a jump $ \jump{\cdot} $.
            Now we pass $ \phaseFieldParam \searrow 0 $ in \eqref{equ:formal asymptotics:%
            sharp interface limit:momentum balance:lhs integrated} and insert
            the matching condition for the pressure and \eqref{equ:formal asymptotics:sharp interface limit:%
            momentum balance:matching normal deriv velocity}. This reveals
            \begin{equation}
                \label{equ:formal asymptotics:sharp interface limit:momentum balance:%
                recov interfacial jump stress tensor}
                \begin{split}
                \phaseFieldParam^{-1}
                &\integral{
                    \tubNeighbourhood{\ifCoordsDelta}{\interfOne{}}
                    \cup
                    \tubNeighbourhood{\ifCoordsDelta}{\interfTwo{}}
                }{}{
                    \left(
                    -
                    \simpleDeriv{\prime}{\left(
                        \hat{\viscosity}
                        \simpleDeriv{\prime}{
                            \schemeLocalFun{\pfVelocity{}{}}_1
                        }
                    \right)}
                    \concat\interfCoords[\phaseFieldParam]{}{} 
                    +
                    \simpleDeriv{\prime}{
                      \left(
                        \transposed{
                          \widehat{\grad{\surfaceProto_{\sdist{}{}}}{}{}
                            \pfVelocity{}{}_0}
                        }
                        \transposed{\hat{\viscosity}}
                      \right)
                    }
                    \extNormal{}
                    \concat\interfCoords[\phaseFieldParam]{}{}\extNormal{}
                    +
                    \extNormal{}\tensorProd\simpleDeriv{\prime\prime}{\schemeLocalFun{\pfVelocity{}{}}_1}
                    \concat\interfCoords[\phaseFieldParam]{}{}\extNormal{}
                    +
                    \simpleDeriv{\prime}{\schemeLocalFun{\pfPressure{}{}}_0}
                    \concat\interfCoords[\phaseFieldParam]{}{} 
                    \extNormal{}
                    \right)\cdot\testFunVelocity{}
                }{\lebesgueM{3}}
                \specConv{\phaseFieldParam\searrow0}
                \\
                &\integral{\interfOne{}\cup\interfTwo{}}{}{
                  	-\jump{
                      	\viscosity(\grad{}{}{}\pfVelocityIdx{0}{}{}+
                        \transposed{\grad{}{}{}\pfVelocityIdx{0}{}{}})
                        -
                        \pfPressureIdx{0}{}{}
                    }
                    \normal{}
                    \cdot\testFunVelocity{}
                }{\hausdorffM{2}},
                \end{split}
            \end{equation}
            where $ \normal{}\in\{\normal[\interfOne{}]{},\normal[\interfTwo{}]{}\} $ depending 
            on what surface the integrand is to be understood.
            
            The right hand side of \eqref{equ:modelling:phase field resulting PDEs:momentum balance}
            in variational form is
            \begin{equation}
                \label{equ:formal asymptotics:singular limit:navier-stokes right hand side}
                \begin{split}
                    f(\testFunVelocity{})
                    &=
                    \LTProd{\tubNeighbourhood{\ifCoordsDelta}{\interfOne{}}}{
                        \grad{\membrPhaseField{}{}}{L^2}{}\pfGenWillmoreEnergyFunct{\membrPhaseFieldSym}{}
                        +
                        \grad{\membrPhaseField{}{}}{L^2}{}\ginzburgLandauEnergyFunct{\membrPhaseFieldSym}{}
                        +
                        \grad{\membrPhaseField{}{}}{L^2}{}
                        \pfCouplingEnergy{\membrPhaseFieldSym,\cortexPhaseFieldSym,\pfSpeciesOne{}{}}{}
                    }{
                        \grad{}{}{}\membrPhaseField{}{}\cdot\testFunVelocity{}
                    }
                    \\
                    &+
                    \LTProd{\tubNeighbourhood{\ifCoordsDelta}{\interfTwo{}}}{
                        \grad{\cortexPhaseField{}{}}{L^2}{}\pfGenWillmoreEnergyFunct{\cortexPhaseFieldSym}{}
                        +
                        \grad{\cortexPhaseField{}{}}{L^2}{}\ginzburgLandauEnergyFunct{\cortexPhaseFieldSym}{}
                        +
                        \grad{\cortexPhaseField{}{}}{L^2}{}
                        \pfCouplingEnergy{\membrPhaseFieldSym,\cortexPhaseFieldSym,\pfSpeciesOne{}{}}{}
                    }{
                        \grad{}{}{}\cortexPhaseField{}{}\cdot\testFunVelocity{}
                    }
                    \\
                    &-
                    \VLTProd{\tubNeighbourhood{\ifCoordsDelta}{\interfTwo{}}}{3}{
                    \integral{\domain}{}{
                        \ginzburgLandauEnergyDens{\membrPhaseField{}{}}{}(y)
                        \orthProjMat{\normal[\cortexPhaseFieldSym]{}}{}
                        \grad{}{}{\pDiff{\pfSpeciesOne{}{}}{}{} \pfCouplingEnergyDens{}{}
                        (\cdot,y,\pfSpeciesOne{}{},\normal[\cortexPhaseFieldSym]{})
                        }
                        \ginzburgLandauEnergyDens{\cortexPhaseFieldSym}{}
                        \pfSpeciesOne{}{}
                    }{\lebesgueM{3}(y)}
                    }{\testFunVelocity{}{}}
                    \\
                    &-
                    \LTProd{\tubNeighbourhood{\ifCoordsDelta}{\interfTwo{}}}{
                        \integral{\tubNeighbourhood{\ifCoordsDelta}{\interfOne{}}}{}{
                            \ginzburgLandauEnergyDens{\membrPhaseField{}{}}{}(y)
                            \pfSpeciesOne{}{}
                            \meanCurv{\cortexPhaseField{}{}}{}
                            \pDiff{\pfSpeciesOne{}{}}{}{}\pfCouplingEnergyDens{}{}
                        }{\lebesgueM{3}(y)}
                    }{
                        \testFunVelocity{}\cdot\normal[\cortexPhaseField{}{}]{}
                    }.
                \end{split}
            \end{equation}
            Note that we can rewrite
            \begin{equation*}
                -
                \LTProd{\tubNeighbourhood{\ifCoordsDelta}{\interfTwo{}}}{
                    \integral{\tubNeighbourhood{\ifCoordsDelta}{\interfOne{}}}{}{
                        \ginzburgLandauEnergyDens{\membrPhaseField{}{}}{}(y)
                        \pfSpeciesOne{}{}
                        \meanCurv{\cortexPhaseField{}{}}{}
                        \pDiff{\pfSpeciesOne{}{}}{}{}\pfCouplingEnergyDens{}{}
                    }{\lebesgueM{3}(y)}
                }{
                    \testFunVelocity{}\cdot\normal[\cortexPhaseField{}{}]{}
                }
                =
                -
                \LTProd{\tubNeighbourhood{\ifCoordsDelta}{\interfTwo{}}}{
                    \pDiff{\pfSpeciesOne{}{}{}}{}{}\pfCouplingIntMembr 
                    \meanCurv{\cortexPhaseFieldSym}{}
                    \pfSpeciesOne{}{}
                }{
                    \testFunVelocity{}\cdot\normal[\cortexPhaseField{}{}]{}
                }
            \end{equation*}
            and
            \begin{align*}
                -
                \VLTProd{\tubNeighbourhood{\ifCoordsDelta}{\interfTwo{}}}{3}{
                    \integral{\domain}{}{
                        \ginzburgLandauEnergyDens{\membrPhaseField{}{}}{}(y)
                        \orthProjMat{\normal[\cortexPhaseFieldSym]{}}{}
                        \grad{}{}{\pDiff{\pfSpeciesOne{}{}}{}{} \pfCouplingEnergyDens{}{}
                        (\cdot,y,\pfSpeciesOne{}{},\normal[\cortexPhaseFieldSym]{})
                        }
                        \ginzburgLandauEnergyDens{\cortexPhaseFieldSym}{}
                        \pfSpeciesOne{}{}
                    }{\lebesgueM{3}(y)}
                }{\testFunVelocity{}{}}
                =
                \\
                -
                \VLTProd{\tubNeighbourhood{\ifCoordsDelta}{\interfTwo{}}}{3}{
                  	\ginzburgLandauEnergyDens{\cortexPhaseFieldSym}{}
                    \orthProjMat{\normal[\cortexPhaseFieldSym]{}}{}
                    \grad{}{}{}
                    \pDiff{\pfSpeciesOne{}{}}{}{}\pfCouplingIntMembr
                    \pfSpeciesOne{}{}
                }{
                  	\testFunVelocity{}
                },
            \end{align*}
            which we use in the following as abbreviation.

            To pass \eqref{equ:formal asymptotics:singular limit:navier-stokes right hand side}
            to the limit, we treat the gradients of the energies separately.
            The gradients of $ \pfGenWillmoreEnergyFunct{}{} $ and $ \ginzburgLandauEnergyFunct{}{} $
            have the same structure for both $ \membrPhaseFieldSym $ and $ \cortexPhaseFieldSym $, and 
            can therefore be treated verbatim.
            For $ \pfCouplingEnergy{}{} $ we distinguish
            the derivatives w.r.t. $ \membrPhaseFieldSym $ and $ \cortexPhaseFieldSym $.
            Let us start our analysis with $ \grad{\phaseFieldProto}{L^2}{}\pfGenWillmoreEnergyFunct{}{} $.

		\subsubsection{Force terms of $ \grad{\phaseFieldProto}{L^2}{}\pfGenWillmoreEnergyFunct{}{} $}
            In this section we show the following lemma:
            \begin{lem}
                \label{lemma:formal asymptotics:singular limit:canham-helfrich force terms:approximation}
                Let $ S \in \{\interfOne{},\interfTwo{}\} $ and $ \phaseFieldProto
                \in \{\membrPhaseFieldSym,\cortexPhaseFieldSym\} $ such that $ S $ 
                is the boundary layer for $ \phaseFieldProto $. The following limit holds true:
                \begin{equation*}
                    \bendRig
                    \integral{\tubNeighbourhood{\ifCoordsDelta}{\surfaceProto}}{}{
                        \left(
                            -\laplacian{}{}{ \chemPot{\phaseFieldProto}{} }	
                            +
                            \chemPot{\phaseFieldProto}{}
                            \phaseFieldParam^{-2}
                            \doubleWellPot[\schemeSupscr{\prime\prime}]{\phaseFieldProto}{}
                        \right)
                        \grad{}{}{}\phaseFieldProto
                        \cdot
                        \testFunVelocity{}
                    }{\lebesgueM{3}}            
                    \specConv{\phaseFieldParam \to 0}
                    -
                    C
                    \integral{\surfaceProto}{}{
                        \left(
                        2 \laplacian{\surfaceProto}{}{}\meanCurv{\surfaceProto}{}
                        +
                        \meanCurv{\surfaceProto}{}
                        \left(
                            \meanCurv{\surfaceProto}{}^2
                            -
                            4\gaussCurv{\surfaceProto}{}
                        \right)
                        \right)
                        \normal[S]{}\cdot\testFunVelocity{}
                    }{\hausdorffM{2}}
                \end{equation*}
                for a constant $ C $.
            \end{lem}
            The strategy of the proof is as follows:
            In Section~\ref{sec:Expanding the L2 gradient of the canham-helfrich energy},
            we have expanded the gradient
            $ 
            	\grad{\phaseFieldProto}{L^2}{}\pfGenWillmoreEnergyFunct{}{}
            $
            and concluded from the energy inequality that 
            all its terms up to order $ \phaseFieldParam^{-1} $ must equal zero.
            The terms remaining on order $\phaseFieldParam^0 $ are shifted to order $ \phaseFieldParam^{-1} $
            by multiplication with $ \grad{}{}{}\phaseFieldProto\cdot\testFunVelocity{} $,
            which is just the right scaling for obtaining the claimed limit using
            Lemma~\ref{lemma:addendum formal asymptotics:concentration}.
            
            \begin{proof}[Proof of~\autoref{lemma:formal asymptotics:singular limit:%
                canham-helfrich force terms:approximation}%
            ]
            Collect all the terms of
            $
                \LTProd{\tubNeighbourhood{\ifCoordsDelta}{\surfaceProto}}{
                    \grad{\phaseFieldProto}{L^2}{}\pfGenWillmoreEnergyFunct{}{}
                }{
                    \grad{}{}{}\phaseFieldProto\cdot\testFunVelocity{}
                }
            $
            on order $\phaseFieldParam^{-1}$:
            \begin{equation*}
                \begin{split}
                    \LTProd{\tubNeighbourhood{\ifCoordsDelta}{\surfaceProto}}{
                        \grad{\phaseFieldProto}{L^2}{}\pfGenWillmoreEnergyFunct{}{}
                    }{
                        \grad{}{}{}\phaseFieldProto\cdot\testFunVelocity{}
                    }
                    &=
                    \phaseFieldParam^{-1}
                    \int_{\tubNeighbourhood{\ifCoordsDelta}{\surfaceProto}}
                        \bigg(
                            -\laplacian{\surfaceProto_{\sdist{}{x}}}{}{
                              	\chemPotIdx{0}{\phaseFieldProto}{}
                            }
                            +
                            \simpleDeriv{\prime}{
                                \chemPotIdxPFWD{1}{\schemeLocalFun{\phaseFieldProto}}{}
                            }
                            \concat
                            j
                            \extMeanCurv{}{}
                            -
                            \simpleDeriv{\prime\prime}{
                                \chemPotIdxPFWD{2}{\schemeLocalFun{\phaseFieldProto}}{}
                            }
                            \concat
                            \interfCoords[\phaseFieldParam]{}{}
                    \\
                            &+
                            \chemPotIdx{2}{\phaseFieldProto}{}                    
                            \doubleWellPot[\schemeSupscr{\prime\prime}]{\phaseFieldProto_0}
                            +
                            \chemPotIdx{0}{\phaseFieldProto}{}
                            \doubleWellPot[\schemeSupscr{(3)}]{\phaseFieldProto_0}
                            \phaseFieldProto_2               
                    \\
	                    &+
                      	\simpleDeriv{\prime}{\schemeLocalFun{\phaseFieldProto}_0}
                        \concat
                        \interfCoords[\phaseFieldParam]{}{}
                        \grad{}{}{}\extMeanCurv{}{}\cdot
                        \extNormal{}
                        \extMeanCurv{}{}
                        +
                        2
                        \phaseFieldParam^{-1}
                        \sdist{}{x}
                        \simpleDeriv{\prime\prime}{
                          	\schemeLocalFun{\phaseFieldProto}_0
                        }
                        \concat
                        \interfCoords[\phaseFieldParam]{}{}
                        \restrFun{\extMeanCurv{}}{\surfaceProto}
                        \left(
                            \restrFun{\extMeanCurv{}}{\surfaceProto}^2
                            -
                            2
                            \restrFun{\extGaussCurv{}{}}{\surfaceProto}
                        \right)
                    \\
                    	&-
                            \simpleDeriv{\prime}{\schemeLocalFun{\phaseFieldProto}_0}
                            \concat
                            \interfCoords[\phaseFieldParam]{}{}
                            \grad{}{2}{}\extMeanCurv{}{}
                            \frobProd 
                            \extNormal{}\tensorProd\extNormal{}
                            -
                            4
                            \phaseFieldParam^{-1}
                            \sdist{}{x}
                            \simpleDeriv{\prime\prime}{\schemeLocalFun{\phaseFieldProto}_0}
                            \concat
                            \interfCoords[\phaseFieldParam]{}{}
                            \left(
                                \restrFun{\extMeanCurv{}}{\surfaceProto}^3
                                - 
                                3 
                                \restrFun{\extMeanCurv{}}{\surfaceProto} 
                                \restrFun{\extGaussCurv{}{}}{\surfaceProto}
                            \right)
                    \bigg)
                    \simpleDeriv{\prime}{\schemeLocalFun{\phaseFieldProto}_0}
                    \concat\interfCoords[\phaseFieldParam]{}{}
                    \normal[\phaseFieldProto]{}\cdot\testFunVelocity{}\;
                    \mathrm{d}\lebesgueM{3}(x)
                    \\
                    &+
                    \landauBigo{\phaseFieldParam}.
                \end{split}
            \end{equation*}
            The terms on the first and second line are from the expansion of the 
            chemical potential and the double well potential.
            The left summand on line three stems from \eqref{equ:formal asymtptotics:inner expansion:%
            consequ bound on L2 grad gen willm:scale -1 main}, the right one from
            \eqref{equ:formal asymtptotics:inner expansion:consequ bound on L2 %
                        grad gen willm:scale -1:exp of mean curv squared}.
            The left summand on line four is taken from 
            \eqref{equ:formal asymptotics:inner expansion:consequ bound on L2 grad %
                        gen willm:finer scale analysis}
            (note 
            \eqref{equ:formal asymptotics:inner expansion:consequ bound on L2 grad %
            gen willm:exp of mean curv dderiv on -2}); the right summand is from
            \eqref{equ:formal asymtptotics:inner expansion:%
            consequ bound on L2 grad gen willm:scale -1 main}            
            (note 
            \eqref{equ:formal asymtptotics:inner expansion:consequ bound on L2 %
            grad gen willm:scale -1:exp of normal deriv mean curv}).            
            First, we substitute some expressions using
            Lemma~\ref{lemma:formal asymptotics:expansions} and
            \eqref{equ:formal asymptotics:inner expansion:L2 grad gen willm:%
            scale -2:result}:
            \begin{equation*}
                \begin{split}
                    \LTProd{\tubNeighbourhood{\ifCoordsDelta}{\surfaceProto}}{
                        \grad{\phaseFieldProto}{L^2}{}\pfGenWillmoreEnergyFunct{}{}
                    }{
                        \grad{}{}{}\phaseFieldProto\cdot\testFunVelocity{}
                    }
                    &=
                    \phaseFieldParam^{-1}
                    \int_{\tubNeighbourhood{\ifCoordsDelta}{\surfaceProto}}
                        \bigg(
                            -\simpleDeriv{\prime}{\schemeLocalFun{\phaseFieldProto}_0}
                            \concat
                            \interfCoords[\phaseFieldParam]{}{}
                            \laplacian{\surfaceProto_{\sdist{}{x}}}{}{}
                              	\extMeanCurv{}
                            +
                            \simpleDeriv{\prime}{
                                \chemPotIdxPFWD{1}{\schemeLocalFun{\phaseFieldProto}}{}
                            }
                            \concat
                            \interfCoords[\phaseFieldParam]{}{}
                            \extMeanCurv{}{}
                            -
                            \simpleDeriv{\prime\prime}{
                                \chemPotIdxPFWD{2}{\schemeLocalFun{\phaseFieldProto}}{}
                            }
                            \concat
                            \interfCoords[\phaseFieldParam]{}{}
                    \\
                            &+
                            \chemPotIdx{2}{\phaseFieldProto}{}                    
                            \doubleWellPot[\schemeSupscr{\prime\prime}]{\phaseFieldProto_0}
                            +
                            \chemPotIdx{0}{\phaseFieldProto}{}
                            \doubleWellPot[\schemeSupscr{(3)}]{\phaseFieldProto_0}
                            \phaseFieldProto_2               
                    \\
                            &+
                            \simpleDeriv{\prime}{\schemeLocalFun{\phaseFieldProto}_0}
                            \concat
                            \interfCoords[\phaseFieldParam]{}{}
                            \extMeanCurv{}
                            \left(
                                \extMeanCurv{}^2
                                -
                                2
                                \extGaussCurv{}{}
                            \right)
                            +
                            2
                            \phaseFieldParam^{-1}
                            \sdist{}{x}
                            \simpleDeriv{\prime\prime}{
                                \schemeLocalFun{\phaseFieldProto}_0
                            }
                            \concat
                            \interfCoords[\phaseFieldParam]{}{}
                            \restrFun{\extMeanCurv{}}{\surfaceProto}
                            \left(
                                \restrFun{\extMeanCurv{}}{\surfaceProto}^2
                                -
                                2
                                \restrFun{\extGaussCurv{}{}}{\surfaceProto}
                            \right)
                    \\
                    		&-
                            2
                            \simpleDeriv{\prime}{\schemeLocalFun{\phaseFieldProto}_0}
                            \concat
                            \interfCoords[\phaseFieldParam]{}{}
                            \left(
                                \extMeanCurv{}^3
                                - 
                                3 
                                \extMeanCurv{}
                                \extGaussCurv{}{}
                            \right)
                            -
                            4
                            \phaseFieldParam^{-1}
                            \sdist{}{x}
                            \simpleDeriv{\prime\prime}{\schemeLocalFun{\phaseFieldProto}_0}
                            \concat
                            \interfCoords[\phaseFieldParam]{}{}
                            \left(
                                \restrFun{\extMeanCurv{}}{\surfaceProto}^3
                                - 
                                3 
                                \restrFun{\extMeanCurv{}}{\surfaceProto} 
                                \restrFun{\extGaussCurv{}{}}{\surfaceProto}
                            \right)
                        \bigg)                        
                        \simpleDeriv{\prime}{\schemeLocalFun{\phaseFieldProto}_0}
                        \concat
                        \interfCoords[\phaseFieldParam]{}{}
                        \normal[\phaseFieldProto]{}
                        \cdot
                        \testFunVelocity{}
                    \mathrm{d}\lebesgueM{3}(x)
                    \\
                    &+
                    \landauBigo{\phaseFieldParam}.
                \end{split}
            \end{equation*}
            Second, we transform the integral using the co-area formula. 
            We denote 
            $ 
            	j(\sigma,\fastVar) = (\orthProj{\surfaceProto}{\sigma},\fastVar)
            $. 
            \begin{equation*}
                \begin{split}
                    \LTProd{\tubNeighbourhood{\ifCoordsDelta}{\surfaceProto}}{
                        \grad{\phaseFieldProto}{L^2}{}\pfGenWillmoreEnergyFunct{}{}
                    }{
                        \grad{}{}{}\phaseFieldProto\cdot\testFunVelocity{}
                    }
                    &=
                    \int_{-\frac{\delta}{\phaseFieldParam}}^{\frac{\delta}{\phaseFieldParam}}
                        \int_{\surfaceProto_{\phaseFieldParam\fastVar}}
                            \bigg(
                                -
                                \simpleDeriv{\prime}{\schemeLocalFun{\phaseFieldProto}_0}
                                \concat j
                                \laplacian{\surfaceProto_{\phaseFieldParam\fastVar}}{}{}
                                \extMeanCurv{}
                                +
                                \simpleDeriv{\prime}{
                                    \chemPotIdxPFWD{1}{\schemeLocalFun{\phaseFieldProto}}{}
                                }
                                \concat j
                                \extMeanCurv{}
                                -
                                \simpleDeriv{\prime\prime}{
                                    \chemPotIdxPFWD{2}{\schemeLocalFun{\phaseFieldProto}}{}
                                }
                                \concat j
                    \\
                                &+
                                \chemPotIdx{2}{\phaseFieldProto}{}
                                \doubleWellPot[\schemeSupscr{\prime\prime}]{
                                    \phaseFieldProto_0
                                }
                                +
                                \chemPotIdx{0}{\phaseFieldProto}{}
                                \doubleWellPot[\schemeSupscr{(3)}]{
                                    \phaseFieldProto_0
                                }
                                \phaseFieldProto_2
                    \\
                    			&+
                                \simpleDeriv{\prime}{\schemeLocalFun{\phaseFieldProto}_0}
                                \concat j
                                \extMeanCurv{}
                                \left(
                                    \extMeanCurv{}^2
                                    -
                                    2
                                    \extGaussCurv{}
                                \right)
                                +
                                2
                                \fastVar
                                \simpleDeriv{\prime\prime}{
                                    \schemeLocalFun{\phaseFieldProto}_0
                                }
                                \concat j
                                \restrFun{\extMeanCurv{}}{\surfaceProto}
                                \left(
                                    \restrFun{\extMeanCurv{}}{\surfaceProto}^2
                                    -
                                    2
                                    \restrFun{\extGaussCurv{}}{\surfaceProto}
                                \right)
                    \\
                    			&-
                                2
                                \simpleDeriv{\prime}{\schemeLocalFun{\phaseFieldProto}_0}
                                \concat j
                                \left(
                                    \extMeanCurv{}^3
                                    - 
                                    3 
                                    \extMeanCurv{}
                                    \extGaussCurv{}
                                \right)                            
                                -
                                4
                                \fastVar
                                \simpleDeriv{\prime\prime}{\schemeLocalFun{\phaseFieldProto}_0}
                                \concat j
                                \left(
                                    \restrFun{\extMeanCurv{}}{\surfaceProto}^3
                                    - 
                                    3 
                                    \restrFun{\extMeanCurv{}}{\surfaceProto} 
                                    \restrFun{\extGaussCurv{}}{\surfaceProto}
                                \right)
                            \quad\bigg)
                        \simpleDeriv{\prime}{\schemeLocalFun{\phaseFieldProto}_0}
                        \concat j
                        \normal[\phaseFieldProto]{}\cdot\testFunVelocity{}\,
                    \mathrm{d}\hausdorffM{2}(\sigma)\,\mathrm{d}\lebesgueM{1}(\fastVar)
                    \\
                    &+
                    \landauBigo{\phaseFieldParam}.
                \end{split}
            \end{equation*}

            With integration by parts, it follows directly that
            $ 
            	\integral{-\infty}{\infty}{
                  	\fastVar
                    \simpleDeriv{\prime\prime}{\schemeLocalFun{\phaseFieldProto}_0}
                    \simpleDeriv{\prime}{\schemeLocalFun{\phaseFieldProto}_0}
                }{\fastVar}
                =
                -\frac{1}{2}
                \integral{-\infty}{\infty}{
                  	\left( 
	                    \simpleDeriv{\prime}{\schemeLocalFun{\phaseFieldProto}_0}
                    \right)^2
                }{\fastVar}.
            $
            Exploiting this property and the independence of $ \schemeLocalFun{\phaseFieldProto}_0 $
            of the first argument of $ j $,
            we obtain
            \begin{equation*}
                \begin{split}
                    &\integral{
                      	-\frac{\ifCoordsDelta}{\phaseFieldParam}
                    }{
                      	\frac{\ifCoordsDelta}{\phaseFieldParam}
                    }{
                      	\integral{\surfaceProto_{\phaseFieldParam\fastVar}}{}{
                            \left(
                            \simpleDeriv{\prime}{\schemeLocalFun{\phaseFieldProto}_0}
                            \concat j
                            \extMeanCurv{}
                            \left(
                                \extMeanCurv{}^2
                                -
                                2
                                \extGaussCurv{}
                            \right)
                            +
                            2
                            \fastVar
                            \simpleDeriv{\prime\prime}{
                                \schemeLocalFun{\phaseFieldProto}_0
                            }
                            \concat j
                            \restrFun{\extMeanCurv{}}{\surfaceProto}
                            \left(
                                \restrFun{\extMeanCurv{}}{\surfaceProto}^2
                                -
                                2
                                \restrFun{\extGaussCurv{}}{\surfaceProto}
                            \right)
                            \right)
                            \simpleDeriv{\prime}{\schemeLocalFun{\phaseFieldProto}_0}
                            \concat j
                        }{\hausdorffM{2}(\sigma)}
                    }{\lebesgueM{1}(\fastVar)}
                    =
                    \\
                    &
                    \integral{
                      	-\frac{\ifCoordsDelta}{\phaseFieldParam}
                    }{
                      	\frac{\ifCoordsDelta}{\phaseFieldParam}
                    }{
                      	\left( \simpleDeriv{\prime}{\schemeLocalFun{\phaseFieldProto}_0} \right)^2
                        \integral{\surfaceProto_{\phaseFieldParam\fastVar}}{}{
                            \extMeanCurv{}
                            \left(
                                \extMeanCurv{}{}^2
                                -
                                2
                                \extGaussCurv{}{}
                            \right)
	                    }{\hausdorffM{2}(\sigma)}
                    }{\lebesgueM{1}(\fastVar)}
                    -
                    \integral{
                      	-\frac{\ifCoordsDelta}{\phaseFieldParam}
                    }{
                      	\frac{\ifCoordsDelta}{\phaseFieldParam}
                    }{
                      	\left(
                            \simpleDeriv{\prime}{
                                \schemeLocalFun{\phaseFieldProto}_0
                            }                            
                        \right)^2
                    }{\lebesgueM{1}(\fastVar)}
                    \integral{
                      	\surfaceProto
                    }{}{
                        \restrFun{\extMeanCurv{}}{\surfaceProto}
                        \left(
                            \restrFun{\extMeanCurv{}}{\surfaceProto}^2
                            -
                            2
                            \restrFun{\extGaussCurv{}}{\surfaceProto}
                        \right)
                    }{\hausdorffM{2}(\sigma)}.
                \end{split}
            \end{equation*}
            Thanks to
            Lemma~\ref{lemma:addendum formal asymptotics:concentration}, 
            we know that this difference converges to zero as $ \phaseFieldParam \searrow 0 $.
            The same reasoning applies for
            \begin{equation*}
                \integral{
                    -\frac{\ifCoordsDelta}{\phaseFieldParam}
                }{
                    \frac{\ifCoordsDelta}{\phaseFieldParam}
                }{
                  	\integral{\surfaceProto_{\phaseFieldParam\fastVar}}{}{
                        \left(
                            -
                            2
                            \simpleDeriv{\prime}{\schemeLocalFun{\phaseFieldProto}_0}
                            \concat j
                            \left(
                                \extMeanCurv{}^3
                                - 
                                3 
                                \extMeanCurv{}
                                \extGaussCurv{}
                            \right)                            
                            -
                            4
                            \fastVar
                            \simpleDeriv{\prime\prime}{\schemeLocalFun{\phaseFieldProto}_0}
                            \concat j
                            \left(
                                \restrFun{\extMeanCurv{}}{\surfaceProto}^3
                                - 
                                3 
                                \restrFun{\extMeanCurv{}}{\surfaceProto} 
                                \restrFun{\extGaussCurv{}}{\surfaceProto}
                            \right)
                        \right)
                        \simpleDeriv{\prime}{\schemeLocalFun{\phaseFieldProto}_0}
                        \concat j
                    }{\hausdorffM{2}}
                }{\lebesgueM{1}(\fastVar)}.
            \end{equation*}

            The remaining terms are treated as follows:
            We recall $ \simpleDeriv{\prime}{\meanCurv[\hat]{}{}} \in \landauBigo{\phaseFieldParam} $
            and $ \simpleDeriv{\prime\prime}{\meanCurv[\hat]{}{}} \in \landauBigo{\phaseFieldParam^2} $
            (cf.~\eqref{equ:formal asymptotics:inner expansion:consequ bound on L2 grad %
            gen willm:exp of mean curv deriv on -2},
            \eqref{equ:formal asymptotics:inner expansion:consequ bound on L2 grad %
                        gen willm:exp of mean curv dderiv on -2})
            and
            \begin{equation*}
              	\chemPotIdxPFWD{2}{\schemeLocalFun{\phaseFieldProto}}{}
                =
                \simpleDeriv{\prime}{\schemeLocalFun{\phaseFieldProto}_2}
                \meanCurv[\hat]{}{}
                -
                \simpleDeriv{\prime\prime}{\schemeLocalFun{\phaseFieldProto}_3}
                +
                \schemeLocalFun{\phaseFieldProto}_3
                \doubleWellPot[\schemeSupscr{\prime\prime}]{
                    \schemeLocalFun{\phaseFieldProto}_0
                }                
            \end{equation*}
            (cf.~\eqref{equ:formal asymptotics:expansion of the L2 gradient of the canham-helfrich energy:%
            chem pot expansion} with
            $ \phaseFieldProto_1 = 0 $).
            Thus, the following expansion holds:
            \begin{equation*}
                \begin{split}
                    \simpleDeriv{\prime\prime}{
                        \chemPotIdxPFWD{2}{
                            \schemeLocalFun{\phaseFieldProto}
                        }{}
                    }
                    =
                    \simpleDeriv{(3)}{\schemeLocalFun{\phaseFieldProto}_2} 
                    \meanCurv[\hat]{}{}         
                    -
                    \simpleDeriv{\prime\prime}{\hat{q}}
                    +
                    \landauBigo{\phaseFieldParam},
                \end{split}
            \end{equation*}
            where
            \begin{equation*}
                \hat{q} =
                \simpleDeriv{\prime\prime}{\schemeLocalFun{\phaseFieldProto}_3}
                -
                \schemeLocalFun{\phaseFieldProto}_3
                \doubleWellPot[\schemeSupscr{\prime\prime}]{
                    \schemeLocalFun{\phaseFieldProto}_0
                }.
            \end{equation*}
            This way we see with 
            $ 
            	\chemPotIdx{0}{\phaseFieldProto}{} 
                = 
                \extMeanCurv{}
                \simpleDeriv{\prime}{\schemeLocalFun{\phaseFieldProto}_0}
                \concat
                \interfCoords[\phaseFieldParam]{}{}
            $
            and
            $
                \chemPotIdxPFWD{1}{\schemeLocalFun{\phaseFieldProto}}{}
                =
                -
                \simpleDeriv{\prime\prime}{
                    \schemeLocalFun{\phaseFieldProto}_2   
                }
                +
                \doubleWellPot[\schemeSupscr{\prime\prime}]{
                    \schemeLocalFun{\phaseFieldProto}_0
                }
                \schemeLocalFun{\phaseFieldProto}_2                         
            $
            (cf.~\eqref{equ:formal asymptotics:expansion of the L2 gradient of the canham-helfrich energy:%
            chem pot expansion} with
            $ \phaseFieldProto_1 = 0 $)
            that
            \begin{equation*}
                \begin{split}
                    -\simpleDeriv{\prime\prime}{
                        \chemPotIdxPFWD{2}{\schemeLocalFun{\phaseFieldProto}}{}
                    }
                    \concat j
                    +
                    \chemPotIdx{2}{\phaseFieldProto}{}                    
                    \doubleWellPot[\schemeSupscr{\prime\prime}]{\phaseFieldProto_0}
                    +
                    \chemPotIdx{0}{\phaseFieldProto}{}
                    \doubleWellPot[\schemeSupscr{(3)}]{\phaseFieldProto_0}
                    \phaseFieldProto_2               
                    &=
                    -
                    \extMeanCurv{}
                    \simpleDeriv{(3)}{
                        \schemeLocalFun{\phaseFieldProto}_2
                    }
                    \concat j
                    +
                    \extMeanCurv{}{}
                    \doubleWellPot[\schemeSupscr{\prime\prime}]{
                        \phaseFieldProto_0
                    }
                    \simpleDeriv{\prime}{
                        \schemeLocalFun{\phaseFieldProto}_2
                    }                  
                    \concat j
                    +
                    \chemPotIdx{0}{\phaseFieldProto}{}
                    \doubleWellPot[\schemeSupscr{(3)}]{\phaseFieldProto_0}
                    \phaseFieldProto_2               
                    \\
                    &+
                    \simpleDeriv{\prime\prime}{
                        \hat{q}
                    }
                    \concat j 
                    -
                    q
                    \doubleWellPot[\schemeSupscr{\prime\prime}]{\phaseFieldProto_0}
                    +
                    \landauBigo{\phaseFieldParam}
                    \\
                    &=
                    \extMeanCurv{}
                    \simpleDeriv{\prime}{
                    \left(
                        -\simpleDeriv{\prime\prime}{
                            \schemeLocalFun{\phaseFieldProto}_2   
                        }
                        +
                        \doubleWellPot[\schemeSupscr{\prime\prime}]{
                            \schemeLocalFun{\phaseFieldProto}_0
                        }
                        \schemeLocalFun{\phaseFieldProto}_2
                    \right)
                    } 
                    \concat j
                    +
                    \simpleDeriv{\prime\prime}{\hat{q}}
                    \concat j
                    -
                    q
                    \doubleWellPot[\schemeSupscr{\prime\prime}]{
                        \schemeLocalFun{\phaseFieldProto}_0
                    }
                    \concat j
                    +
                    \landauBigo{\phaseFieldParam}
                    \\
                    &=
                    \extMeanCurv{}
                    \simpleDeriv{\prime}{
                        \chemPotIdxPFWD{1}{\schemeLocalFun{\phaseFieldProto}}{}
                    }
                    \concat j
                    +
                    \simpleDeriv{\prime\prime}{\hat{q}}
                    \concat j 
                    -
                    q
                    \doubleWellPot[\schemeSupscr{\prime\prime}]{
                        \schemeLocalFun{\phaseFieldProto}_0
                    }
                    \concat j
                    +
                    \landauBigo{\phaseFieldParam},
                \end{split}
            \end{equation*}
            where we used
            $ 
                \chemPotIdx{0}{\phaseFieldProto}{}
            	\doubleWellPot[\schemeSupscr{(3)}]{\phaseFieldProto_0}
                \phaseFieldProto_2 
            	= 
                \simpleDeriv{\prime}{\phaseFieldProto_0}
                \concat j
                \doubleWellPot[\schemeSupscr{(3)}]{\phaseFieldProto_0}
                \phaseFieldProto_2 
                \extMeanCurv{} 
                =
                \simpleDeriv{\prime}{
                    \left(
                        \doubleWellPot[\schemeSupscr{\prime\prime}]{\schemeLocalFun{\phaseFieldProto}_0}
                    \right)
                }
                \concat j
                \phaseFieldProto_2 
                \extMeanCurv{}.
            $
            We therefore obtain,
            \begin{equation*}
                \begin{split}
                    &\int_{-\frac{\ifCoordsDelta}{\phaseFieldParam}}^{\frac{\ifCoordsDelta}{\phaseFieldParam}}
                    	\int_{\surfaceProto_{\phaseFieldParam\fastVar}}
                            \bigg(
                                -\simpleDeriv{\prime}{\schemeLocalFun{\phaseFieldProto}_0}
                                \concat j
                                \laplacian{\surfaceProto_{\phaseFieldParam\fastVar}}{}{}
                                \extMeanCurv{}
                                +
                                \extMeanCurv{}
                                \simpleDeriv{\prime}{
                                    \chemPotIdxPFWD{1}{\schemeLocalFun{\phaseFieldProto}}{}
                                }
                                \concat j
                                -
                                \simpleDeriv{\prime\prime}{
                                    \chemPotIdxPFWD{2}{\schemeLocalFun{\phaseFieldProto}}{}
                                }
                                \concat j
                                \\
                                &+
                                \chemPotIdx{2}{\phaseFieldProto}{}
                                \doubleWellPot[\schemeSupscr{\prime\prime}]{
                                    \phaseFieldProto_0
                                }
                                +
                                \chemPotIdx{0}{\phaseFieldProto}{}
                                \doubleWellPot[\schemeSupscr{(3)}]{
                                    \phaseFieldProto_0
                                }
                                \phaseFieldProto_2
                            \bigg)
		                    \simpleDeriv{\prime}{\schemeLocalFun{\phaseFieldProto}_0}
                            \concat j
        		            \normal[\phaseFieldProto]{}\cdot\testFunVelocity{}\;
                    	\mathrm{d}\hausdorffM{2}\,
                    \mathrm{d}\lebesgueM{1}(\fastVar)               
                    =
                    \\
                    &\integral{
                      	-\frac{\ifCoordsDelta}{\phaseFieldParam}
                    }{
                        \frac{\ifCoordsDelta}{\phaseFieldParam}
                    }{
                      	\integral{\surfaceProto_{\phaseFieldParam\fastVar}}{}{
                            \left(
                                -
                                \simpleDeriv{\prime}{\schemeLocalFun{\phaseFieldProto}_0}
                                \concat j
                                \laplacian{\surfaceProto_{\phaseFieldParam\fastVar}}{}{}\extMeanCurv{}
                                +
                                2
                                \extMeanCurv{}
                                \simpleDeriv{\prime}{
                                    \chemPotIdxPFWD{1}{\schemeLocalFun{\phaseFieldProto}}{}
                                }
                                \concat j
                                +
                                \simpleDeriv{\prime\prime}{\hat{q}}
                                \concat j
                                -
                                \hat{q}
                                \concat j
                                \doubleWellPot[\schemeSupscr{\prime\prime}]{
                                    \schemeLocalFun{\phaseFieldProto}_0
                                }
                                \concat j
                            \right)
		                    \simpleDeriv{\prime}{\schemeLocalFun{\phaseFieldProto}_0}
                            \concat j
        		            \normal[\phaseFieldProto]{}\cdot\testFunVelocity{}
                        }{\hausdorffM{2}(\sigma)}
                    }{\lebesgueM{1}(\fastVar)}
                    +
                    \landauBigo{\phaseFieldParam}.
                \end{split}
            \end{equation*}
            To treat
            $
                \integral{
                    -\frac{\ifCoordsDelta}{\phaseFieldParam}
                }{
                    \frac{\ifCoordsDelta}{\phaseFieldParam}
                }{
                    \simpleDeriv{\prime}{\schemeLocalFun{\phaseFieldProto}_0}
                    \integral{\surfaceProto_{\phaseFieldParam\fastVar}}{}{
                        \simpleDeriv{\prime\prime}{\schemeLocalFun{q}}\concat j
                        \normal[\phaseFieldProto]{}\cdot\testFunVelocity{}
                    }{\hausdorffM{2}(\sigma)}
                }{\lebesgueM{1}(\fastVar)}
            $,
            we take a global parametrisation $ \funSig{\gamma}{\reals^2}{\surfaceProto} $
            (this is w.l.o.g. since in case there is no global parametrisation,
            we make all the following calculations locally and patch the integrals together afterwards),
            and define
            $ 
            	\gamma_{\phaseFieldParam\fastVar}(\tangParamVar) 
                = 
                \gamma(\tangParamVar) 
                + 
                \phaseFieldParam\fastVar\normal[S]{\gamma(\tangParamVar)} 
            $.
            With the area formula, we obtain
            \begin{equation*}
                \begin{split}
                    &\integral{
                      	-\frac{\ifCoordsDelta}{\phaseFieldParam}
                    }{
                        \frac{\ifCoordsDelta}{\phaseFieldParam}
                    }{
                      	\simpleDeriv{\prime}{\schemeLocalFun{\phaseFieldProto}_0}(\fastVar)
                      	\integral{\surfaceProto_{\phaseFieldParam\fastVar}}{}{
                          	\simpleDeriv{\prime\prime}{\schemeLocalFun{q}}\concat j
        		            \normal[\phaseFieldProto]{}
                            \cdot
                            \testFunVelocity{}
                        }{\hausdorffM{2}(\sigma)}
                    }{\lebesgueM{1}(\fastVar)}
                    =
                    \\
                    &\integral{
                      	-\frac{\ifCoordsDelta}{\phaseFieldParam}
                    }{
                        \frac{\ifCoordsDelta}{\phaseFieldParam}
                    }{
                      	\simpleDeriv{\prime}{\schemeLocalFun{\phaseFieldProto}_0}(\fastVar)
                      	\integral{\reals^2}{}{
                          	\simpleDeriv{\prime\prime}{\schemeLocalFun{q}}
                            (\gamma(s),\fastVar)
        		            \normal[\phaseFieldProto]{}
                            \concat \gamma_{\phaseFieldParam\fastVar}
                            \cdot
                            \testFunVelocity{}
                            \concat
                            \gamma_{\phaseFieldParam\fastVar}
                            \jacobian{\gamma_{\phaseFieldParam\fastVar}}
                        }{\lebesgueM{2}(\tangParamVar)}
                    }{\lebesgueM{1}(\fastVar)}.
                \end{split}
            \end{equation*}
            Note that 
            $
            	j(\gamma_{\phaseFieldParam\fastVar}(\tangParamVar),\fastVar)
                =
                (\gamma(\tangParamVar),\fastVar).
            $
            We further integrate by parts
            \begin{equation*}
                \begin{split}
                    &\integral{
                      	-\frac{\ifCoordsDelta}{\phaseFieldParam}
                    }{
                        \frac{\ifCoordsDelta}{\phaseFieldParam}
                    }{
                      	\simpleDeriv{\prime}{\schemeLocalFun{\phaseFieldProto}_0}(\fastVar)
                      	\integral{\reals^2}{}{
                          	\simpleDeriv{\prime\prime}{\schemeLocalFun{q}}
                            (\gamma(s),\fastVar)
                            \jacobian{\gamma_{\phaseFieldParam\fastVar}}
                            \normal[\phaseFieldProto]{}
                            \concat
                            \gamma_{\phaseFieldParam\fastVar}
                            \cdot
                            \testFunVelocity{}
                            \concat
                            \gamma_{\phaseFieldParam\fastVar}
                        }{\lebesgueM{2}(\tangParamVar)}
                    }{\lebesgueM{1}(\fastVar)}
                    =
                    \\
                    &-
                    \integral{
                      	-\frac{\ifCoordsDelta}{\phaseFieldParam}
                    }{
                        \frac{\ifCoordsDelta}{\phaseFieldParam}
                    }{
                      	\simpleDeriv{\prime\prime}{\schemeLocalFun{\phaseFieldProto}_0}(\fastVar)
                      	\integral{\reals^2}{}{
                          	\simpleDeriv{\prime}{\schemeLocalFun{q}}
                            (\gamma(s),\fastVar)
                            \jacobian{\gamma_{\phaseFieldParam\fastVar}}
                            \normal[\phaseFieldProto]{}
                            \concat
                            \gamma_{\phaseFieldParam\fastVar}
                            \cdot
                            \testFunVelocity{}
                            \concat
                            \gamma_{\phaseFieldParam\fastVar}
                        }{\lebesgueM{2}(\tangParamVar)}
                    }{\lebesgueM{1}(\fastVar)}
                    \\
                    &+
                    \integral{\reals^2}{}{
                        \left[
                            \simpleDeriv{\prime}{\hat{q}}
                            (\gamma(s),\cdot)
                            \simpleDeriv{\prime}{\schemeLocalFun{\phaseFieldProto}_0}
                            \jacobian{\gamma_{\phaseFieldParam\fastVar}}
                            \normal[\phaseFieldProto]{}
                            \concat
                            \gamma_{\phaseFieldParam\fastVar}
                            \cdot
                            \testFunVelocity{}
                            \concat
                            \gamma_{\phaseFieldParam\fastVar}
                        \right]_{-\delta/\phaseFieldParam}^{\delta/\phaseFieldParam}
                    }{\lebesgueM{2}(\tangParamVar)}
                    \\
                    &+
                    \integral{
                      	-\frac{\ifCoordsDelta}{\phaseFieldParam}
                    }{
                      	\frac{\ifCoordsDelta}{\phaseFieldParam}
                    }{
                        \simpleDeriv{\prime\prime}{\schemeLocalFun{\phaseFieldProto}_0}
                      	\integral{\reals^2}{}{
                            \simpleDeriv{\prime}{\hat{q}}
                            (\gamma(s),\fastVar)
                            \jacobian{\gamma_{\phaseFieldParam\fastVar}}
                          	\simpleDeriv{\prime}{\left(
                                \normal[\phaseFieldProto]{}
                                \concat
                                \gamma_{\phaseFieldParam\fastVar}
                                \cdot
                                \testFunVelocity{}
                                \concat
                                \gamma_{\phaseFieldParam\fastVar}
                          	\right)}
                        }{\lebesgueM{2}(\tangParamVar)}
                    }{\lebesgueM{1}(\fastVar)}
                    +
                    \landauBigo{\phaseFieldParam}.
                \end{split}
            \end{equation*}
            We observe
            \begin{equation*}
                \begin{split}
                    \simpleDeriv{\prime}{\left(
                        \normal[\phaseFieldProto]{}
                        \concat
                        \gamma_{\phaseFieldParam\fastVar}
                        \cdot
                        \testFunVelocity{}
                        \concat
                        \gamma_{\phaseFieldParam\fastVar}
                    \right)}
                  	&=
                    \phaseFieldParam
                    \transposed{\grad{}{}{}\normal[\phaseFieldProto]{}}
                    \concat\gamma_{\phaseFieldParam\fastVar}
                    \normal[\surfaceProto]{}
                    \concat\gamma
                    \cdot
                    \testFunVelocity{}
                    \concat
                    \gamma_{\phaseFieldParam\fastVar}
                    +
                    \phaseFieldParam
                    \normal[\phaseFieldProto]{}
                    \concat
                    \gamma_{\phaseFieldParam\fastVar}
                    \cdot
                    \transposed{\grad{}{}{}\testFunVelocity{}}
                    \concat
                    \gamma_{\phaseFieldParam\fastVar}
                    \normal[\surfaceProto]{}\concat\gamma
                    \\
                    &=
                    \phaseFieldParam
                    \left(
                    \frac{1}{\abs{\phaseFieldProto}}
                    \transposed{\grad{}{2}{}\phaseFieldProto}
                    \orthProjMat{\normal[\phaseFieldProto]{}}{}
                    \right)\concat\gamma_{\phaseFieldParam\fastVar}
                    \normal[\surfaceProto]{}
                    \concat\gamma
                    \cdot
                    \testFunVelocity{}
                    \concat
                    \gamma_{\phaseFieldParam\fastVar}                    
                    +
                    \landauBigo{\phaseFieldParam},
            	\end{split}
            \end{equation*}
            and we have 
            $
            	\orthProjMat{\normal[\phaseFieldProto]{}}{}
                \concat\gamma_{\phaseFieldParam\fastVar}
                \normal[\surfaceProto]{} \in \landauBigo{\phaseFieldParam^2}
                \concat\gamma
            $
            (see \eqref{equ:formal asymptotics:expansion of the L2 gradient of the coupling energy:%
                expansion of the normal}),
            so
            $$
                \integral{
                    -\frac{\ifCoordsDelta}{\phaseFieldParam}
                }{
                    \frac{\ifCoordsDelta}{\phaseFieldParam}
                }{
                    \simpleDeriv{\prime\prime}{\schemeLocalFun{\phaseFieldProto}_0}
                    \integral{\reals^2}{}{
                        \simpleDeriv{\prime}{\hat{q}}
                        \jacobian{\gamma_{\phaseFieldParam\fastVar}}
                        \simpleDeriv{\prime}{\left(
                            \normal[\phaseFieldProto]{}
                            \concat
                            \gamma_{\phaseFieldParam\fastVar}
                            \cdot
                            \testFunVelocity{}
                            \concat
                            \gamma_{\phaseFieldParam\fastVar}
                        \right)}
                    }{\lebesgueM{2}(\tangParamVar)}
                }{\lebesgueM{1}(\fastVar)}
                \in \landauBigo{\phaseFieldParam}.
            $$
            With 
            Jacobi's formula for derivatives of determinants, it can be seen 
            that the derivative of the Jacobian w.r.t $ \fastVar $ is also in $ \landauBigo{\phaseFieldParam} $.
            Integrating by parts one more time leads us therefore to
            \begin{equation*}
                \begin{split}
                    &\integral{
                      	-\frac{\ifCoordsDelta}{\phaseFieldParam}
                    }{
                        \frac{\ifCoordsDelta}{\phaseFieldParam}
                    }{
                      	\simpleDeriv{\prime}{\schemeLocalFun{\phaseFieldProto}_0}(\fastVar)
                      	\integral{\reals^2}{}{
                          	\simpleDeriv{\prime\prime}{\schemeLocalFun{q}}
                            (\gamma(\tangParamVar),\fastVar)
                            \jacobian{\gamma_{\phaseFieldParam\fastVar}}
                            \normal[\phaseFieldProto]{}
                            \concat
                            \gamma_{\phaseFieldParam\fastVar}
                            \cdot
                            \testFunVelocity{}
                            \concat
                            \gamma_{\phaseFieldParam\fastVar}
                        }{\lebesgueM{2}(\tangParamVar)}
                    }{\lebesgueM{1}(\fastVar)}
                    =
                    \\
                    &
                    \integral{
                      	-\frac{\ifCoordsDelta}{\phaseFieldParam}
                    }{
                        \frac{\ifCoordsDelta}{\phaseFieldParam}
                    }{
                        \simpleDeriv{(3)}{\schemeLocalFun{\phaseFieldProto}_0}(\fastVar)
                      	\integral{\reals^2}{}{
                            \schemeLocalFun{q}
                            (\gamma(\tangParamVar),\fastVar)
                            \normal[\phaseFieldProto]{}
                            \concat
                            \gamma_{\phaseFieldParam\fastVar}
                            \cdot
                            \testFunVelocity{}
                            \concat
                            \gamma_{\phaseFieldParam\fastVar}
                            \jacobian{\gamma_{\phaseFieldParam\fastVar}}
                        }{\lebesgueM{2}(\tangParamVar)}
                    }{\lebesgueM{1}(\fastVar)}
                    \\
                    &+
                    \integral{\reals^2}{}{
                        \left[
                            \simpleDeriv{\prime}{\hat{q}}
                            (\gamma(s),\cdot)
                            \simpleDeriv{\prime}{\schemeLocalFun{\phaseFieldProto}_0}
                            \normal[\phaseFieldProto]{}
                            \concat
                            \gamma_{\phaseFieldParam\fastVar}
                            \cdot
                            \testFunVelocity{}
                            \concat
                            \gamma_{\phaseFieldParam\fastVar}
                            \jacobian{\gamma_{\phaseFieldParam\fastVar}}                            
                        \right]_{-\delta/\phaseFieldParam}^{\delta/\phaseFieldParam}
                    }{\lebesgueM{2}(\tangParamVar)}
                    \\
                    &-
                    \integral{\reals^2}{}{
                        \left[
                            \hat{q}(\gamma(s),\cdot)
                            \simpleDeriv{\prime\prime}{\schemeLocalFun{\phaseFieldProto}_0}
                            \normal[\phaseFieldProto]{}
                            \concat
                            \gamma_{\phaseFieldParam\fastVar}
                            \cdot
                            \testFunVelocity{}
                            \concat
                            \gamma_{\phaseFieldParam\fastVar}
                            \jacobian{\gamma_{\phaseFieldParam\fastVar}}
                        \right]_{-\delta/\phaseFieldParam}^{\delta/\phaseFieldParam}
                    }{\lebesgueM{2}(\tangParamVar)}
                    +
                    \landauBigo{\phaseFieldParam}.
                \end{split}
            \end{equation*}
            The last integrals vanish for $ \phaseFieldParam \searrow 0 $ since
            $ \simpleDeriv{\prime}{\schemeLocalFun{\phaseFieldProto}_0} $
            and
            $ \simpleDeriv{\prime\prime}{\schemeLocalFun{\phaseFieldProto}_0} $
            vanish for $ \fastVar\to\pm\infty$ (see 
            \eqref{equ:formal asymptotics:lim matching principle:membr phase field:outer},
            \eqref{equ:formal asymptotics:lim matching principle:membr phase field:inner}). 
            Hence,
            \begin{equation*}
                \begin{split}
                    \integral{
                      	-\frac{\ifCoordsDelta}{\phaseFieldParam}
                    }{
                        \frac{\ifCoordsDelta}{\phaseFieldParam}
                    }{
                      	\integral{\surfaceProto_{\phaseFieldParam}}{}{
                            \left(
                                \simpleDeriv{\prime\prime}{\hat{q}}
                                -
                                \hat{q}
                                \doubleWellPot[\schemeSupscr{\prime\prime}]{
                                    \schemeLocalFun{\phaseFieldProto}_0
                                }
                            \right)
                            \simpleDeriv{\prime}{\schemeLocalFun{\phaseFieldProto}_0}
                            \normal[\phaseFieldProto]{}
                            \cdot
                            \testFunVelocity{}
	                    }{\hausdorffM{2}(\sigma)}
                    }{\lebesgueM{1}(\fastVar)}
                    \specConv{\phaseFieldParam\searrow0}
                    \integral{-\infty}{\infty}{
                        \simpleDeriv{(3)}{\schemeLocalFun{\phaseFieldProto}_0}
                        -
                        \simpleDeriv{\prime}{\schemeLocalFun{\phaseFieldProto}_0}                        
                        \doubleWellPot[\schemeSupscr{\prime\prime}]{
                            \schemeLocalFun{\phaseFieldProto}_0
                        }                      	
                    }{\lebesgueM{1}(\fastVar)}
                    \integral{\surfaceProto}{}{
                        q
                        \normal[\phaseFieldProto]{}
                        \cdot
                        \testFunVelocity{}
                    }{\hausdorffM{2}(\tangParamVar)}.
                \end{split}
            \end{equation*}
            Note that
            $
                \simpleDeriv{(3)}{\schemeLocalFun{\phaseFieldProto}_0}
                -
                \simpleDeriv{\prime}{\schemeLocalFun{\phaseFieldProto}_0}                        
                \doubleWellPot[\schemeSupscr{\prime\prime}]{
                    \schemeLocalFun{\phaseFieldProto}_0
                }
                =
                0
            $
            (property of the optimal profile),
            so the whole integral vanishes in the limit.

            The last term we need to investigate is
            $
            	2\meanCurv[\hat]{}{}\simpleDeriv{\prime}{
                  	\chemPotIdxPFWD{1}{\schemeLocalFun{\phaseFieldProto}}{}
                },
            $ 
            and we already know (see 
            \eqref{equ:formal asymptotics:expansion of the L2 gradient of the canham-helfrich %
            energy:order minus one})
            \begin{equation*}
            	2\meanCurv[\hat]{}{}
                \simpleDeriv{\prime}{
                  	\chemPotIdxPFWD{1}{\schemeLocalFun{\phaseFieldProto}}{}
                }
                =
                -
                \meanCurv[\hat]{}{}
                \left(
                    \restrFun{\meanCurv[\hat]{}{}}{\surfaceProto}^2
                    -
                    4
                    \restrFun{\gaussCurv[\hat]{}{}}{\surfaceProto}
                \right)
                \simpleDeriv{\prime}{
                    \left(
                        \simpleDeriv{\prime}{\schemeLocalFun{\phaseFieldProto}_0}
                        \fastVar
                    \right)
                }
                =
                -\meanCurv[\hat]{}{}
                \left(
                    \restrFun{\meanCurv[\hat]{}{}}{\surfaceProto}^2
                    -
                    4
                    \restrFun{\gaussCurv[\hat]{}{}}{\surfaceProto}
                \right)
                \left(
                    \simpleDeriv{\prime\prime}{\schemeLocalFun{\phaseFieldProto}_0}
                    \fastVar
                    +
                    \simpleDeriv{\prime}{\schemeLocalFun{\phaseFieldProto}_0}
                \right).
            \end{equation*}
            Now we observe
            \begin{equation*}
                \begin{split}
                    &\integral{
                      	-\frac{\ifCoordsDelta}{\phaseFieldParam}
                    }{
                        \frac{\ifCoordsDelta}{\phaseFieldParam}
                    }{
                      	\integral{\surfaceProto_{\phaseFieldParam\fastVar}}{}{
                            -\extMeanCurv{}
                            \left(
                                \restrFun{\extMeanCurv{}}{\surfaceProto}^2
                                -
                                4
                                \restrFun{\extGaussCurv{}}{\surfaceProto}
                            \right)
                            \left(
                                \simpleDeriv{\prime\prime}{\schemeLocalFun{\phaseFieldProto}_0}
                                \concat j
                                \fastVar
                                +
                                \simpleDeriv{\prime}{\schemeLocalFun{\phaseFieldProto}_0}
                                \concat j
                            \right)
                            \simpleDeriv{\prime}{\schemeLocalFun{\phaseFieldProto}_0}
                            \concat j
                            \normal[\phaseFieldProto]{}\cdot\testFunVelocity{}
                        }{\hausdorffM{2}(\sigma)}
                    }{\lebesgueM{1}(\fastVar)}
                    \specConv{\phaseFieldParam\searrow0}
                    \\
                    -
                    &\integral{-\infty}{\infty}{
                        \simpleDeriv{\prime\prime}{\schemeLocalFun{\phaseFieldProto}_0}
                        \simpleDeriv{\prime}{\schemeLocalFun{\phaseFieldProto}_0}
                        \fastVar                      	
                        +
                        \left(\simpleDeriv{\prime}{\schemeLocalFun{\phaseFieldProto}_0}\right)^2
                    }{\lebesgueM{1}(\fastVar)}
                    \integral{\surfaceProto}{}{
                        \meanCurv[\hat]{}{}
                        \left(
                            \restrFun{\meanCurv[\hat]{}{}}{\surfaceProto}^2
                            -
                            4
                            \restrFun{\gaussCurv[\hat]{}{}}{\surfaceProto}
                        \right)
                        \normal[\phaseFieldProto]{}\cdot\testFunVelocity{}
                    }{\hausdorffM{2}(\sigma)},
                \end{split}
            \end{equation*}
            and finally use 
            $
                \integral{-\infty}{\infty}{
                    \simpleDeriv{\prime\prime}{\schemeLocalFun{\phaseFieldProto}_0}
                    \simpleDeriv{\prime}{\schemeLocalFun{\phaseFieldProto}_0}
                    \fastVar                      	
                    +
                    \left(\simpleDeriv{\prime}{\schemeLocalFun{\phaseFieldProto}_0}\right)^2
                }{\lebesgueM{1}(\fastVar)}
                =
                \frac{1}{2}
                \integral{-\infty}{\infty}{
                    \left(\simpleDeriv{\prime}{\schemeLocalFun{\phaseFieldProto}_0}\right)^2
                }{\lebesgueM{1}(\fastVar)}
                .
            $
            \end{proof}
        \subsubsection{Force terms of $ \grad{\phaseFieldProto}{L^2}{}\ginzburgLandauEnergyFunct{}{}$}
        	A small calculation reveals
            $$
            	\grad{\phaseFieldProto}{L^2}{}\ginzburgLandauEnergyFunct{\phaseFieldProto}{}
                =
                \elComprMod
                \glGenChemPot{\phaseFieldProto}{},
            $$
            and from \eqref{equ:formal asymptotics:CH-L2 gradient expansion:chem pot expansion}
            we have with 
            $
            \simpleDeriv{\prime\prime}{\schemeLocalFun{\phaseFieldProto}_0}
            \concat
            \interfCoords[\phaseFieldParam]{}{}
            -
            \doubleWellPot[\schemeSupscr{\prime}]{\phaseFieldProto_0}
            =0
            $
            and \eqref{equ:formal asymptotics:inner expansion:L2 grad gen willm:scale -2:result}
            $$
                \glGenChemPot{\phaseFieldProto}{}
                =
                \elComprMod
                \simpleDeriv{\prime}{\schemeLocalFun{\phaseFieldProto}_0}
                \concat\interfCoords[\phaseFieldParam]{}{}
                \extMeanCurv{}
                +
                \landauBigo{\phaseFieldParam}.
            $$
            So convergence under the integral follows by 
            Lemma~\ref{lemma:addendum formal asymptotics:concentration}:
            $$
            	\integral{\tubNeighbourhood{\ifCoordsDelta}{\surfaceProto}}{}{
                    \left(\simpleDeriv{\prime}{\schemeLocalFun{\phaseFieldProto}_0}\right)^2
                    \concat\interfCoords[\phaseFieldParam]{}{}
                    \elComprMod
                    \extMeanCurv{}                  	
                    \normal[\phaseFieldProto]{}
                    \cdot
                    \testFunVelocity{}
                }{\lebesgueM{3}}
                +
                \landauBigo{\phaseFieldParam}
                \specConv{\phaseFieldParam\searrow0}
                \tanhIntConst
                \integral{\surfaceProto}{}{
                    \elComprMod
                    \extMeanCurv{}                  	
                    \normal[\surfaceProto]{}
                    \cdot
                    \testFunVelocity{}                  	
                }{\hausdorffM{2}}.
            $$
	  	\subsubsection{Coupling energy force terms}
            We now come to the limit of the terms
            \begin{align*}
            	&\LTProd{\tubNeighbourhood{\ifCoordsDelta}{\interfOne{}}}{
                  	\grad{\membrPhaseField{}{}}{L^2}{}
                    \pfCouplingEnergy{}{}
                }{
                  	\grad{}{}{}\membrPhaseField{}{}
                    \cdot
                    \testFunVelocity{}
                },
                \\
            	&\LTProd{\tubNeighbourhood{\ifCoordsDelta}{\interfTwo{}}}{
                  	\grad{\cortexPhaseField{}{}}{L^2}{}
                    \pfCouplingEnergy{}{}
                }{
                  	\grad{}{}{}\cortexPhaseField{}{}
                    \cdot
                    \testFunVelocity{}
                },
                \\
                \bar{G}_\phaseFieldParam \colonequals
                &-
                \LTProd{\tubNeighbourhood{\ifCoordsDelta}{\interfTwo{}}}{
                  \pDiff{\pfSpeciesOne{}{}}{}{}\pfCouplingIntMembr \meanCurv{\cortexPhaseFieldSym}{}
                  \pfSpeciesOne{}{}
                }{
                    \testFunVelocity{}\cdot\normal[\cortexPhaseField{}{}]{}
                }, 
                \text{\;and\;}
                \\
                \bar{H}_\phaseFieldParam\colonequals
                &-
                \VLTProd{\tubNeighbourhood{\ifCoordsDelta}{\interfTwo{}}}{3}{
                  	\ginzburgLandauEnergyDens{\cortexPhaseFieldSym}{}
                    \orthProjMat{\normal[\cortexPhaseFieldSym]{}}{}
                    \grad{}{}{}
                    \pDiff{\pfSpeciesOne{}{}}{}{}\pfCouplingIntMembr
                    \pfSpeciesOne{}{}
                }{
                  	\testFunVelocity{}
                }
            \end{align*}
            in
            \eqref{equ:formal asymptotics:singular limit:navier-stokes right hand side}.
            \paragraph{Limit of 
                            $ 
                                \LTProd{\tubNeighbourhood{\ifCoordsDelta}{\Gamma}}{
                                    \grad{\membrPhaseField{}{}}{L^2}{}
                                    \pfCouplingEnergy{}{}
                                }{
                                    \grad{}{}{}\membrPhaseField{}{}
                                    \cdot
                                    \testFunVelocity{}
                                }
                            $
                        } 
            Recall \eqref{equ:formal asymtptotics:expanding the L2 gradient of the coupling energy:%
            L2 gradient membrane} with the term abbreviations introduced therein.
            Then
            \begin{equation}
                \label{equ:formal asymptotics:singular limits:L2 gradient membrane coupling energy}
                \LTProd{\tubNeighbourhood{\ifCoordsDelta}{\interfOne{}}}{
                    \grad{\membrPhaseField{}{}}{L^2}{}
                    \pfCouplingEnergy{}{}
                }{
                    \grad{}{}{}\membrPhaseField{}{}
                    \cdot
                    \testFunVelocity{}
                }
                =
                \LTProd{\tubNeighbourhood{\ifCoordsDelta}{\interfOne{}}}{
                    \abs{\grad{}{}{}\membrPhaseField{}{}}  
                  	\left(
                        A_\phaseFieldParam + B_\phaseFieldParam
                    \right)
                }{
                    \normal[\membrPhaseField{}{}]{}
                    \cdot
                    \testFunVelocity{}
                }
                \equalscolon
                \bar{A}_\phaseFieldParam+\bar{B}_\phaseFieldParam.
            \end{equation}
            We further define 
            \begin{equation*}
                \tilde{A}_\phaseFieldParam=
                \integral{\tubNeighbourhood{\ifCoordsDelta}{\interfTwo{}}}{}{
                    \phaseFieldParam^{-1}
                    \left(
                        \frac{1}{2}
                        \left(\simpleDeriv{\prime}{\cortexPhaseFieldIFCIdx{0}{}{}}(\interfCoords{\NOARG}{x})
                        \right)^2
                        +
                        \doubleWellPot{\cortexPhaseFieldIdx{0}{\NOARG}{x}}
                    \right)
                    \pfCouplingEnergyDens{x,\cdot}{
                      	\pfSpeciesOneIdx{0}{}{},\normal[\cortexPhaseFieldIdx{0}{}{}]{}
                    }
                    +
                    \landauBigo{1}
                }{\lebesgueM{3}(x)},
            \end{equation*}
            so that
        	\begin{equation}
                \label{equ:formal asymptotics:singular limit:coupling energy L2 gradient:membr pf grad:first}
                A_\phaseFieldParam = 
                \simpleDeriv{\prime}{\membrPhaseFieldIFCIdx{0}{}{}}
                \concat\interfCoords[\phaseFieldParam]{}{}
                \extMeanCurv{}{}
                \tilde{A}_\phaseFieldParam.
            \end{equation}
            Due to the optimal profile for $ \cortexPhaseFieldIdx{0}{}{} $,
            it holds
            $ 
            	\doubleWellPot{\cortexPhaseFieldIdx{0}{}{}} 
                = 
                \left(
                    \simpleDeriv{\prime}{\cortexPhaseFieldIFCIdx{0}{}{}}
                \right)^2
            $, and applying 
            Lemma~\ref{lemma:addendum formal asymptotics:concentration}, we have as 
            $ \phaseFieldParam\searrow0 $,
            \begin{equation*}
                \begin{split}
                    \tilde{A}_\phaseFieldParam
                    &=
                    \integral{\tubNeighbourhood{\ifCoordsDelta}{\interfTwo{}}}{}{
                        \phaseFieldParam^{-1}
                        \frac{3}{2}
                        \left(
                            \simpleDeriv{\prime}{\cortexPhaseFieldIFCIdx{0}{}{}}
                            (\interfCoords{\NOARG}{x})
                        \right)^2
                        \pfCouplingEnergyDens{x,\cdot}{
                            \pfSpeciesOneIdx{0}{}{},
                            \normal[\cortexPhaseFieldIdx{0}{}{}]{}
                        }
                        +
                        \landauBigo{1}
                    }{\lebesgueM{3}(x)}
                    \\
                    &\rightarrow
                    \frac{3\tanhIntConst}{2}
                    \integral{\interfTwo{}}{}{
                        \pfCouplingEnergyDens{x,\cdot}{
                            \pfSpeciesOneIdx{0}{}{},\normal[\interfTwo{}]{}
                        }
                    }{\hausdorffM{2}(x)}.
                \end{split}
            \end{equation*}
            In order to apply 
            Lemma~\ref{lemma:addendum formal asymptotics:concentration} on $ \bar{A}_\phaseFieldParam $,
            we shall show that for $ y_\phaseFieldParam $ converging to $ y $, 
            $ \tilde{A}_\phaseFieldParam(y_\phaseFieldParam) $ converges.
            To this purpose, we
            write
            $ 
            	\tilde{A}_\phaseFieldParam(y)
                = 
                F_\phaseFieldParam[c(\cdot,y,\pfSpeciesOneIdx{0}{}{},\cortexPhaseFieldIdx{0}{}{})] 
            $
            for all $ y \in \domain $,
            where
            $
                \funSig{F_\phaseFieldParam}{\lebesgueSet{2}{\domain}}{\reals}
            $ 
            is linear and continuous.
            $ \tilde{A}_\phaseFieldParam(y_\phaseFieldParam) $ converging is then equivalent to
            $ F_\phaseFieldParam[c_\phaseFieldParam] = \tilde{A}_\phaseFieldParam(y_\phaseFieldParam) $ converging
            for
            $ c_\phaseFieldParam = c(\cdot,y_\phaseFieldParam,\pfSpeciesOneIdx{0}{}{},\cortexPhaseFieldIdx{0}{}{}) $.
            We calculate
            \begin{equation}
                \label{equ:formal asymptotics:singular limit:momentum balance:banach-steinhaus argument}
                \abs{F_\phaseFieldParam[c_\phaseFieldParam]-F_0[c_0]}
                \leq
                \abs{F_\phaseFieldParam[c_\phaseFieldParam-c_0]}
                +
                \abs{F_\phaseFieldParam[c_0]-F_0[c_0]}.
            \end{equation}
            The previous calculations directly show $ \abs{F_\phaseFieldParam[c_0]-F_0[c_0]} \specConv{}0 $.
            For the other summand, it holds
            $
                \abs{F_\phaseFieldParam[c_\phaseFieldParam-c_0]}
                \leq
                \norm{F_\phaseFieldParam}\LTNorm{\domain}{c_\phaseFieldParam-c_0}.
            $
            Since we have convergence of $ F_\phaseFieldParam[f] $ for every $ f \in \lebesgueSet{2}{\domain} $,
            the 
            Banach-Steinhaus theorem 
            implies $ \norm{F_\phaseFieldParam} < \infty $. Also, the last term converges
            to zero ($ \pfCouplingEnergyDens{}{} $ is continuous in $ y $), and so the left hand side of
            \eqref{equ:formal asymptotics:singular limit:momentum balance:banach-steinhaus argument}
            converges to zero.
            Then, as $ \phaseFieldParam\searrow0 $,
            \begin{equation*}
                \begin{split}
                    \bar{A}_\phaseFieldParam
                    &=
                    \integral{\tubNeighbourhood{\ifCoordsDelta}{\interfOne{}}}{}{
                        \abs{\grad{y}{}{}\membrPhaseField{\NOARG}{y}}
                        A_\phaseFieldParam(y)
                        \normal[\membrPhaseField{}{}]{y}
                        \cdot
                        \testFunVelocity{y}
                    }{\lebesgueM{3}(y)}
                    \\
                    &=
                    \phaseFieldParam^{-1}
                    \integral{\tubNeighbourhood{\ifCoordsDelta}{\interfOne{}}}{}{
                        \left(
                            \simpleDeriv{\prime}{\membrPhaseFieldIFCIdx{0}{}{}}
                            (\interfCoords{\NOARG}{y})
                        \right)^2
                        \extMeanCurv{y}
                        \tilde{A}_\phaseFieldParam(y)
                        \normal[\membrPhaseField{}{}]{y}
                        \cdot
                        \testFunVelocity{y}
                    }{\lebesgueM{3}(y)}
                    +
                    \landauBigo{\phaseFieldParam}
                    \\
                    &\rightarrow
                    \tanhIntConst
                    \integral{\interfOne{}}{}{
                        \meanCurv{\interfOne{}}{y}
                        \tilde{A}_0(y)   	
                        \normal[\interfOne{}]{y}
                        \cdot
                        \testFunVelocity{y}
                    }{\hausdorffM{2}(y)}
                    \\
                    &=
                    \frac{3\tanhIntConst^2}{2}
                    \integral{\interfOne{}}{}{
                        \meanCurv{\interfOne{}}{y}
                        \integral{\interfTwo{}}{}{
                            \pfCouplingEnergyDens{x,y}{
                                \pfSpeciesOneIdx{0}{}{},
                                \normal[\interfTwo{}]{}
                            }
                        }{\hausdorffM{2}(x)}
                        \normal[\interfOne{}]{y}
                        \cdot
                        \testFunVelocity{y}
                    }{\hausdorffM{2}(y)}.
                \end{split}
            \end{equation*}

           Concerning $ \bar{B}_\phaseFieldParam $ one argues analogously, as 
           $ \phaseFieldParam\searrow0 $,
           \begin{equation*}
               \begin{split}
                    \bar{B}_\phaseFieldParam
                    &=
                    \integral{\tubNeighbourhood{\ifCoordsDelta}{\interfOne{}}}{}{
                        \abs{\grad{y}{}{}\membrPhaseField{\NOARG}{y}}
                        B_\phaseFieldParam(y)
                      	\normal[\membrPhaseField{}{}]{y}
                        \cdot
                        \testFunVelocity{y}
                    }{\lebesgueM{3}(y)}
                    \\
                    &=
                    -
                    \phaseFieldParam^{-1}
               	    \int_{\tubNeighbourhood{\ifCoordsDelta}{\interfOne{}}}
                        \left(\simpleDeriv{\prime}{\membrPhaseFieldIFCIdx{0}{}{}}
                        \concat\interfCoords[\phaseFieldParam]{}{}\right)^2
                        \int_{\tubNeighbourhood{\ifCoordsDelta}{\interfTwo{}}}
                            \phaseFieldParam^{-1}
                            \frac{3}{2}
                            \left(
                                \simpleDeriv{\prime}{\cortexPhaseFieldIFCIdx{0}{}{}}
                                (\interfCoords[\phaseFieldParam]{\NOARG}{x})
                            \right)^2
                            \extNormal{y}\cdot
                            \grad{y}{}{}
                            \pfCouplingEnergyDens{x,y}{
                              	\pfSpeciesOneIdx{0}{}{},\normal[\cortexPhaseFieldIdx{0}{}{}]{}
                            }
                            +
                            \landauBigo{1}\;
                        \mathrm{d}\lebesgueM{3}(x)
                        \normal[\membrPhaseField{}{}]{}
                        \cdot
                        \testFunVelocity{}\;
                    \mathrm{d}\lebesgueM{3}(y)
                    \\
                    &\rightarrow
                    -
                    \frac{3\tanhIntConst^2}{2}
                    \integral{\interfOne{}}{}{
                        \normal[\interfOne{}]{}\cdot
                        \integral{\interfTwo{}}{}{
                            \grad{y}{}{}
                            \pfCouplingEnergyDens{x,y}{
                              	\pfSpeciesOneIdx{0}{}{},
                            	\normal[\interfTwo{}]{}
                            }
                        }{\hausdorffM{2}(x)}
                        \normal[\interfOne{}]{}
                        \cdot
                        \testFunVelocity{}
                    }{\hausdorffM{2}(y)}.
                \end{split}
            \end{equation*}
            
            \paragraph{Limit of 
                            $ 
                                \LTProd{\tubNeighbourhood{\ifCoordsDelta }{\Sigma}}{
                                    \grad{\cortexPhaseField{}{}}{L^2}{}
                                    \pfCouplingEnergy{}{}
                                }{
                                  	\grad{}{}{}\cortexPhaseField{}{}\cdot\testFunVelocity{}
                                }
                            $
                        }
            We recall
            \eqref{equ:formal asymptotics:expanding the L2 gradient of the coupling energy:%
            cortex L2 gradient} 
            and write
            \begin{equation}
                \label{equ:formal asymptotics:singular limits:L2 gradient cortex coupling energy}
                \begin{split}
                    \LTProd{\tubNeighbourhood{\ifCoordsDelta}{\interfTwo{}}}{
                        \grad{\cortexPhaseField{}{}}{\lebesgueSet{2}{\domain}}{}
                        \pfCouplingEnergy{}{}
                    }{
                        \grad{}{}{}\cortexPhaseField{}{}\cdot\testFunVelocity{}
                    }
                    =
                    \LTProd{\tubNeighbourhood{\ifCoordsDelta}{\interfTwo{}}}{
                        \abs{\grad{}{}{}\cortexPhaseField{}{}}
                        \left(C_\phaseFieldParam + D_\phaseFieldParam + E_\phaseFieldParam\right)
                    }{
                        \normal[\cortexPhaseField{}{}]{}\cdot\testFunVelocity{}
                    }       
                    +
                    \landauBigo{\phaseFieldParam}
                    =
                    \bar{C}_\phaseFieldParam+\bar{D}_\phaseFieldParam+\bar{E}_\phaseFieldParam
                    +
                    \landauBigo{\phaseFieldParam}.
                \end{split}
            \end{equation}
            The term $ \bar{C}_\phaseFieldParam $ is treated just like $ \bar{A}_\phaseFieldParam $, so we 
            obtain in the limit $ \phaseFieldParam\searrow0 $
            \begin{equation*}
                \begin{split}
                    \bar{C}_\phaseFieldParam
                    &=
                    \integral{\tubNeighbourhood{\ifCoordsDelta}{\interfTwo{}}}{}{
                        \left(
                          	-\phaseFieldParam\laplacian{}{}{}\cortexPhaseFieldSym
                            +
                            \phaseFieldParam^{-1}\doubleWellPot[\schemeSupscr{\prime}]{
                              	\cortexPhaseFieldSym
                            }
                        \right)
                        \integral{\tubNeighbourhood{\ifCoordsDelta}{\interfOne{}}}{}{
                            \ginzburgLandauEnergyDens{\membrPhaseField{}{}}{}(y)
                            \pfCouplingEnergyDens{x,y}{
                              	\pfSpeciesOne{}{},\normal[\cortexPhaseField{}{}]{}
                            }
                        }{\lebesgueM{3}(y)}
                        \normal[\cortexPhaseField{}{}]{x}
                        \cdot
                        \testFunVelocity{x}
                    }{\lebesgueM{3}(x)}
                    \\
                    &\rightarrow
                    \frac{3\tanhIntConst^2}{2}
                    \integral{\interfTwo{}}{}{
                        \meanCurv{\interfTwo{}}{x}
                        \integral{\interfOne{}}{}{
                            \pfCouplingEnergyDens{x,y}{\pfSpeciesOneIdx{0}{}{},\normal[\interfTwo{}]{}}
                        }{\hausdorffM{2}(y)}
                        \normal[\interfTwo{}]{x}
                        \cdot
                        \testFunVelocity{x}
                    }{\hausdorffM{2}(x)}.
                \end{split}
            \end{equation*}
            Using the expansion of $ D_\phaseFieldParam $, we compute further
            \begin{equation*}
                \begin{split}
                    \bar{D}_\phaseFieldParam 
                    &=
                    -
                    \integral{\tubNeighbourhood{\ifCoordsDelta}{\interfTwo{}}}{}{
                        \abs{\grad{}{}{}\cortexPhaseField{\NOARG}{}}
                        \phaseFieldParam
                        \grad{}{}{}\cortexPhaseField{}{}
                        \cdot
                        \integral{\tubNeighbourhood{\ifCoordsDelta}{\interfOne{}}}{}{
                            \ginzburgLandauEnergyDens{\membrPhaseField{}{}}{}(y)
                            \grad{x}{}{
	                            \pfCouplingEnergyDens{x,y}{
                                  	\pfSpeciesOne{}{},\normal[\cortexPhaseField{}{}]{}
                                }
                            }
                        }{\lebesgueM{3}(y)}
                        \normal[\cortexPhaseField{}{}]{}
                        \cdot
                        \testFunVelocity{}                  	
                    }{\lebesgueM{3}(x)}
                    \\
                    &=
                    -\integral{\tubNeighbourhood{\ifCoordsDelta}{\interfTwo{}}}{}{
                        \abs{\grad{}{}{}\cortexPhaseField{}{}}
                        \simpleDeriv{\prime}{\cortexPhaseFieldIFCIdx{0}{}{}}
                        \concat \interfCoords[\phaseFieldParam]{}{}
                        \extNormal{}
                        \cdot
                        \integral{\tubNeighbourhood{\ifCoordsDelta}{\interfOne{}}}{}{
                            \ginzburgLandauEnergyDens{\membrPhaseField{}{}}{}(y)
                            \grad{x}{}{}\pfCouplingEnergyDens{x,y}{
                                \pfSpeciesOne{}{},
                                \normal[\cortexPhaseField{}{}]{}
                            }
                        }{\lebesgueM{3}(y)}
                        \normal[\cortexPhaseField{}{}]{}
                        \cdot
                        \testFunVelocity{}                        
                    }{\lebesgueM{3}(x)}
                    \\
                    &-
                    \integral{\tubNeighbourhood{\ifCoordsDelta}{\interfTwo{}}}{}{
                        \abs{\grad{}{}{}\cortexPhaseField{}{}}
                        \simpleDeriv{\prime}{\cortexPhaseFieldIFCIdx{0}{}{}}
                        \concat \interfCoords[\phaseFieldParam]{}{}
                        \extNormal{}
                        \cdot
                        \integral{\tubNeighbourhood{\ifCoordsDelta}{\interfOne{}}}{}{
                            \ginzburgLandauEnergyDens{\membrPhaseField{}{}}{}(y)
                            \pDiff{\pfSpeciesOne{}{}}{}{}
                            \pfCouplingEnergyDens{x,y}{
                                \pfSpeciesOne{}{},
                                \normal[\cortexPhaseField{}{}]{}
                            }
                            \grad{x}{}{}\pfSpeciesOneIdx{0}{}{}
                        }{\lebesgueM{3}(y)}
                        \normal[\cortexPhaseField{}{}]{}
                        \cdot
                        \testFunVelocity{}                        
                    }{\lebesgueM{3}(x)}
                    +
                    \landauBigo{\phaseFieldParam}
                    \\
                    &=
                    -\integral{\tubNeighbourhood{\ifCoordsDelta}{\interfTwo{}}}{}{
                        \phaseFieldParam^{-1}
                        \left(
                            \simpleDeriv{\prime}{\cortexPhaseFieldIFCIdx{0}{}{}}
                        \right)^2
                        \concat \interfCoords[\phaseFieldParam]{}{}
                        \extNormal{}
                        \cdot
                        \integral{\tubNeighbourhood{\ifCoordsDelta}{\interfOne{}}}{}{
                            \phaseFieldParam^{-1}
                                \frac{3}{2}
                                \left(
                                    \simpleDeriv{\prime}{\membrPhaseFieldIFCIdx{0}{}{}}
                                    (\interfCoords[\phaseFieldParam]{\NOARG}{y})
                                \right)^2
                            \grad{x}{}{}\pfCouplingEnergyDens{x,y}{
                                \pfSpeciesOneIdx{0}{}{},
                                \normal[\cortexPhaseFieldIdx{0}{}{}]{}
                            }
                        }{\lebesgueM{3}(y)}
                        \normal[\cortexPhaseField{}{}]{}
                        \cdot
                        \testFunVelocity{}                        
                    }{\lebesgueM{3}(x)}
                    \\
                    &-
                    \integral{\tubNeighbourhood{\ifCoordsDelta}{\interfTwo{}}}{}{
                        \phaseFieldParam^{-1}
                        \left(
                            \simpleDeriv{\prime}{\cortexPhaseFieldIFCIdx{0}{}{}}
                        \right)^2
                        \concat \interfCoords[\phaseFieldParam]{}{}
                        \extNormal{}
                        \cdot
                        \integral{\tubNeighbourhood{\ifCoordsDelta}{\interfOne{}}}{}{
                            \phaseFieldParam^{-1}
                                \frac{3}{2}
                                \left(
                                    \simpleDeriv{\prime}{\membrPhaseFieldIFCIdx{0}{}{}}
                                    (\interfCoords[\phaseFieldParam]{\NOARG}{y})
                                \right)^2
                            \pDiff{\pfSpeciesOne{}{}}{}{}
                            \pfCouplingEnergyDens{x,y}{
                                \pfSpeciesOneIdx{0}{}{},
                                \normal[\cortexPhaseFieldIdx{0}{}{}]{}
                            }
                            \grad{x}{}{}\pfSpeciesOneIdx{0}{}{}
                        }{\lebesgueM{3}(y)}
                        \normal[\cortexPhaseField{}{}]{}
                        \cdot
                        \testFunVelocity{}                        
                    }{\lebesgueM{3}(x)}
                    +
                    \landauBigo{\phaseFieldParam}
                    \\
                    &\overset{(1)}{=}
                    -\integral{\tubNeighbourhood{\ifCoordsDelta}{\interfTwo{}}}{}{
                        \phaseFieldParam^{-1}
                        \left(
                            \simpleDeriv{\prime}{\cortexPhaseFieldIFCIdx{0}{}{}}
                        \right)^2
                        \concat \interfCoords[\phaseFieldParam]{}{}
                        \extNormal{}
                        \cdot
                        \integral{\tubNeighbourhood{\ifCoordsDelta}{\interfOne{}}}{}{
                            \phaseFieldParam^{-1}
                                \frac{3}{2}
                                \left(
                                    \simpleDeriv{\prime}{\membrPhaseFieldIFCIdx{0}{}{}}
                                    (\interfCoords[\phaseFieldParam]{\NOARG}{y})
                                \right)^2
                            \grad{x}{}{}\pfCouplingEnergyDens{x,y}{
                                \pfSpeciesOneIdx{0}{}{},
                                \normal[\cortexPhaseFieldIdx{0}{}{}]{}
                            }
                        }{\lebesgueM{3}(y)}
                        \normal[\cortexPhaseField{}{}]{}
                        \cdot
                        \testFunVelocity{}                        
                    }{\lebesgueM{3}(x)}
                    +
                    \landauBigo{\phaseFieldParam}
                    \\
                    &\rightarrow
                    -\frac{3\tanhIntConst^2}{2}
                    \integral{\interfTwo{}}{}{
                      	\normal[\interfTwo{}]{}
                        \cdot
                        \integral{\interfOne{}}{}{
                          	\grad{x}{}{}
                            \pfCouplingEnergyDens{x,y}{
                              	\pfSpeciesOneIdx{0}{}{},
                                \normal[\interfTwo{}]{}
                            }
                        }{\hausdorffM{2}(y)}
                        \normal[\interfTwo{}]{}
                        \cdot
                        \testFunVelocity{}
                    }{\hausdorffM{2}(x)}\;\text{as}\;\phaseFieldParam\searrow 0,
                \end{split}
            \end{equation*}
            where for $ (1) $, we observe that due to
            \eqref{equ:formal asymptotics:inner expansion:species to leading order:%
            constant in normal direction}, 
            $ \grad{x}{}{}\pfSpeciesOneIdx{0}{}{}
            = \grad{\interfTwo{}_{\sdist{}{}}}{}{}\pfSpeciesOneIdx{0}{}{} $, thus
            $ \extNormal{}\cdot\grad{x}{}{}\pfSpeciesOneIdx{0}{}{} = 0 $.

            At last, we turn to $ \bar{E}_\phaseFieldParam $:
            \begin{equation*}
                \begin{split}
                    \bar{E}_\phaseFieldParam
                    &=
                    -
                    \integral{\domain}{}{
                        \phaseFieldParam^{-1}
                        \frac{3}{2}
                        \left(\simpleDeriv{\prime}{\cortexPhaseFieldIFCIdx{0}{}{}}\right)^2
                        \concat
                        \interfCoords[\phaseFieldParam]{}{}
                        \integral{\domain}{}{
                            \phaseFieldParam^{-1}
                            \frac{3}{2}
                            \left(
                                \simpleDeriv{\prime}{\membrPhaseFieldIFCIdx{0}{}{}}
                                (\interfCoords[\phaseFieldParam]{\NOARG}{y})
                            \right)^2
                            \diver{\interfTwo{}}{}{\grad{\normal{}}{}{}\pfCouplingEnergyDens{}{}}
                            (x,y,\pfSpeciesOneIdx{0}{}{},\normal[\cortexPhaseFieldIdx{0}{}{}]{})
                        }{\lebesgueM{3}(y)}
                        \normal[\cortexPhaseField{}{}]{}
                        \cdot
                        \testFunVelocity{}
                    }{\lebesgueM{3}(x)}
                    \\
                    &-
                    \integral{\domain}{}{
                        \phaseFieldParam^{-1}
                        \frac{3}{2}
                        \left(\simpleDeriv{\prime}{\cortexPhaseFieldIFCIdx{0}{}{}}\right)^2
                        \concat
                        \interfCoords[\phaseFieldParam]{}{}
                        \integral{\domain}{}{
                            \phaseFieldParam^{-1}
                            \frac{3}{2}
                            \left(
                                \simpleDeriv{\prime}{\membrPhaseFieldIFCIdx{0}{}{}}
                                (\interfCoords[\phaseFieldParam]{\NOARG}{y})
                            \right)^2
                            \grad{\normal{}}{}{}\pfCouplingEnergyDens{x,y}{
                              	\pfSpeciesOneIdx{0}{}{},\normal[\cortexPhaseFieldIdx{0}{}{}]{}
                            }
                            \cdot
                            \extMeanCurv{x}
                            \extNormal{x}
                        }{\lebesgueM{3}(y)}
                        \normal[\cortexPhaseField{}{}]{}
                        \cdot
                        \testFunVelocity{}
                    }{\lebesgueM{3}(x)}
                    \\
                    &\rightarrow
                    -\left(\frac{3\tanhIntConst}{2}\right)^2
                    \integral{\interfTwo{}}{}{
                      	\integral{\interfOne{}}{}{
                            \diver{\interfTwo{}}{}{\grad{\normal{}}{}{}\pfCouplingEnergyDens{}{}}
                            (x,y,\pfSpeciesOneIdx{0}{\NOARG}{},\normal[\interfTwo{}]{})
                            +
                            \grad{\normal{}}{}{}
                            \pfCouplingEnergyDens{x,y}{
                              	\pfSpeciesOneIdx{0}{}{},\normal[\interfTwo{}]{}
                            }
                            \cdot
                            \meanCurv{\interfTwo{}}{x}
                            \normal[\interfTwo{}]{x}                            
                        }{\hausdorffM{2}(y)}
                        \normal[\interfTwo{}]{}
                        \cdot
                        \testFunVelocity{}
                    }{\hausdorffM{2}(x)}\;\text{as}\;\phaseFieldParam\searrow 0.
                \end{split}
            \end{equation*}

        \paragraph{Limit of $ \bar{G}_\phaseFieldParam$}
        Using the expansion we computed before, we obtain
        \begin{equation*}
            \begin{split}
                \bar{G}_\phaseFieldParam 
                &=-\integral{\tubNeighbourhood{\ifCoordsDelta}{\interfTwo{}}}{}{
                    \pfSpeciesOneIdx{0}{}{}
                    \phaseFieldParam^{-1}
                    \left(
                        \simpleDeriv{\prime}{\cortexPhaseFieldIFCIdx{0}{}{}}
                        \concat\interfCoords[\phaseFieldParam]{}{}
                    \right)^2
                    \extMeanCurv{}
                    \extNormal{}
                    \cdot
                    \testFunVelocity{}
                    \integral{\tubNeighbourhood{\ifCoordsDelta}{\interfOne{}}}{}{
                        \phaseFieldParam^{-1}
                        \frac{3}{2}
                        \left(
                            \simpleDeriv{\prime}{\membrPhaseFieldIFCIdx{0}{}{}}
                            \concat
                            \interfCoords[\phaseFieldParam]{}{}
                        \right)^2
                        \pDiff{\pfSpeciesOneIdx{}{}{}}{}{}
                        \pfCouplingEnergyDens{x,y}{
                          	\pfSpeciesOneIdx{0}{}{},
                            \normal[\cortexPhaseFieldIdx{0}{}{}]{}
                        }
                    }{\lebesgueM{3}(y)}
                }{\lebesgueM{3}(x)}
                +
                \landauBigo{\phaseFieldParam}
                \\
                &\rightarrow
                -\frac{3\tanhIntConst^2}{2}
                \integral{\interfTwo{}}{}{
                  	\pfSpeciesOneIdx{0}{}{}
                    \meanCurv{\interfTwo{}}{}
                    \normal[\interfTwo{}]{}
                    \cdot
                    \testFunVelocity{}
                    \integral{\interfOne{}}{}{
                        \pDiff{\pfSpeciesOne{}{}}{}{}
                        \pfCouplingEnergyDens{x,y}{
                          	\pfSpeciesOneIdx{0}{}{},
                            \normal[\interfTwo{}]{}
                        }
                    }{\hausdorffM{2}(y)}
                }{\hausdorffM{2}(x)}\;\text{as}\;\phaseFieldParam\searrow 0.
            \end{split}
        \end{equation*}

        \paragraph{Limit of $ \bar{H}_\phaseFieldParam $}
            On the expansion derived before, we use
            $ 
            	\doubleWellPot{\cortexPhaseFieldIdx{0}{}{}} 
                = 
                \left(
                    \simpleDeriv{\prime}{\cortexPhaseFieldIFCIdx{0}{}{}}
                \right)^2
            $,
            and then pass to the limit $\phaseFieldParam\searrow 0$:
            \begin{align*}
                \bar{H}_\phaseFieldParam
                &=
                -\int_{\tubNeighbourhood{\ifCoordsDelta}{\interfTwo{}}}
                \phaseFieldParam^{-1}
                \frac{3}{2}
                \left(
                    \simpleDeriv{\prime}{\cortexPhaseFieldIFCIdx{0}{}{}}
                \right)^2
                \concat\interfCoords[\phaseFieldParam]{\NOARG}{x}
                \integral{\tubNeighbourhood{\ifCoordsDelta}{\interfOne{}}}{}{
                    \phaseFieldParam^{-1}
                    \frac{3}{2}
                    \left(
                        \simpleDeriv{\prime}{\membrPhaseFieldIFCIdx{0}{\NOARG}{}}
                    \right)^2
                    \concat\interfCoords[\phaseFieldParam]{\NOARG}{y}
                    \grad{\interfTwo{}_{\sdist{}{}}}{}{\pDiff{\pfSpeciesOne{}{}}{}{} \pfCouplingEnergyDens{}{}
                    (\cdot,y,\pfSpeciesOneIdx{0}{}{},\normal[\cortexPhaseFieldIdx{0}{}{}]{})}
                    \cdot\testFunVelocity{}
                    \pfSpeciesOneIdx{0}{}{}
                }{\lebesgueM{3}(y)}   
                \mathrm{d}\lebesgueM{3}(x)
                \\
                &+\landauBigo{\phaseFieldParam}
                \\
                &\rightarrow
                -
                \left(\frac{3\tanhIntConst}{2}\right)^2
                \integral{\interfTwo{}}{}{
                  	\integral{\interfOne{}}{}{
                        \grad{\interfTwo{}_{\sdist{}{}}}{}{\pDiff{\pfSpeciesOne{}{}}{}{} \pfCouplingEnergyDens{}{}
                        (\cdot,y,\pfSpeciesOneIdx{0}{}{},\normal[\cortexPhaseFieldIdx{0}{}{}]{})}
                        \cdot\testFunVelocity{}
                        \pfSpeciesOneIdx{0}{}{}                      	
                    }{\hausdorffM{2}}
                }{\hausdorffM{2}}.
            \end{align*}

        \subsubsection{Deriving the jump conditions}
        We now observe
        that all the previously computed limits have equal, corresponding terms in the right hand sides of
        \eqref{equ:modelling:sharp interface classical PDE model:normal stress jump if one}
        and
        \eqref{equ:modelling:sharp interface classical PDE model:normal stress jump if two}.
        Note that 
        these appear with a minus on the right hand side of the momentum balance, so
        we negate them here accordingly.
        We start by labelling them (here in variational form):
        \begin{equation*}
            \begin{split}
                \bar{I}_0
                &=
                \integral{\interfOne{}}{}{
                (
                    -
                    \grad{y}{}{}\couplingIntCortex^0\cdot\normal[\interfOne{}]{}
                    +
                    \meanCurv{\interfOne{}}{}
                    \couplingIntCortex^0
                )
                \testFunVelocity{}
                \cdot
                \normal[\interfOne{}]{}
                }{\hausdorffM{2}},
                \\
                \bar{J}_0
                &=
                \integral{\interfTwo{}}{}{
                    \left(
	                -
                    \grad{x}{}{}\couplingIntMembr^0\cdot\normal[\interfTwo{}]{}
                    +
                    \meanCurv{\interfTwo{}}{}
                    \couplingIntMembr^0
                    \right)
                    \testFunVelocity{}\cdot
                    \normal[\interfTwo{}]{}
                }{\hausdorffM{2}},
                \\
                \bar{K}_0
                &=
                -
                \integral{\interfTwo{}}{}{
                    \pDiff{\speciesOne{}{}}{}{}
                    \couplingIntMembr^0
                    \meanCurv{\interfTwo{}}{}
                    \speciesOne{}{}
                    \testFunVelocity{}\cdot
                    \normal[\interfTwo{}]{}
                }{\hausdorffM{2}},
                \\
                \bar{L}_0
                &=
                -
                \integral{\interfTwo{}}{}{
                    \grad{\interfTwo{}}{}{
                    \pDiff{\pfSpeciesOne{}{}}{}{}\couplingIntMembr^0
                    }
                    \cdot
                    \pfSpeciesOne{}{}
                    \testFunVelocity{}
                }{\hausdorffM{2}},
                \\
                \bar{M}_0
                &=
                -
                \integral{\interfTwo{}}{}{
                    \left(
                        \diver{\interfTwo{}}{}{
                            \grad{\normal{}}{}{}\couplingIntMembr^0
                        }
                        +
                        \meanCurv{\interfTwo{}}{}
                        \left(
                        \grad{\normal{}}{}{}\couplingIntMembr^0
                        \cdot
                        \normal[\interfTwo{}]{}
                        \right)
                    \right)
                    \testFunVelocity{}\cdot
                    \normal[\interfTwo{}]{}
                }{\hausdorffM{2}}.                
            \end{split}
        \end{equation*}

        We see immediately that $ \bar{B}_0 + \bar{A}_0 = \frac{3Z^2}{2}\bar{I}_0 $.
        Further,
        \begin{equation*}
            \begin{split}
		        \bar{D}_0 + \bar{C}_0
                &=
                -
                \frac{3\tanhIntConst^2}{2}
                \integral{\interfTwo{}}{}{
                    \integral{\interfOne{}}{}{
                        \normal[\interfTwo{}]{}
                        \cdot
                        \grad{x}{}{}\pfCouplingEnergyDens{}{}
	                    \normal[\interfTwo{}]{}\cdot\testFunVelocity{}
                    }{\hausdorffM{2}}
                }{\hausdorffM{2}}
                +
                \frac{3\tanhIntConst^2}{2}
                \integral{\interfTwo{}}{}{
                    \meanCurv{\interfTwo{}}{}
                    \integral{\interfOne{}}{}{
                        \pfCouplingEnergyDens{}{}
                    }{\hausdorffM{2}}
                    \normal[\interfTwo{}]{}
                    \cdot
                    \testFunVelocity{}
                }{\hausdorffM{2}}
                = \frac{3Z^2}{2}\bar{J}_0.
            \end{split}
        \end{equation*}
        Clearly, $ \bar{G}_0 = \frac{3Z^2}{2}\bar{K}_0 $, 
        $\bar{H}_0 = \left( \frac{3\tanhIntConst}{2} \right)^2\bar{L}_0 $, and 
        $ \bar{E}_0 = \left(\frac{3Z}{2}\right)^2\bar{M}_0$.

        Equating \eqref{equ:formal asymptotics:sharp interface limit:momentum balance:%
        recov interfacial jump stress tensor} on the left and the limit of
        Lemma~\ref{lemma:formal asymptotics:singular limit:canham-helfrich force terms:approximation}, 
        as well as the terms for the Ginzburg-Landau energy gradient and
        $ \bar{A}_0 + \bar{B}_0 + \bar{C}_0 + \bar{D}_0 + \bar{E}_0 + \bar{G}_0 + \bar{H}_0 $
        on the right, we find 
        \eqref{equ:modelling:sharp interface classical PDE model:normal stress jump if one}
        and
        \eqref{equ:modelling:sharp interface classical PDE model:normal stress jump if two}.
    \subsection{Phase field evolution equations}
        The equations
        \eqref{equ:modelling:phase field resulting PDEs:first interface evolution}
        and 
        \eqref{equ:modelling:phase field:resulting PDEs:second interface evolution}
        are just the tautologies $ 0  = 0 $ in the outer region, so we are only concerned
        with them close to the boundary layers.
        On their left hand sides, we find at leading order $ \phaseFieldParam^{-1} $
        \begin{equation}
            \label{equ:formal asymptotics:sharp interface limit:pf evol equations:
            lhs}
            \begin{split}
                -\shapeTransfVelNormal^{\surfaceProto}
                \simpleDeriv{\prime}{\schemeLocalFun{\phaseFieldProto}_0}
                \concat\interfCoords[\phaseFieldParam]{}{}
                +
                \pfVelocityIdx{0}{}{}
                \cdot
                \extNormal{}
                \simpleDeriv{\prime}{\schemeLocalFun{\phaseFieldProto}_0}
                \concat\interfCoords[\phaseFieldParam]{}{}
            \end{split}
        \end{equation}
        for $ \phaseFieldProto \in \{\membrPhaseFieldSym,\cortexPhaseFieldSym\} $ 
        with corresponding boundary layer $ \surfaceProto\in\{\interfOne{},\interfTwo{}\} $.

    	From the energy inequality, the following bound holds,
        \begin{equation}
            \label{equ:formal asymptotics:sharp interface limit:pf evol equations:%
            bound on energy gradient}
            \phaseFieldParam^{\alpha}
            \integral{0}{\finTime}{
              	\integral{\tubNeighbourhood{\ifCoordsDelta}{\surfaceProto}}{}{
	            	\abs{
                      	\grad{}{}{}\grad{\phaseFieldProto}{L^2}{}\freeEnergyFunct{}{}
                    }^2
                }{\lebesgueM{3}}
            }{t}
            \in \landauBigo{1}.
        \end{equation}
        We have found in Section~\ref{sec:Expanding the L2 gradient of the canham-helfrich energy} that
        $ \grad{\phaseFieldProto}{L^2}{}\freeEnergyFunct{\phaseFieldProto}{} \in
        \landauBigo{1} $, so
        $$
        	\grad{}{}{}\grad{\phaseFieldProto}{L^2}{}\freeEnergyFunct{}{}
            =
            \sum_{i=-1}^2 \phaseFieldParam^i \hat{f}_i\concat\interfCoords[\phaseFieldParam]{}{}
            +
            \landauBigo{\phaseFieldParam^3}.
        $$
        Thus
        $$
        	\abs{
            \widehat{
        	\grad{}{}{}\grad{\phaseFieldProto}{L^2}{}\freeEnergyFunct{}{}}
            }^2
            =
            \phaseFieldParam^{-2}
            \hat{f}_{-1}^2
            +
            2\phaseFieldParam^{-1}\hat{f}_{-1}\hat{f}_0
            +
            \left(\hat{f}_0^2 + \hat{f}_{-1}\hat{f}_1\right)
            +
            \phaseFieldParam\left(2\hat{f}_{-1}\hat{f}_2 + 2\hat{f}_0\hat{f}_1\right)
            +
            \phaseFieldParam^2\left(\hat{f}_1^2 + 2\hat{f}_{-1}\hat{f}_3 + 2\hat{f}_0\hat{f}_2
            \right)
            +
            \landauBigo{1}.
        $$
        Equation \eqref{equ:formal asymptotics:sharp interface limit:pf evol equations:%
            bound on energy gradient} directly implies
        \begin{equation*}
            \integral{0}{\finTime}{
              	\integral{
                  	-\frac{\ifCoordsDelta}{\phaseFieldParam}
                }{
                  	\frac{\ifCoordsDelta}{\phaseFieldParam}
                }{
                  	\integral{\surfaceProto_{\phaseFieldParam\fastVar}}{}{
		            	\abs{
                          	\widehat{\grad{}{}{}\grad{\phaseFieldProto}{L^2}{}\freeEnergyFunct{}{}}
                        }^2
                        (\orthProj{\surfaceProto}{\sigma},\fastVar)
                    }{\hausdorffM{2}(\sigma)}
                }{\fastVar}
            }{t}
            \in \landauBigo{\phaseFieldParam^{-\alpha-1}}.
        \end{equation*}
        Since $ \alpha < 1 $ (required in Section~\ref{sec:Expanding the L2 gradient of the canham-helfrich energy}), 
        $ \hat{f}_{-1} = 0 $, so $ \widehat{\grad{}{}{}\grad{\phaseFieldProto}{L^2}{}
        \freeEnergyFunct{}{}} \in \landauBigo{1} $.
        Further, $ \diver{}{}{\phaseFieldParam^\alpha\grad{}{}{\grad{\phaseFieldProto}{L^2}{}
          \freeEnergyFunct{}{}}} \in \landauBigo{\phaseFieldParam^{-1+\alpha}} $,
        which is of lower order than the left hand side 
        \eqref{equ:formal asymptotics:sharp interface limit:pf evol equations:
        lhs} for $ \alpha > 0 $, so we obtain from the phase field evolution to leading order
        \begin{equation}
            \pfVelocityIdx{0}{}{}
            \cdot
            \extNormal{}
            =
            \shapeTransfVelNormal^{\surfaceProto}
        \end{equation}
        meaning that the interface $ \surfaceProto $ is driven purely by the 
        fluid's velocity in normal direction, and this is equivalent to the 
        Hamilton-Jacobi equations 
        \eqref{equ:modelling:sharp interface classical PDE: HJE membr}, and
        \eqref{equ:modelling:sharp interface classical PDE: HJE cortex}.
    \subsection{Species subsystem}
    	Like the phase field equations, the species subsystem 
        \eqref{equ:modelling:phase field:resulting PDEs:reaction-diffusion of species one}
        \eqref{equ:modelling:phase field:resulting PDEs:reaction-diffusion of species two}
        is meaningless in the outer region. 
        For reasons of symmetry, it suffices to conduct the asymptotic analysis for the 
        equation of $ \pfSpeciesOne{}{} $: it carries over to $ \pfSpeciesTwo{}{} $ 
        verbatim.
        
        We start our analysis by expanding the term 
        $
        	\diver{}{}{
              	\ginzburgLandauEnergyDens{\cortexPhaseFieldSym}{}
              	\speciesOneDiffusiv{}\grad{}{}{}\pfSpeciesOne{}{}
            }
        $. W.l.o.g., we assume that $ \simpleDeriv{\prime}{\schemeLocalFun{\speciesOneDiffusiv{}}}=0 $.
        First we derive from
        \eqref{equ:appendix:formally matched asymptotics:%
        gradient rescaled distance} and using \eqref{equ:formal asymptotics:inner expansion:%
        species to leading order:%
        constant in normal direction}
        \begin{align*}
        	\speciesOneDiffusiv{}\grad{}{}{}\pfSpeciesOne{}{}
            &=
            \phaseFieldParam^{-1}\speciesOneDiffusiv{}
            \simpleDeriv{\prime}{
            	\pfSpeciesOneIFCIdx{0}{}{}
            }
            \concat\interfCoords[\phaseFieldParam]{}{}
            \extNormal{}
            +
            \speciesOneDiffusiv{}{}\grad{\interfTwo{}_{\sdist{}{}}}{}{}
            \pfSpeciesOneIdx{0}{}{}
            +
            \speciesOneDiffusiv{}\simpleDeriv{\prime}{\pfSpeciesOneIFCIdx{1}{}{}}
            \concat\interfCoords[\phaseFieldParam]{}{}
            \extNormal{}
            +
            \phaseFieldParam\speciesOneDiffusiv{}
            \grad{\interfTwo{}_{\sdist{}{}}}{}{}\pfSpeciesOneIdx{1}{}{}
            +
            \landauBigo{\phaseFieldParam^2}
            \\
            &=
            \speciesOneDiffusiv{}{}\grad{\interfTwo{}_{\sdist{}{}}}{}{}
            \pfSpeciesOneIdx{0}{}{}
            +
            \speciesOneDiffusiv{}\simpleDeriv{\prime}{\pfSpeciesOneIFCIdx{1}{}{}}
            \concat\interfCoords[\phaseFieldParam]{}{}
            \extNormal{}
            +
            \phaseFieldParam\speciesOneDiffusiv{}
            \grad{\interfTwo{}_{\sdist{}{}}}{}{}\pfSpeciesOneIdx{1}{}{}
            +
            \landauBigo{\phaseFieldParam^2}.
        \end{align*}
        With 
        \eqref{equ:formal asymptotics:expansion of the L2 gradient of the coupling energy:%
            helping expansions:ginzburg landau density}
        and thanks to
        \eqref{equ:formal asymptotics:inner expansion:L2 grad gen willm:%
        scale -2:result} and $ \cortexPhaseFieldIFCIdx{0}{}{} $ being the optimal profile,
        we have also
        \begin{align*}
            \ginzburgLandauEnergyDens{\cortexPhaseFieldSym}{}
            =
            \phaseFieldParam^{-1}\frac{3}{2}\left(\simpleDeriv{\prime}{\cortexPhaseFieldIFCIdx{0}{}{}}\right)^2
            \concat\interfCoords[\phaseFieldParam]{}{}
            +
            \landauBigo{\phaseFieldParam}.
        \end{align*}
        Multiplying both equations gives
        \begin{equation}
            \label{equ:formal asymptotics:sharp interface limit:species subsystem:%
            expansion of flux}
            \ginzburgLandauEnergyDens{\cortexPhaseFieldSym}{}
            \speciesOneDiffusiv{}\grad{}{}{}\pfSpeciesOne{}{}
            =
            \phaseFieldParam^{-1}\frac{3}{2}
            \left(\simpleDeriv{\prime}{\cortexPhaseFieldIFCIdx{0}{}{}}\right)^2
            \concat\interfCoords[\phaseFieldParam]{}{}
            \speciesOneDiffusiv{}
            \left(
              	\grad{\interfTwo{}_{\sdist{}{}}}{}{}\pfSpeciesOneIdx{0}{}{}
                +
                \simpleDeriv{\prime}{\pfSpeciesOneIFCIdx{1}{}{}}
                \concat\interfCoords[\phaseFieldParam]{}{}
                \extNormal{}
            \right)
            +
            \frac{3}{2}\left(\simpleDeriv{\prime}{\cortexPhaseFieldIFCIdx{0}{}{}}\right)^2
            \concat\interfCoords[\phaseFieldParam]{}{}
            \speciesOneDiffusiv{}
            \grad{\interfTwo{}_{\sdist{}{}}}{}{}\pfSpeciesOneIdx{1}{}{}
            +
            \landauBigo{\phaseFieldParam}.
        \end{equation}
        We also expand 
        \begin{equation}
            \label{equ:formal asymptotics:sharp interface limit:species subsystem:%
            expansion of add flux}
            \begin{split}
            \diver{}{}{
              	\ginzburgLandauEnergyDens{\cortexPhaseFieldSym}{}
                \pfSpeciesOne{}{}
                \pfVelocity{}{}_\tau
            }
            &=
            \diver{}{}{
                \phaseFieldParam^{-1}
                \frac{3}{2}
                \left(\simpleDeriv{\prime}{\cortexPhaseFieldIFCIdx{0}{}{}}\right)^2
                \concat\interfCoords[\phaseFieldParam]{}{}
                \pfSpeciesOneIdx{0}{}{}
                (\pfVelocityIdx{0}{}{})_\tau
                +
                \frac{3}{2}
                \left(\simpleDeriv{\prime}{\cortexPhaseFieldIFCIdx{0}{}{}}\right)^2
                \concat\interfCoords[\phaseFieldParam]{}{}
                \left(
                    \pfSpeciesOneIdx{1}{}{}
                    (\pfVelocityIdx{0}{}{})_\tau                
                    +
                    \pfSpeciesOneIdx{0}{}{}
                    (\pfVelocityIdx{1}{}{})_\tau
                \right)
            	+
                \landauBigo{\phaseFieldParam}
            }
            \\
            &=
            \phaseFieldParam^{-2}
            \simpleDeriv{\prime}{\left(
                \frac{3}{2}
                \left(\simpleDeriv{\prime}{\cortexPhaseFieldIFCIdx{0}{}{}}\right)^2
                \pfSpeciesOneIFCIdx{0}{}{}
                (\pfVelocityIFCIdx{0}{}{})_\tau
            \right)
            }
            \concat\interfCoords[\phaseFieldParam]{}{}
            \cdot\extNormal{}
            +
            \phaseFieldParam^{-1}
            \simpleDeriv{\prime}{\left(
                \frac{3}{2}
                \left(\simpleDeriv{\prime}{\cortexPhaseFieldIFCIdx{0}{}{}}\right)^2
                \left(
                    \pfSpeciesOneIFCIdx{1}{}{}
                    (\pfVelocityIFCIdx{0}{}{})_\tau                
                    +
                    \pfSpeciesOneIFCIdx{0}{}{}
                    (\pfVelocityIFCIdx{1}{}{})_\tau
                \right)
            \right)}
            \concat\interfCoords[\phaseFieldParam]{}{}
            \cdot\extNormal{}
            \\
            &+
            \phaseFieldParam^{-1}
            \diver{\interfTwo{}_{\sdist{}{}}}{}{
                \frac{3}{2}
                \left(\simpleDeriv{\prime}{\cortexPhaseFieldIFCIdx{0}{}{}}\right)^2
                \concat\interfCoords[\phaseFieldParam]{}{}
                \pfSpeciesOneIdx{0}{}{}
                (\pfVelocityIdx{0}{}{})_\tau
            }
            +\landauBigo{1}.
            \end{split}
        \end{equation}
        The first and second summands vanish since 
        $ (\pfVelocityIFCIdx{i}{}{})_\tau \cdot \hat{\nu} =
        \transposed{\pfVelocityIFCIdx{i}{}{}}\orthProjMat{\normal[\cortexPhaseFieldIFC{}{}]{}}{}
        \hat{\nu} = \landauBigo{\phaseFieldParam^2} $, $ i \in\{1,2\} $, thanks to 
        \eqref{equ:formal asymptotics:expansion of the L2 gradient of the coupling energy:%
        expansion of the normal}.
        
        Second, we compute (making again use of \eqref{equ:formal asymptotics:inner expansion:%
        species to leading order:constant in normal direction})
        \begin{align*}
            \diver{}{}{
              	\ginzburgLandauEnergyDens{\cortexPhaseFieldSym}{}
              	\speciesOneDiffusiv{}\grad{}{}{}\pfSpeciesOne{}{}
            }
            &=
            \phaseFieldParam^{-2}
            \frac{3}{2}
            \simpleDeriv{\prime}{
                \left(
              	\left(
                  	\simpleDeriv{\prime}{
                      	\cortexPhaseFieldIFCIdx{0}{}{}
                    }
                \right)^2
                \right)
            }
            \concat\interfCoords[\phaseFieldParam]{}{}
            \speciesOneDiffusiv{}
            \left(
              	\grad{\interfTwo{}_{\sdist{}{}}}{}{}\pfSpeciesOneIdx{0}{}{}
                +
                \simpleDeriv{\prime}{\pfSpeciesOneIFCIdx{1}{}{}}
                \concat\interfCoords[\phaseFieldParam]{}{}
                \extNormal{}
            \right)
            \cdot\extNormal{}
            +\phaseFieldParam^{-2}
            \frac{3}{2}
            \left(
                \simpleDeriv{\prime}{
                    \cortexPhaseFieldIFCIdx{0}{}{}
                }
            \right)^2
            \concat\interfCoords[\phaseFieldParam]{}{}
            \speciesOneDiffusiv{}
            \simpleDeriv{\prime\prime}{\pfSpeciesOneIFCIdx{1}{}{}}
            \concat\interfCoords[\phaseFieldParam]{}{}
            \\
            &+
            \landauBigo{\phaseFieldParam^{-1}}
            \\
            &=
            \phaseFieldParam^{-2}
            \frac{3}{2}
            \left(
            \simpleDeriv{\prime}{
                \left(
              	\left(
                  	\simpleDeriv{\prime}{
                      	\cortexPhaseFieldIFCIdx{0}{}{}
                    }
                \right)^2
                \right)
            }
            \concat\interfCoords[\phaseFieldParam]{}{}
            \simpleDeriv{\prime}{\pfSpeciesOneIFCIdx{1}{}{}}
            \concat\interfCoords[\phaseFieldParam]{}{}
            +
            \left(
                \simpleDeriv{\prime}{
                    \cortexPhaseFieldIFCIdx{0}{}{}
                }
            \right)^2
            \concat\interfCoords[\phaseFieldParam]{}{}
            \speciesOneDiffusiv{}
            \simpleDeriv{\prime\prime}{\pfSpeciesOneIFCIdx{1}{}{}}
            \concat\interfCoords[\phaseFieldParam]{}{}            
            \right)
            \\
            &=
            \phaseFieldParam^{-2}
            \frac{3}{2}
            \speciesOneDiffusiv{}
            \simpleDeriv{\prime}{
            \left(
                \left(
                    \simpleDeriv{\prime}{
                        \cortexPhaseFieldIFCIdx{0}{}{}
                    }
                \right)^2
                \simpleDeriv{\prime}{\pfSpeciesOneIFCIdx{1}{}{}}
            \right)}
            \concat\interfCoords[\phaseFieldParam]{}{}
            +
            \landauBigo{\phaseFieldParam^{-1}}.
        \end{align*}
        All other terms in \eqref{equ:modelling:phase field:resulting PDEs:reaction-diffusion of species one}
        are in $ \landauBigo{\phaseFieldParam^{-1}} $. Thus,
        $
            \left(
                \simpleDeriv{\prime}{
                    \cortexPhaseFieldIFCIdx{0}{}{}
                }
            \right)^2
            \simpleDeriv{\prime}{\pfSpeciesOneIFCIdx{1}{}{}}
        $
        is constant in $ \fastVar $. We observe that 
        $
            \left(
                \simpleDeriv{\prime}{
                    \cortexPhaseFieldIFCIdx{0}{}{}
                }
            \right)^2
        $ decays for large $ \abs{\fastVar} $, and for the expression to remain constant,
        $
            \simpleDeriv{\prime}{\pfSpeciesOneIFCIdx{1}{}{}}
        $ must either blow up or be zero constantly. We can exclude the former case by 
        matching, so 
        $$
        	\simpleDeriv{\prime}{\pfSpeciesOneIFCIdx{1}{}{}} = 0.
        $$
        Then, \eqref{equ:formal asymptotics:sharp interface limit:species subsystem:%
            expansion of flux} simplifies 
        and we obtain 
        \begin{equation*}
            \begin{split}
            \diver{}{}{
              	\ginzburgLandauEnergyDens{\cortexPhaseFieldSym}{}
              	\speciesOneDiffusiv{}\grad{}{}{}\pfSpeciesOne{}{}
            }
            &=
            \phaseFieldParam^{-1}
            \frac{3}{2}
            \bigg(
                \simpleDeriv{\prime}{
                    \left(
                    \left(
                        \simpleDeriv{\prime}{
                            \cortexPhaseFieldIFCIdx{0}{}{}
                        }
                    \right)^2
                    \right)
                }
                \concat\interfCoords[\phaseFieldParam]{}{}
                \speciesOneDiffusiv{}
                \grad{\interfTwo{}_{\sdist{}{}}}{}{}
                \pfSpeciesOneIdx{1}{}{}
                \cdot\extNormal{}
                \\
                &+
                \left(
                    \simpleDeriv{\prime}{
                        \cortexPhaseFieldIFCIdx{0}{}{}
                    }
                \right)^2
                \concat\interfCoords[\phaseFieldParam]{}{}                
                \left(
                \diver{\interfTwo{}_{\sdist{}{}}}{}{
	                \speciesOneDiffusiv{}        
                    \grad{\interfTwo{}_{\sdist{}{}}}{}{}\pfSpeciesOneIdx{0}{}{}
                }
                +
                \speciesOneDiffusiv{}
                \simpleDeriv{\prime}{\left(
                    \widehat{
                      	\grad{\interfTwo{}_{\sdist{}{}}}{}{}
                        \pfSpeciesOneIdx{1}{}{}
                    }
                \right)}
                \concat\interfCoords[\phaseFieldParam]{}{}
                \cdot\extNormal{}
                \right)
            \bigg)
            +
            \landauBigo{1}
            \\
            &=
            \phaseFieldParam^{-1}
            \frac{3}{2}
            \left(
                \simpleDeriv{\prime}{
                    \cortexPhaseFieldIFCIdx{0}{}{}
                }
            \right)^2
            \concat\interfCoords[\phaseFieldParam]{}{}                
            \diver{\interfTwo{}_{\sdist{}{}}}{}{
                \speciesOneDiffusiv{}        
                \grad{\interfTwo{}_{\sdist{}{}}}{}{}\pfSpeciesOneIdx{0}{}{}
            }.
            \end{split}
        \end{equation*}

        Finally, we find in 
        \eqref{equ:modelling:phase field:resulting PDEs:reaction-diffusion of species one} (note
        \eqref{equ:formal asymptotics:sharp interface limit:species subsystem:%
            expansion of add flux}, \eqref{equ:formal asymptotics:%
            expansion of the L2 gradient of the coupling energy:%
                helping expansions:phase field mean curv}, and the
        independence of $ \cortexPhaseFieldIFCIdx{0}{}{} $ from the tangential variable
        due to its being the optimal profile)
        to leading order $ \phaseFieldParam^{-1} $
        \begin{align*}
            \frac{3}{2}
            \left(\simpleDeriv{\prime}{\schemeLocalFun{\cortexPhaseFieldSym}_0}\right)^2
            \concat\interfCoords[\phaseFieldParam]{}{}
            \left(
                \pDiff{t}{}{}\pfSpeciesOneIdx{0}{}{}
                -
                \diver{\interfTwo{}_{\sdist{}{}}}{}{
                  	\speciesOneDiffusiv{}
                	\grad{\interfTwo{}_{\sdist{}{}}}{}{}\pfSpeciesOneIdx{0}{}{}
                }
            \right)
            -
            (\simpleDeriv{\prime}{\schemeLocalFun{\cortexPhaseFieldSym}_0})^2
            \concat\interfCoords[\phaseFieldParam]{}{}
            \pfSpeciesOneIdx{0}{}{}
            \extMeanCurv{}
            \pfVelocityIdx{0}{}{}\cdot\extNormal{}
            +
            \frac{3}{2}
            \left(\simpleDeriv{\prime}{\cortexPhaseFieldIFCIdx{0}{}{}}\right)^2
            \concat\interfCoords[\phaseFieldParam]{}{}
            \diver{\interfTwo{}_{\sdist{}{}}}{}{
                \pfSpeciesOneIdx{0}{}{}
                (\pfVelocityIdx{0}{}{})_\tau
            }
            =
            \\
            \frac{3}{2}\left(\simpleDeriv{\prime}{\cortexPhaseFieldIFCIdx{0}{}{}}\right)^2
            \concat\interfCoords[\phaseFieldParam]{}{}
            \reactionOne{\pfSpeciesOneIdx{0}{}{},\pfSpeciesTwoIdx{0}{}{}}{
              	\membrPhaseFieldIdx{0}{}{},
                \normal[\cortexPhaseFieldIdx{0}{}{}]{}
            },
        \end{align*}
        which is 
        \eqref{equ:modelling:sharp interface classical PDE:species one evoluation}
        up to constants.
        On the right hand side, we used the expansion
        $$
            \reactionOne{\pfSpeciesOne{}{},\pfSpeciesTwo{}{}}{
              	\membrPhaseField{}{},
                \normal[\cortexPhaseField{}{}]{}
            }
            =
            \reactionOne{\pfSpeciesOneIdx{0}{}{},\pfSpeciesTwoIdx{0}{}{}}{
              	\membrPhaseFieldIdx{0}{}{},
                \normal[\cortexPhaseFieldIdx{0}{}{}]{}
            }
            +\landauBigo{\phaseFieldParam},
        $$
        which holds with \eqref{equ:formal asymptotics:inner expansion:consequ bound on coupl energy:%
        variation of normal}
        and
        \begin{align*}
            \grad{\pfSpeciesOne{}{}}{L^2}{}\reactionOne{}{},
            \grad{\pfSpeciesTwo{}{}}{L^2}{}\reactionOne{}{}, 
            \grad{\membrPhaseFieldSym}{L^2}{}\reactionOne{}{},
            \grad{\normal{}}{L^2}{}\reactionOne{}{}
            \in \landauBigo{1}.
        \end{align*}
        
    So altogether, we could argue with the help of formally matched asymptotic expansions   that solutions of the PDE system
    \eqref{equ:modelling:phase field:resulting PDEs}, under suitable 
    Assumptions~\ref{it:formal asymptotics:ass:existence of solutions:ex}--%
    \ref{it:formal asymptotics:inner expansion struct:cond for phase fields}, converge to solutions
    of \eqref{equ:modelling:sharp interface system} as $ \phaseFieldParam \searrow 0 $.
        
\section{Conclusion}
We have made plausible that both the diffuse and sharp interface modelling approaches are compatible in the
sense that their solutions are approximations of each other. To arrive at this conclusion, we leveraged
the method of formal asymptotic analysis.

From a mathematical perspective, it is desirable to prove this result rigorously like \cite{AL2018} or
\cite{FL2021} did for related PDE systems. 
The main problems to deal with will likely be analysing the leading order terms in the expansion of
the Canham-Helfrich energy and controlling the pressure, showing it does not blow up near the diffuse layers.
An excellent stock of techniques for analysing the Canham-Helfrich energy is already provided in 
\cite{FL2021}. However, they analyse the pure Willmore flow problem, and so there is no coupling with a 
fluid, nor with a species subsystem like in the PDE system \eqref{equ:modelling:phase field:resulting PDEs} 
investigated here, which poses additional problems like possible pressure blow-ups. 
Controlling the pressure for a phase-field-Navier-Stockes coupling is investigated
in \cite{AL2018}. A step towards a rigorous analysis of \eqref{equ:modelling:phase field:resulting PDEs}
might therefore be possible by uniting the results of both works and leaving the species subsystem aside.

From a modelling perspective, our results increase the confidence that qualitatively both abstractions---%
the sharp interface or diffuse layer abstraction---are equivalent and the focus can now be more on 
other aspects like numerical feasibility. 

\subsection{Acknowledgements}
The authors gratefully acknowledge the support by the Graduiertenkolleg 2339
IntComSin of the Deutsche Forschungsgemeinschaft (DFG, German Research Foundation)
-- Project-ID 321821685.
We thank \ack Helmut Abels for insightful discussions.

\providecommand{\bysame}{\leavevmode\hbox to3em{\hrulefill}\thinspace}
\providecommand{\MR}{\relax\ifhmode\unskip\space\fi MR }
\providecommand{\MRhref}[2]{%
  \href{http://www.ams.org/mathscinet-getitem?mr=#1}{#2}
}
\providecommand{\href}[2]{#2}


\begin{thebibliography}{99}
\bibitem{AL2018}
H.~Abels and Y.~Liu, \emph{Sharp interface limit for a {Stokes/Allen-Cahn}
  system}, Archive for Rational Mechanics and Analysis \textbf{229} (2018),
  417--502.

\bibitem{AC2016}
R.~Alert and J.~Casademunt, \emph{Bleb nucleation through membrane peeling},
  Physical Review Letters \textbf{116} (2016), Article~068101.

\bibitem{ACBS2015}
R.~Alert, J.~Casademunt, J.~Brugués, and P.~Sens, \emph{Model for probing
  membrane-cortex adhesion by micropipette aspiration and fluctuation
  spectroscopy}, Biophysical Journal \textbf{108} (2015), 1878--1886.

\bibitem{CF1986}
G.~Caginalp and P.~Fife, \emph{Phase-field methods for interfacial boundaries},
  Physical Review B \textbf{33} (1986), Article~7792.

\bibitem{C1970}
P.~B. Canham, \emph{The minimum energy of bending as a possible explanation of
  the bioconcave shape of the human red blood cell}, Journal of Theoretical
  Biology \textbf{26} (1970), 61--81.

\bibitem{CP2008}
G.~T. Charras and E.~Paluch, \emph{Blebs lead the way: how to migrate without
  lamellipodia}, Nature Reviews Molecular Cell Biology \textbf{9} (2008),
  730--736.

\bibitem{QLRW2009}
Q.~Du, C.~Liu, R.~Ryham, and X.~Wang, \emph{Energetic variational approaches in
  modeling vesicle and fluid interactions}, Physica D \textbf{238} (2009),
  923--930.

\bibitem{E1979}
W.~Eckhaus, \emph{Asymptotic analysis of singular perturbations}, Studies in
  Mathematics and its Applications, vol.~9, North-Holland Publishing Company,
  1979.

\bibitem{FL2021}
M.~Fei and Y.~Liu, \emph{Phase-field approximation of the {Willmore Flow}},
  Archive for Rational Mechanics and Analysis \textbf{241} (2021), 1655--1706.

\bibitem{GT2001}
D.~Gilbarg and N.~S. Trudinger, \emph{Elliptic partial differential equations
  of second order}, 3 ed., Grundlehren der Mathematischen Wissenschaften, vol.
  224, Springer, 2010.

\bibitem{H1973}
W.~Helfrich, \emph{Elastic properties of lipid bilayers: Theory and possible
  experiments}, Zeitung f\"ur Naturforschung \textbf{28} (1973), 693--703.

\bibitem{Ho1995}
M.~J. Holmes, \emph{Introduction to perturbation methods}, Texts in Applied
  Mathematics, vol.~20, Springer, 1995.

\bibitem{LGR2019}
I.~Lavi, M.~Goudarzi, E.~Raz, N.~S. Gov, R.~Voituriez, and P.~Sens,
  \emph{Cellular blebs and membrane invaginations are coupled through membrane
  tension buffering}, Biophysical Journal \textbf{117} (2019), 1485--1495.

\bibitem{LCM2012}
F.~Y. Lim, K.-H. Chiam, and L.~Mahadevan, \emph{The size, shape, and dynamics
  of cellular blebs}, Europhysics Letters \textbf{100} (2012), Article~28004.

\bibitem{O1931}
L.~Onsager, \emph{Reciprocal relations in irreversible processes i}, Physical
  Review \textbf{37} (1931), 405--426.

\bibitem{Othm2018}
H.~G. Othmer, \emph{Eukaryotic cell dynamics from crawlers to swimmers}, WIREs
  Computational Molecular Science \textbf{9} (2018), e1376.

\bibitem{PR2013}
E.~K. Paluch and E.~Raz, \emph{The role and regulation of blebs in cell
  migration}, Current Opinion in Cell Biology \textbf{25} (2013), 582--590.

\bibitem{QWS2006}
T.~Qian, X.-P. Wang, and P.~Shen, \emph{A variational approach to moving
  contact line hydrodynamics}, Journal of Fluid Mechanics \textbf{564} (2006),
  333--360.

\bibitem{SDN2020}
B.~Stinner, A.~Dedner, and A.~Nixon, \emph{A finite element method for a fourth
  order surface equation with application to the onset of cell blebbing},
  Frontiers in Applied Mathematics and Statistics \textbf{6} (2020),
  Article~21.

\bibitem{S2021}
W.~Strychalski, \emph{3d computational modeling of bleb initiation dynamics},
  Front. Phys. Sec. Biophysics \textbf{9} (2021), Article~775465.

\bibitem{SCLG2015}
W.~Strychalski, C.~A. Copos, O.~L. Lewis, and R.~D. Guy, \emph{A poroelastic
  immersed boundary method with applications to cell biology}, Journal of
  Computational Physics \textbf{282} (2015), 77--97.

\bibitem{SG2013}
W.~Strychalski and R.~D. Guy, \emph{A computational model of bleb formation},
  Mathematical Medicine and Biology \textbf{30} (2013), 115--130.

\bibitem{Wa2008}
X.~Wang, \emph{Asymptotics analysis of phase field formulations of bending
  elasticity models}, SIAM Journal on Mathematical Analysis \textbf{39} (2008),
  1367--1401.

\bibitem{W2021}
P.~Werner, \emph{Sharp and diffuse interface models for the evolution of
  surfaces that are immersed in fluids and coupled through surfactants}, Ph.D.
  thesis, Friedrich-Alexander-Universit\"at Erlangen-N\"urnberg, 2021.

\bibitem{WBFG2021}
P.~Werner, M.~Burger, F.~Frank, and H.~Garcke, \emph{A diffuse interface model
  for cell blebbing including membrane-cortex coupling with linker dynamics},
  SIAM Journal on Applied Mathematics \textbf{82} (2022), 1091--1112.

\bibitem{WBP2020}
P.~Werner, M.~Burger, and J.~Pietschmann, \emph{A {PDE} model for bleb
  formation and interaction with linker proteins}, Transactions of Mathematics
  and its Applications \textbf{4} (2020), Article~1.

\end{thebibliography}

\end{document}